\documentclass[a4paper,12pt]{amsart}
\usepackage[utf8]{inputenc}
\usepackage{amsthm}
\usepackage{amsmath}
\usepackage{bm}
\usepackage{enumitem}
\usepackage{amsfonts}
\usepackage{amssymb}
\usepackage{mathtools}
\usepackage{appendix}
\usepackage{mathrsfs}
\usepackage{setspace}
\usepackage{comment}
\usepackage{xcolor}
\usepackage{todonotes}
\usepackage{thmtools} 
\usepackage[foot]{amsaddr}
\RequirePackage[a4paper,top=2.54cm,bottom=2.54cm,left=1.90cm,right=1.90cm,%
                headsep=1em,includehead,includefoot]{geometry}
\usepackage[pdfdisplaydoctitle,colorlinks,breaklinks,urlcolor=blue,linkcolor=blue,citecolor=blue]{hyperref} 
\usepackage[nameinlink,capitalise]{cleveref}

\newcommand{\E}{\ensuremath{\mathbb{E}}}
\providecommand{\U}[1]{\protect\rule{.1in}{.1in}}
\newtheorem{theorem}{Theorem}

\newtheorem{definition}[theorem]{Definition}

\newtheorem{lemma}[theorem]{Lemma}

\newtheorem{proposition}[theorem]{Proposition}
\newtheorem{remark}[theorem]{Remark}

\newcommand{\T}{\mathbb{T}}
\newcommand{\R}{\mathbb{R}}
\newenvironment{acknowledgements}{%
  \begin{abstract}
}{%
  \end{abstract}
}

\newcommand{\Div}{\text{div}}
\title[Stretching   of   Polymers and turbulence: FP equation and scaling limits]
{Stretching   of   Polymers and turbulence: Fokker Planck equation, special stochastic scaling limit and stationary law}
\author[F. Flandoli]{Franco Flandoli}
\address{Scuola Normale Superiore, Piazza dei Cavalieri, 7, 56126 Pisa, Italia}
\email{\href{mailto:franco.flandoli at sns.it}{franco.flandoli at sns.it}}
\author[Y. Tahraoui]{Yassine Tahraoui}
\address{Scuola Normale Superiore, Piazza dei Cavalieri, 7, 56126 Pisa, Italia}
\email{\href{mailto:yassine.tahraoui at sns.it}{yassine.tahraoui at sns.it}}
\date\today

\begin{document}
	\begin{abstract}
	    The aim of this work is understanding the stretching mechanism of stochastic models of turbulence acting on a simple model of dilute polymers. We consider a turbulent model that is white noise in time and activates frequencies in a shell $N\leq |k|\leq 2N$ and investigate the scaling limit as $N\rightarrow \infty$, under suitable intensity assumption, such that the stretching term has a finite limit covariance. The polymer density equation, initially an SPDE, converges weakly to a limit deterministic equation with a new term. Stationary solutions can be computed and show power law decay in the polymer length.
	\end{abstract}
	\maketitle
\tableofcontents
 	\section{Introduction}
Polymers are complex molecular systems, that can be vaguely thought as a chain
of springs. In a turbulent fluid they are usually found in two states called
\textit{coil} and \textit{stretched}. The coil state is like a spherical or
ellipsoidal rolled chain, which may be more or less elongated, but still in
roll position. The stretched state is when the chain is elongated, more
similar to a stright line than a sphere. Turbulence with small stretching
intensity produces only a small perturbation of the coil state, while strongly
stretching turbulence may lead to the stretched state; when the polymer pass
from one state to the other we speak of coil-stretch transition. \\

In this work, we consider in dimension $2$ or $3$ the following \textit{Hookean} model:
\begin{align}\label{Intro1}
	\begin{cases}
	dR_t&=\nabla u(X_t,t) R_tdt-\dfrac{1}{\beta}R_tdt+\sqrt{2}\sigma d\mathcal{W}_t, \\
	dX_t&=u(X_t,t)dt,
	\end{cases}
	\end{align}
where $X_t$ is the polymer position (the    center  of  mass) and $R_t$ is  the end-to-end vector, representing the orientation and elongation of the chain,  see e.g.  \cite[Section  4.2]{BriTL09}. The polymer is embedded into a fluid having velocity $u(t,x)$, which stretches $R_t$ by $\nabla u(x,t)$. The equation for $R_t$ contains also a damping (restoring) term with relaxation time $\beta$ and Brownian fluctuations 
$\sqrt{2}\sigma d\mathcal{W}_t$ where, to simplify the notations we have denoted by $\sigma^2$ the product $\frac{kT}{\beta}$, $k$ being Boltzmann constant and  $T$ being the temperature. \\

The statistics of the polymer length $r$, say the diameter, have been investigated by several authors in the physical literature, see for instance  \cite{Balk,FalGawVer,Gerashen,PicardoLanceVinc} and other references mentioned below. In the coil state,  the distribution of $r$ is found to be power law
\begin{equation}
f\left(  r\right)  \sim r^{-1-\alpha}\qquad\text{for relatively large
}r\label{power law}%
\end{equation}
with the exponent $\alpha$ positive (so that $f$ is normalizable). The
exponent $\alpha$ depends on the stretching properties of the turbulent flow:
the highest is the stretching intensity, the lowest is $\alpha$. At $\alpha=0$
one has the coil-stretch transition. In the stretched state, the precise mathematics depends on the idealizations of
the model.  If we had introduced a superlinear damping instead of the linear
damping $-\frac{1}{\beta}R_t$, this would produce a sort of cut-off at very high lengths ( e.g. \textit{FENE} model, see \autoref{rmq-FENE}), so that the behavior (\ref{power law}) would be true only
in a range
\[
r_{0}<<r<<r_{1}%
\]
and globally the function $f$ would still be a pdf. In our idealization of
linear damping, the stretching may overcome the damping and lead to infinite
length in the asymptotic regime, which is the idealized signature of strech
state, see \autoref{Section-stationry-power-law}.\\

Clearly, one would like to predict the exponent $\alpha$ based on turbulence
features. The theory developed on physical grounds by \cite{Balk}
and  \cite{Chertkov} tells us that $\alpha$ is related to the Lyapunov
exponents of the turbulent flow. Precisely, if $\phi_{t}\left(  x\right)  $ is
the Lagrangian flow associated to the turbulent fluid, and we define%
\[
\mathcal{L}\left(  q\right)  =\lim_{t\rightarrow\infty}\frac{1}{t}%
\log\mathbb{E}\left[  \left\vert D\phi_{t}\left(  x\right)  \right\vert
^{q}\right]
\]
then $\alpha$ satisfies 
$
\frac{\alpha}{2\beta}=\mathcal{L}\left(  \alpha\right).
$
However, the computation of $\mathcal{L}\left(  \alpha\right)  $ is not
trivial in general. \\

The theory just mentioned does not make use of scale separation (and not even delta correlation in time of the turbulent fluid), hence it is quite general (much more than ours). In our paper we consider a special class of turbulent flows, delta-correlated in time, having a precise dyadic scale $\ell$, namely
based on Fourier components of modes $\left\vert k\right\vert \in\left[
\ell^{-1},2\ell^{-1}\right]  $. When $\ell$ is very small, precisely in the
scaling limit as $\ell\rightarrow0$, we discover a power law of the form
(\ref{power law}). This provides another proof (or a more rigorous proof,
although in a particular regime) of the emergence of power law in polymer
length $r$. In addition, we find a very direct relation between $\alpha$ and
the turbulent flow characteristics, thanks essentially to the special
structure of the flow and the separation of scales intrinsic in the scaling
limit; the non-trivial informations on the Lyapunov exponents are not needed anymore.\\

For the convenience of the reader, let us explain heuristically the main result of this work then precise results  (see \autoref{section-main-result} ) and  rigorous proofs   are provided in the main body of the paper.

\subsection{Sketch of the  main results}
We consider a dilute (non-interacting) family of polymers subject to equations
\eqref{Intro1}, thus described by the kinetic equation for the density $f^{N}\left(
x,r,t\right)  $ of polymers with position $x$ and length $r$ at time $t$%
\begin{align}\label{Intro2}
	\begin{cases}\partial_t f^N(x,r,t)&+\Div_x(u^N(x,t)f^N(x,r,t))+\Div_r((\nabla u^N(x,t)r-\dfrac{1}{\beta}r)f^N(x,r,t))=	\sigma^2\Delta_r f^N(x,r,t)\\	
		f^N|_{t=0}&=f_0.
			\end{cases}\end{align}

Here $u^{N}\left(x,t\right)  $ is the fluid velocity. We assume that
$u^{N}\left(  x,t\right)  $ is made of two components, a deterministic
large-scale one $u_{L}\left(  x,t\right)  $ and a stochastic one, modeling
small-scale turbulence, of the form $\sum_{k\in K}\sigma_{k}^{N}\left(
x\right)  \partial_{t}W_{t}^{k}$, acting in Stratonovich form, that will be
described below in detail; hence%
\[
u^{N}\left(  x,t\right)  =u_{L}\left(  x,t\right)  +\circ\sum_{k\in K}%
\sigma_{k}^{N}\left(  x\right)  \partial_{t}W_{t}^{k}.
\]
Here we just stress the fact that the coefficients $\sigma_{k}^{N}\left(
x\right)  $ of the turbulent part depend on a parameter $N$ so that, when $N$
increases, they represent smaller and smaller space scales, precisely Fourier
frequencies $N\leq\left\vert k\right\vert \leq2N$, providing a separation of
scale regime essential for our analysis. The noise acts on $f^{N}\left(
x,r,t\right)  $ in transport form, but incorporating also the stretching
action by the term $\nabla u^{N}\left(  x,t\right)  r$.\\

 Our aim in this work is twofold. 
   Firstly,  present a rigorous mathematical framework and proofs to study a stochastic Fokker-Planck equations with transport noise and general class of initial data. These equations  has an hyperbolic nature  with    respect to  the space variable  $x\in\T^2$  and  polymer length vector variable $r\in\mathbb{R}^2$. Thus, we should  combine techniques for weighted spaces to handle the $r$-variable, spatial homogeneity and mirror symmetry of the covariance operator to handle the stochastic part  and   prove existence of ”quasi-regular weak solution”, see \autoref{Def-sol}. Then, combine commutators techniques to handle the space variable, using a density result based on Wiener chaos decomposition to prove uniqueness in this class of solution, see \autoref{TH2}. We refer to  \autoref{section-outlines}  for discussion about the mathematical challenges related to these equations. \\
   Secondly, understand the stretching power on the polymer, in the limit as $N\rightarrow \infty$  when the noise  activates frequencies in a shell $N\leq |k|\leq 2N$.  The final result  is a limit model, of Fokker-Planck type, with a new diffusion term in the radial variable of the polymer, with non-homogeneous and degenerate coefficients. Its radially symmetric stationary solutions are explicit and have power-law tails.  Let us describe quickly the scaling limit result, see \autoref{main-thm-2} for the  precise result.
   Under the assumptions described in \autoref{Section1-desciption}, we prove that $f^N(x,r,t)$ weakly converges to the solution of a deterministic equation of the form (the results below contain also a deterministic term in the velocity, which is omitted here in the Introduction for notational simplicity)
	\begin{align}\label{Intro3}
	\begin{cases}\partial_t \overline{f}(x,r,t)&-
 \Div_r(\dfrac{1}{\beta}r \overline{f}(x,r,t))=	\sigma^2\Delta_r \overline{f}(x,r,t) + 
 \dfrac{1}{2}\Div_r(A(r) \nabla \overline{f}(x,r,t))        \\	
		\overline{f}|_{t=0}&=f_0
			\end{cases}\end{align}
   where $A(r)=k_T (3\left\vert r\right\vert^2I-2r\otimes r)$ and
$k_T=\dfrac{\pi \log(2)}{8}a^2$ where $a$ is an intensity parameter of the noise. The new diffusion term in the $r$-variable is one of the most important novelty of this work. It captures the statistical properties of the stretching mechanism. We compute the explicit rotation-invariant solution of the associated stationary equation and find it has a power law decay for large 
   $|r|$:
   \[
\overline{f}\left(  \vert r\vert \right)  \sim\left\vert r\right\vert ^{-\frac{2}{k_{T}\beta}}%
\]
   indicating large values with high probability. The constants in the power are directly associated to those of the stochastic model of $u^N.$ This fact was predicted in the physical literature  based on other models and assumptions, see e.g. \cite{Balk}. In our model we identify a simple link between the power of the tail and parameters of the turbulence model, see \autoref{Section-stationry-power-law}.\\
   
    We wish to draw the reader's  attention   to  the following:  introduce the   mean    of  the structure   tensor  $\mathbf{T}(t,x):=\int_{\mathbb{R}^2}r\otimes rf^N(t,x,r)dr$    then $\mathbf{T}^N$\footnote{The    subscript  $N$ to  stress  the dependence  on  $N$ because of  the presence    of  $u^N$.}   satisfies   the following  closed  system  of  PDEs  
\begin{align}\label{macro-intro}
\begin{cases}
&\partial_{t}\mathbf{T}^N+u^N\cdot\nabla \mathbf{T}^N   =(\nabla u^N) \mathbf{T}^N+\mathbf{T}^N(\nabla u^N)^t -\dfrac{2}{\beta}(\mathbf{T}^N-kT\mathbf{I}) \\
&\quad \mathbf{T}^N|_{t=0} =\mathbf{T}_0,
\end{cases}
\end{align}
The last    equation   \eqref{macro-intro}  is    a    macroscopic  Oldroyd-B   model (see e.g. \cite[Section 2.8]{cioranescu2016})   and the tensor  $\mathbf{T}$    characterize the viscoelastic (non-Newtonian   part)   of  the flow   \eqref{Intro1}. Many    rheological behavior    can be  detected    such    as  shear   viscosity,  normal  stress  difference  and overshoot   phenomenon  in  contrast    with    Newtonian   flows,  see e.g.    \cite{BriTL09}  and we  refer   e.g.    to \cite{cioranescu2016,tahraoui2023,tahraoui2023local}     for other   types   of  non-Newtonian   flows.   As  we    discuss in  \autoref{Section-stationry-power-law},   the explicit rotation-invariant solution of the  stationary equation    associated  with  the limit   equation     has a power-law density  and  therefore  it       is  not sufficient  to  study   the system    \eqref{macro-intro}  only,   this    is  another reason  we  base    our analysis    on  FP  equation  \eqref{Intro2}. We  will    comment about  the limit   equation    associated with    \eqref{macro-intro}  in  \autoref{Rmq-SDE-FP}. Finally, Although the  results and proofs are presented in 2D, similar results  hold  in 3D after some cosmetic changes and mostly  the form of the stochastic turbulent velocity, see \autoref{remark-3D-case}.\\

   This work is a key step in a research project on the effects of small-scale turbulence on large-scale motion, started a few years ago after the development of a new technique. In order to understand the novelties it may be useful to review these recent developments and stress similarities and
differences.
\subsection{Literature review}
 We divide the presentation in two subsections, the scalar and the
vector case. The second one is characterized by a peculiar issue, the presence
of stretching, which is by no means an incremental detail over the scalar
case; and the present work aims to give relevant information on the stretching case.

\subsubsection{The scalar case}

The foundational work \cite{Galeati} investigated a problem that, to some extent,
could be considered as a variant of diffusion approximation results in
stochastic homogenization theory: a stochastic first order differential
operator in Stratonovich form gives rise, in a suitable scaling limit, to a
deterministic diffusion term. The stochastic equation had the form (we do not
describe in detail the spatial domain, the finite or countable set where the
index $k$ varies and other details which have no relevance for the purpose of
this section, see \cite{Galeati})%
\begin{align*}\begin{cases}
d\rho_{t}^{N} &  =\sum_{k}\theta_{k}^{N}e_{k}\cdot\nabla\rho_{t}^{N}\circ
dW_{t}^{k}\\
\rho^{N}|_{t=0} &  =\rho_{0}\in L^{2}%
\end{cases}
\end{align*}
with suitable vector fields $e_{k}$ (a suitable subset of a complete
orthonormal system in $L^{2}$) and the family of real valued coefficients
$\theta^{N}:=\left(  \theta_{k}^{N}\right)  $. Under the assumptions
$$\lim_{N\rightarrow\infty}\left\Vert \theta^{N}\right\Vert _{\ell^{2}}%
^{2}=\kappa, \quad \lim_{N\rightarrow\infty}\left\Vert \theta^{N}\right\Vert
_{\ell^{\infty}}=0,$$ it has been proved that, in a weak sense (namely against
test functions or in suitable negative order Sobolev spaces) the solution
$\rho^{N}$ converges in mean square to the unique deterministic solution
$\rho_{t}$ of the heat equation%
\begin{align*}
\begin{cases}
\partial_{t}\rho_{t} &  =\kappa\Delta\rho_{t}\\
\rho|_{t=0} &  =\rho_{0}%
\end{cases}
\end{align*}
(possibly with the constant $\kappa$ modified by a constant related to space
dimension and the choice of the orthonormal system). The technical assumption
mentioned above on $\theta^{N}$ can be reformulated in terms of the noise
covariance function:%
\[
Q_{N}\left(  x,y\right)  =\sum_{k}\left(  \theta_{k}^{N}\right)  ^{2}%
e_{k}\left(  x\right)  \otimes e_{k}\left(  y\right)  =\mathbb{E}\left[
W\left(  x,1\right)  \otimes W\left(  y,1\right)  \right]
\]
where $W\left(  x,t\right)  =\sum_{k}\theta_{k}^{N}e_{k}\left(  x\right)
W_{t}^{k}$ and the associated linear operator on $L^{2}$ vector fields
$v\left(  x\right)  $%
\[
\left(  \mathbb{Q}_{N}v\right)  \left(  x\right)  =\int Q_{N}\left(
x,y\right)  v\left(  y\right)  dy.
\]
The assumption is that the trace of $\mathbb{Q}_{N}$ converges to $\kappa$ and
the operator norm to zero. Heuristically, it means that the function
$Q_{N}\left(  x,y\right)  $ tends to a non zero value along the diagonal and
to zero outside, in other words
\begin{align}
\begin{cases}
\lim_{N\rightarrow0}Q_{N}\left(  x,x\right)   &  =\kappa
\label{classical scaling}\\
\lim_{N\rightarrow0}Q_{N}\left(  x,y\right)   &  =0\text{ for }x\neq
y
\end{cases}
\end{align}
namely the variance of the noise remains constant at the limit but the space
correlation goes to zero.\\

We have said that this is related to results in homogenization theory \cite{MK}, but it is important to notice that the model and strategy are
different from classical homogenization and especially they are very efficient for Generalizations to nonlinear problems. The result of \cite{Galeati} has been
extended, indeed, to 2D Euler equations, 2D Navier-Stokes equations, and other nonlinear models like Keller-Siegel and nonlinear heat equations, see
\cite{FGL1} and \cite{FGL2} for the basic results in this direction, obtained
both with compactness methods and more quantitatively with estimates on stochastic convolutions. The results have been supplemented by the analysis of Gaussian fluctuations and Large Deviations \cite{GL}. Mixing and dissipation enhancement results can also be obtained by this approach, see \cite{A-mixing,DongLuoTang,FL2,FGLtrans}, \cite[section 9]{FGL2} and  
\cite{Luo,LuoBin}. The fact that the limit equation
contains an additional strongly elliptic term sometimes has a feedback effect
of regularization on the approximating stochastic problems (the closeness in
some norm provides additional a priori bounds), leading to delay of blow-up results \cite{FGLreg}. See also \cite{A,FHuang,FlaElibook,FlaMorlacchiPapini,FlaRusso,HuangWei,LuoBousinn,LuoTang,LuoWang,Papini} for various other applications, examples and  numerical approximations. The problem is of interest also for Large Eddy Simulations and
Boussinesq assumption, see \cite{FLuoLuongo} for a Smagorinsky type result. The scaling limit has been also transposed to particle systems, see
for instance \cite{FlaLuoparticelle,GuoLuo}. This list of works is certainly non complete but it gives the feeling of the fertility of the scaling limit result proved in \cite{Galeati}.\\

Let us also mention that transport and transport-stretching noise (the second being the topic of next subsection) is currently used in many other researches
and for different purposes; above we have restricted the attention to works
dealing with the special scaling limit to an additional diffusion. The list of
references on other research directions on transport-stretching noise is too
long, let us only mention as examples \cite{Crisan1,Goodair24,Ephrati,H,Harouna,Memin}.

\subsubsection{The  vector fields case}

The works discussed above dealt with scalar problems, where transport noise
induces an additional elliptic operator. When the turbulent fluid acts on
vector fields, the action has two aspects: not only transport but also
\textit{stretching}: namely the differential $D\phi_{t}$ of the Lagrangian
flow $\phi_{t}$ of a fluid produces a modification of the length of vectors
$v$, $v\mapsto$ $D\phi_{t}v$. Since we consider divergence free fluids, the
deformation tensor $Du$ of the velocity field $u$ has zero trace, the symmetrization typically has a positive eigenvalues, indicating that an increase of length should be expected for many directions $v$.\\

Understanding deterministic and random stretching is much more difficult and
open, with respect to transport. Two indications of this difficulty are that
stretching is considered the most important problem in view of the open
problem of blow-up for the 3D Navier-Stokes equations \cite{Fef}; and that in
contrast to a wide range of results on mixing and dissipation enhancement
produced by chaotic transport of scalars, there is very few on the effect of
complex stretching \cite{CZ}.
The problem of random stretching has been approached also by the scaling limit
described above, and the research of the present paper is along this line. Let us see what previous papers did on this topic.\\

The results are more fragmentary than the scalar case. One of the first works on the action of vector fields has been done in \cite{FL1} on the 3D Navier-Stokes equations in vorticity form, for the vorticity $\omega=\nabla u$, $u$ being
the velocity of the fluid. However, precisely because the effect of the random
stretching term $\left(  \omega\cdot\nabla e_{k}\right)  $, more precisely
\[
\sum_{k}\theta_{k}^{N}\left(  \omega\cdot\nabla e_{k}\right)  \circ dW_{t}^{k}%
\]
was (and it  stills until now) not well understood, this term was neglected and
only the transport%
\[
\sum_{k}\theta_{k}^{N}\left(  e_{k}\cdot\nabla\omega\right)  \circ dW_{t}^{k}%
\]
was considered (with the need of a projection operator, see \cite{FL1} for
motivation and details). The result is similar to the scalar case
\cite{FGLreg}, namely a delay of blow-up, relevant because of the relevance of
the equation, but still incomplete because it does not explain what would
happen in the presence of random stretching. See also a generalisation to
rough path noise \cite{FHLN} and to magnetic hydrodynamics \cite{Luo}.\\

Concerning the 3D Navier-Stokes equations, a very important progress appeared
recently \cite{A2}. It is concerned with the Navier-Stokes equations in
velocity form and the noise, of transport type, is at that level
(energy-preserving instead of circulation-preserving as it should be the
transport-stretching noise at the level of vorticity equation, see the
foundational work \cite{H}). The stretching action of such noise is a bit
hidden but the mathematical progress is very important since it shows that a
reasonable noise (with some form of stretching) may still regularize the 3D
Navier-Stokes type equations (the result needs an hyperviscosity but of lower
degree than the one of the deterministic theory).\\

Then three works appeared with explicit inclusion of a random stretching term,
all of them basically assuming that the noise satisfies the classical scaling
(\ref{classical scaling}) of the scalar case. The first one, \cite{FlandoliLuo2024}
deals with a Navier-Stokes type system called 2D-3C, having 3D features but
also 2D simplifications, which allows one to control the strength of the
stretching term; the physics is motivated by a limit of fast rotating systems.
The second one is \cite{ButFlaLuo} on a passive magnetic field in a 3D thin domains,
where the smallness of one of the three dimensions is a key ingredient for the
control of stretching; the mathematics is quite intricate and thus the result
is developed for the simplified case of a passive vector field. The third one
is \cite{BL2024}, again on the passive magnetic field, but in a full three
dimensional domain. Stretching is full, in a sense, here but the key
ingredient is controlling the solution in a negative Sobolev space, physically
speaking at the level of the magnetic potential. As already said, but crucial
for the comparison with the present paper, all these works try to assume that
the noise scales, in $N$, as in (\ref{classical scaling}), which in particular
it means that we expect a Laplacian coming for the transport, in the limit
equation. Heuristically, when the Laplacian-due-to-transport persists,
stretching could be infinite (at least a priori). The ways to overcome this
blow-up of stretching, in the above papers, has been either to constrain it by a 2D-3C structure, or by a smallness parameter related to thin nomains, or controlling the solution in a negative Sobolev space.\\

Let us come to the different scaling in $N$ used in the present paper, which
also appears in \cite{BFLT2024} for very different purposes (a precise comparison
is made below). Here we have stretching of a passive vector quantity $R_{t}$,
that in Lagrangian form is  written as
\[
\sum_{k}\theta_{k}^{N}\nabla e_{k}\left(  X_{t}^{N}\right)  R_{t}^{N}\circ
dW_{t}^{k}%
\]
($X_{t}$ being the position of the polymer), and in Eulerian form as%
\[
\sum_{k}\theta_{k}^{N}\left(  \nabla e_{k}\left(  x\right)  r\right)
\cdot\nabla f^{N}\left(  x,t\right)  \circ dW_{t}^{k}%
\]
($\nabla e_{k}$ is a matrix acting on the vectors $R_{t}^{N}$ and $r$
respectively). Then we are faced also with the covariance matrix%
\[
C_{N,r}\left(  x,y\right)  =\sum_{k}\left(  \theta_{k}^{N}\right)  ^{2}\left(
\nabla e_{k}\left(  x\right)  r\right)  \otimes\left(  \nabla e_{k}\left(
y\right)  r\right)  .
\]
Clearly, since (roughly speaking)\ $\left\vert \nabla e_{k}\left(  x\right)
\right\vert \sim\left\vert k\right\vert $, we cannot hope that the classical
scaling (\ref{classical scaling}) of the scalar case holds and simultaneously
also the covariance matrix $C_{N,r}\left(  x,y\right)  $ has a finite limit in
$N$. For this reason we introduce a different scaling w.r.t.
(\ref{classical scaling}), where $Q_{N}\left(  x,y\right)  $ goes to zero
outside and also on the diagonal with the right speed but $C_{N,r}\left(
x,y\right)  $ satisfies, very heuristically,
\begin{align*}
\lim_{N\rightarrow0}C_{N,r}\left(  x,x\right)   &  =A\left(  r\right)  \\
\lim_{N\rightarrow0}C_{N,r}\left(  x,y\right)   &  =0\text{ for }x\neq y
\end{align*}
where $A\left(  r\right)  $ is the matrix function described in this paper
(equal to a multiple of $3\left\vert r\right\vert ^{2}I-2r\otimes r$). The
consequence of this scaling is that no diffusion in the $x$-variable appears,
opposite to the scalar case; but a diffusion in $r$ appears, with diffusion
matrix $A\left(  r\right)  $. In other words, the effect of transport becomes
irrelevant in the scaling limit, while the effect of stretching has a limit
which is a diffusion in the length variable $r$, with a special diffusion matrix.\\

Finally, let us compare the present work with \cite{BFLT2024}. The work is devoted
to a stochastic PDE for a passive magnetic field, subject to a model of
turbulence fluid. In \cite{BFLT2024} it is shown that it is possible to
investigate a Vlasov-type equation, suitably associated to the equation of the
magnetic field, prove a diffusion limit result for it and deduce informations
on the increase of magnetic field, in the scaling regime that controls the
derivatives of the stretching used also here. The two works stem from a
similar physical intuition, but the two problems are very different, the
present one starting from a Lagrangian description of a polymer, while
\cite{BFLT2024} from the Eulerian fluid dynamic description of a magnetic field.
Moreover, the mathematical techniques to deal with them are completely
different: in \cite{BFLT2024} the Vlasov-type equation is investigated in the
space of measures, in the spirit of Young measure solutions, and the technique
to prove well posedness and scaling limit is ultimately based on the
properties of the solutions of the original Eulerian SPDE. On the contrary, in
the present paper we analyze the Fokker-Planck equation directly in a space of
functions, with non trivial technical details related to weighted spaces (not
used before in the above mentioned literature on these scaling limit problems).\\

\subsubsection*{Structure of the paper}
 The manuscript is   organized   as   follows: in \autoref{Section1-desciption}, we  present the functional  and  stochastic  settings.   Then we introduce  in \autoref{Corrector section} the stochastic  FP  in  It\^o   form    after   presenting  some    properties  of  the covariance  operator.         \autoref{section-main-result}    collects    the main    results of  our work,   the computation of   the explicit rotation-invariant solution of the associated stationary equation of  \eqref{equ-final}. In \autoref{section-outlines}, we highlight the ideas of the proof and some formal calculations as well as the technical challenges.  \autoref{Section-exitence}   is  devoted to  the proof   of  the existence   and uniqueness  of  \textit{quasi-regular}    solution    to the  stochastic  FP \eqref{Ito-FP}.  In  \autoref{Section-diffusion},  we  prove   the convergence of the  stochastic  FP \eqref{Ito-FP}   to  the limit   PDE \eqref{equ-final}.       
 
 \section{Notations, definitions and problem formulation}\label{Section1-desciption}

 \subsubsection{Notations    and functional  setting  }
We  will    consider the periodic boundary conditions with respect to  the spacial variable    $x$,    namely  $x$ belongs to  the  2-dimensional torus $\T^2= (\mathbb{R}/2\pi\mathbb{Z})^2.$  On  the other   hand,   the end-to-end  vector  variable  $r$   belongs to  $\R^2.$   Let	$m\in	\mathbb{N}^*$  and  introduce the following Lebesgue and Sobolev spaces with polynomial weight, namely
\begin{align*}
&L^2_{r,m}(\mathbb{T}^2\times\mathbb{R}^2):=\{f:\mathbb{T}^2\times\mathbb{R}^2\to \mathbb{R}:  \int_{\mathbb{T}^2\times\mathbb{R}^2}\vert f(x,r)\vert^2(1+\vert r \vert^2)^{m/2}dxdr:= \Vert f\Vert_{L^2_{r,m}}^2 < \infty \},\\
&H^l_{r,m}(\mathbb{T}^2\times\mathbb{R}^2):=\{ f:\mathbb{T}^2\times\mathbb{R}^2\to \mathbb{R}: \hspace{-0.2cm}\sum_{\vert\gamma\vert+\vert\beta\vert \leq l }\int_{\mathbb{T}^2\times\mathbb{R}^2}\vert\partial_x^\gamma\partial_r^\beta f(x,r)\vert^2(1+\vert r \vert^2)^{m/2}dxdr= \Vert f\Vert_{H^l_{r,m}}^2 < \infty \},
\end{align*}
where	$l\in	\mathbb{N}^*.$	Now,	let	us	precise	the	functional	setting	to	study	\eqref{Ito-FP}.	We	will	use	the	following	notations	$$H^2_{r,4}(\mathbb{T}^2\times\mathbb{R}^2):=V,\quad	L^2_{r,2}(\mathbb{T}^2\times\mathbb{R}^2):=H.$$	

We	recall	the definition of	 inner products	defined	on	the	spaces	$V$	and	$H$.
\begin{align*}
(h,g)_V&:= \sum_{\vert\gamma\vert+\vert\beta\vert \leq 2 }\int_{\mathbb{T}^2\times\mathbb{R}^2}\partial_x^\gamma\partial_r^\beta h(x,r)\partial_x^\gamma\partial_r^\beta g(x,r)(1+\vert r \vert^2)^2dxdr,\quad	\forall	g,h \in V;\\
(h,g)_H&:= \int_{\mathbb{T}^2\times\mathbb{R}^2} h(x,r) g(x,r)(1+\vert r \vert^2)dxdr,\quad	\forall	g,h \in H;\\
(h,g)&:= \int_{\mathbb{T}^2\times\mathbb{R}^2} h(x,r) g(x,r)dxdr,\quad	\forall	g,h \in L^2(\mathbb{T}^2\times\mathbb{R}^2).
\end{align*}

  Since	$L^{\infty}(0,T;H)$	is	not	separable,	
it is	convenient	to	introduce	the	following	space:
$$L^2_{w-*}(\Omega;L^\infty(0,T;H))=\{ 	u:\Omega\to	L^\infty(0,T;H)	\text{	is	 weakly-* measurable		and	}	\E\Vert	u\Vert_{L^\infty(0,T;H)}^2<\infty\},$$
where	weakly-* measurable	stands	for	the	measurability	when	$L^\infty(0,T;H)$	is	endowed	with	the	$\sigma$-algebra	generated	by	the	Borel	sets	of	weak-*	topology, we recall that (see 	\cite[Thm. 8.20.3]{Edwards}) $$ L^2_{w-*}(\Omega ;L^\infty([0, T];H))\simeq \left(L^2(\Omega ;L^1([0, T];H^\prime)) \right)^\prime ,$$	

For $y=(y_1,y_2)\in   \R^2,$  $y^\perp$   stands  for $(-y_2,y_1).$ We  recall the following   notations:
\begin{itemize}
    \item   $(\nabla_x  g)_{i=1,2}=(\dfrac{\partial g}{\partial  x_i})_{i=1,2}; (\nabla_r  g)_{i=1,2}=(\dfrac{\partial g}{\partial  r_i})_{i=1,2} $ for scalar   function    $g$.
    \item   $(\nabla_x  g)_{i,j=1,2}=(\dfrac{\partial g^i}{\partial  x_j})_{i,j=1,2}; (\nabla_r  g)_{i,j=1,2}=(\dfrac{\partial g^i}{\partial  r_j})_{i,j=1,2} $ for vector  valued   function   $g$.
  \item     $\Delta_xg=\displaystyle\sum_{i=1}^2\dfrac{\partial^2 g}{\partial^2  x_i};  \Delta_rg=\displaystyle\sum_{i=1}^2\dfrac{\partial^2 g}{\partial^2  r_i}$  for scalar   function    $g$.   
\end{itemize}
and     $\Div_{x/r}   g=\nabla_{x/r}    \cdot   g$  for vector  valued   function   $g$.
We  don't   stress  the subscript  in $\nabla_x$  when    it  is  clear   from    the context.\\

In  order   to  prove   a    uniqueness  results, we  will  need  some    regularization  kernel. More  precisely,     let $\delta>0$    and      $\rho$ be   a smooth density of a probability measure on
$\R^2$, compactly supported in $B(0, 1)$    and define  the approximation of identity for the convolution on $\R^2$ as  $\rho_\delta(y)=\dfrac{1}{\delta^2}\rho(\dfrac{y}{\delta})$  (We also assume that $\rho$ is radially symmetric).     Since   we  are working on  $\T^2\times \R^2,$  we  recall  that    for any  integrable function $g$    on $\T^2$, $g$ can be extended
periodically to a locally integrable function on the whole $\R^2$   and
convolution $\rho_\delta * g$ is    meaningful   and  $\rho_\delta * g$ is still a $C^\infty$-periodic function. \\

Finally,
throughout the article,  we denote by $C,C_i, i\in \mathbb{N}$,   generic constants, which may vary from line to line.
	\subsubsection{Assumptions on the noise}\label{assumption_noise-2D}
Consider $\mathbb{Z}_{0}^{2}:=\mathbb{Z}^2-\{(0,0)\}$ divided into its four quadrants (write
$k=\left(  k_{1},k_{2}\right)  $)%
\begin{align*}
K_{++}  & =\left\{  k\in\mathbb{Z}_{0}^{2}:k_{1}\geq0,k_{2}>0\right\};\quad
K_{-+}   =\left\{  k\in\mathbb{Z}_{0}^{2}:k_{1}<0,k_{2}\geq0\right\}  \\
K_{--}  & =\left\{  k\in\mathbb{Z}_{0}^{2}:k_{1}\leq0,k_{2}<0\right\} ; \quad
K_{+-}   =\left\{  k\in\mathbb{Z}_{0}^{2}:k_{1}>0,k_{2}\leq0\right\}
\end{align*}
and set%
\begin{align*}
K_{+}  & =K_{++}\cup K_{+-};  \quad 
K_{-}   =K_{-+}\cup K_{--}  \text{  and }   K=K_{+}\cup K_{-}.
\end{align*}
Let $(\Omega,\mathcal{F},(\mathcal{F}_t)_t,P)$  be  a   complete   filtered probability space.  
    Define
\begin{align*}
\sigma_{k}^N\left(  x\right)     =\theta_{k}^N\frac{k^{\perp}}{\left\vert
k\right\vert }\cos k\cdot x,\qquad k\in K_{+},  \quad 
\sigma_{k}^N\left(  x\right)     =\theta_{k}^N\frac{k^{\perp}}{\left\vert
k\right\vert }\sin k\cdot x,\qquad k\in K_{-}%
\end{align*}
where
\begin{align*}
\theta^N_{k}&=\dfrac{a}{\left\vert k\right\vert^2 },  \qquad   N \leq \left\vert k\right\vert \leq  2N,\qquad N\in \mathbb{N}^* ; \quad
\theta^N_{k}=0 \qquad \text{elsewhere.}
\end{align*}
  This is the main assumption about the shell structure of the noise; and $a$ is a positive constant measuring the intensity.
Since $\theta^N_{k}$ depends only on $\left\vert k\right\vert,$  sometimes we write
$\theta_{\left\vert k\right\vert }$.
Notice that%
\begin{align*}
\partial_{i}\sigma_{k}^{\alpha}\left(  x\right)    & =-\theta_{k}\frac
{k_{i}\left(  k^{\perp}\right)  _{\alpha}}{\left\vert k\right\vert }\sin
k\cdot x,\qquad k\in K_{+},\quad    \alpha,i=1,2,\\
\partial_{i}\sigma_{k}^{\alpha}\left(  x\right)    & =\theta_{k}\frac
{k_{i}\left(  k^{\perp}\right)  _{\alpha}}{\left\vert k\right\vert }\cos
k\cdot x,\qquad k\in K_{-}.%
\end{align*}

Let us also consider a family $(W_t^k)_t^{k\in    \mathbb{Z}^2_0}$ of independent Brownian motions on the probability space $(\Omega,\mathcal{F},P)$. On the same probability space we shall soon assume that there exists another independent 2-dimensional Brownian motion $(\mathcal{W}_t)_t$.

\subsubsection{Dense    subsets  in  the space   $L^2(\Omega)$} The purpose of this part is to recall a density result, which will play a crucial role in \autoref{Subsection-uniq} and \autoref{Subsection-uniq-quasi-weak}. Namely,
we    will    prove   that    uniqueness  holds   a  particular  class  of  solution,    see \autoref{Def-sol}. It is sometimes called Wiener uniqueness and was used for instance by \cite{LeJanRaimond,Maurelli,NevesOlivera,FedrizziNevesOlivera,FlaOli2018}.
For that, let  $\mathcal{G}_t$ be  the filtration associated   with     $(W_t^k)_t^{k\in    \mathbb{Z}^2_0}$   namely $
      \mathcal{G}_t=\sigma\{W_s^k;  s\in[0,t],  k\in    \mathbb{Z}^2_0\},$
  and   denote  by  $\overline{\mathcal{G}}_t$    its completed   filtration\footnote{We assume that $\mathcal{F}_0$ contains all the $P$-null subset of    $\Omega$.}. For $T>0,$  let us  introduce
  \begin{align*}
      \mathcal{H}&=L^2(\Omega,\overline{\mathcal{G}}_T,P),   \quad     M_n=\{k\in  \mathbb{Z}^2_0;\quad  \vert   k\vert  \leq    n\}\\
      G&=\displaystyle\bigcup_{n\in\mathbb{N}}G_n;\quad  G_n=\{g=(g_k)_{k\in M_n};    g_k\in  L^2(0,T);\quad    \forall k\in    M_n\}.  \end{align*}
  For   $n\in   \mathbb{N}, g\in    G_n$,   we  set
  \begin{align*}
      e_g(t)&=\exp{\big(\sum_{k\in M_n}\int_0^tg_k(s)dW^k(s)-\dfrac{1}{2}\sum_{k\in    M_n}\int_0^t\vert   g_k(s)\vert^2   ds\big)}, \text{  for }   t\in[0,T];\\
         \mathcal{D}&=\{e_g(T);    \quad   g\in    G\}.
  \end{align*}
  From  It\^o   formula,    we  get
  $
      de_g(t)=\displaystyle\sum_{k\in M_n}g_k(t)e_g(t)dW^k(t).
  $ 
 Based  on  the Wiener chaos decomposition, we    recall    the following   result, see \cite[Ch.   1]{Nualart2006}.
      \begin{lemma}
          $\mathcal{D}$ is  dense   in  $\mathcal{H}.$
      \end{lemma}

\subsubsection{Lagrangian description  and  stochastic Fokker-Planck equation }\label{sub-why-FP}	
Let $X_t\in \mathbb{R}^2$ and $R_t\in \mathbb{R}^2$ be     the position and end-to-end vector of the polymer. In the Introduction we have generically stated that they satisfies \eqref{Intro1}. Here we shall be more specific. Let $(X_t,R_t)$ satisfying
\begin{align}\label{eqn-Larg-short}
	\begin{cases}
	dR_t&=\nabla u(X_t,t) R_tdt-\dfrac{1}{\beta}R_tdt+\sqrt{2}\sigma d\mathcal{W}_t, \\
	dX_t&=u(X_t,t)dt,
	\end{cases}
	\end{align}
 where $\nabla$    is  the gradient    with    respect to  $x$-variable.
    We assume that the velocity field $u$ is the sum of a large scale divergence-free component $u_L(x,t)$ (deterministic, with a reasonable smoothness specified below) plus a stochastic small-scale component, precisely given by the noise coefficients introduced above, in other words
\begin{align}\label{turb-velocity-2D-eq}
  u(x,t):= u^N(x,t)=u_L(x,t)+
 \circ\sum_{k\in K}\sigma_k^N(x)\partial_t W^k, 
\end{align}
 where we choose the Stratonovich multiplication both in virtue of Wong-Zakai principle (a white noise is the idealization of smooth noise) and because of conservation laws. Since the Brownian motions $(\mathcal{W}_t)_t$, due to thermal force( microscopic level), in \eqref{eqn-Larg-short} acts on a very short time scale and also that the polymer length  is  much smaller than the viscous length of the turbulent flow, we write firstly the FP equations with respect to it,  with quenched velocity. Then we use the random velocity  given  by \eqref{turb-velocity-2D-eq} in the resulting  FP equation, with  Brownian motions $(W_t^k)_t^{k\in K}$ having a different  time scale (mesoscopic/macroscopic level). 
   Formally speaking, we may associate a stochastic Fokker-Planck equation in Stratonovich form:
 \begin{align}\label{Stra-FP}
\begin{cases}\partial_tf^N(x,r,t)&+\Div_x(u_L(x,t)f^N(x,r,t))+\Div_r((\nabla u_L(t,x)r-\dfrac{1}{\beta}r)f^N(x,r,t))=	\sigma^2\Delta_rf^N(x,r,t)\\	&\quad -\sum_{k\in K}\sigma_k^N.\nabla_xf^N(x,r,t)\circ\partial_tW^k -\sum_{k\in K}(\nabla \sigma_k^Nr).\nabla_rf^N(x,r,t)\circ \partial_tW^k\\	
	f^N|_{t=0}&=f_0,
	\end{cases}\end{align}
 where we have used the properties $\Div_x(\sigma_k^N)=0,    k\in    K$, $\Div_x(u_L)=0$    and  $\Div_r(\nabla \sigma_k^Nr)=0$\footnote{$\Div_r(\nabla \sigma_k^Nr)=0$  is  consequence of  $\Div_x(\sigma_k^N)=0$.}. Therefore,  \eqref{Stra-FP}  will be used as the main equation we study in this work. \\
 
 Before we formulate rigorously the meaning of the equation \eqref{Stra-FP}, let us rewrite the previous equation (formally) from the Stratonovich to the   It\^o  form. The question is the form of the {It\^o-Stratonovich  corrector. This computations requires some additional care with respect to previously known cases developed in the literature, hence we devote to it a separate section.
  
		\section{It\^o-Stratonovich  correctors}\label{Corrector section}
   		Let       $\psi$ be a given smooth  function  and			set $Q(x,y):=\sum_{k\in K}\sigma_k^N(x)\otimes\sigma_k^N(y).$ 
	In this part, we will write the It\^o form associated with  \eqref{Stra-FP} when $Q$ is space-homogeneous and has a mirror symmetry property.\\	
	
	Denote by $L_k\psi=-(\sigma_k^N.\nabla_x\psi +(\nabla \sigma_k^Nr).\nabla_r\psi)$ then the corrector term is given by $\dfrac{1}{2}\sum_{k\in K}L_kL_k\psi$. Recall that $\Div_x(\sigma_k^N)=0$ and let us compute $L_kL_k\psi.$
	\begin{align*}
	L_kL_k\psi=&-(\sigma_k^N.\nabla_xL_k\psi +(\nabla \sigma_k^Nr).\nabla_rL_k\psi)\\	
	=&(\sigma_k^N.\nabla_x(\sigma_k^N.\nabla_x\psi +(\nabla \sigma_k^Nr).\nabla_r\psi) +(\nabla \sigma_k^Nr).\nabla_r(\sigma_k^N.\nabla_x\psi +(\nabla \sigma_k^Nr).\nabla_r\psi)\\	
	=&\Div_x((\sigma_k^N \otimes \sigma_k^N) \nabla_x\psi)+\Div_r(((\nabla \sigma_k^Nr) \otimes (\nabla \sigma_k^Nr)) \nabla_r\psi)\\&
+	\sigma_k^N.\nabla_x((\nabla \sigma_k^Nr).\nabla_rf)+(\nabla \sigma_k^Nr).\nabla_r(\sigma_k^N.\nabla_x\psi).
	\end{align*}

 \begin{lemma}\label{lem1}
		Assume  the noise is space-homogeneous    i.e. $Q(x,y)=Q(x-y)$ and  $Q(x,x)=Q(0)$, be  a constant matrix, then 	\begin{align*}
I(\psi)=I^1(\psi)+I^2(\psi)&=\sum_{k\in K}	\sigma_k^N.\nabla_x((\nabla \sigma_k^Nr).\nabla_r\psi)+(\nabla \sigma_k^Nr).\nabla_r(\sigma_k^N.\nabla_x\psi)\\&=2	\sum_{l,\gamma,i=1}^2\partial_{x_\gamma} Q_{i,l}(0)\partial_{x_l}(r_\gamma\partial_{r_i}\psi).	\end{align*}
If  moreover $Q$    has the  mirror symmetry property   i.e.    $Q(x)=Q(-x)$, then $I(\psi)=0$.
	\end{lemma}
\begin{proof}
Let $\psi$ be a smooth   function    (we drop the dependence of  $\sigma_k$ on $N$ here for the simplicity of notation). We have
		\begin{align*}
I^1(\psi)&=	\sum_{k\in K} \sum_{l,\gamma,i=1}^2\sigma_k^l\partial_{x_l} (\partial_{x_\gamma}\sigma_k^ir_\gamma\partial_{r_i}\psi)\\&= \sum_{k\in K} \sum_{l,\gamma,i=1}^2\sigma_k^l(\partial_{x_l} \partial_{x_\gamma}\sigma_k^i)r_\gamma\partial_{r_i}\psi+\sum_{k\in K} \sum_{l,\gamma,i=1}^2\sigma_k^l \partial_{x_\gamma}\sigma_k^i\partial_{x_l}(r_\gamma\partial_{r_i}\psi).
	\end{align*}
First,	let us compute the second term in the last equation. We have 	\begin{align*}
			\sum_{k\in K}  \sigma_k^l(y) \partial_{x_\gamma}\sigma_k^i(x)=
			\partial_{x_\gamma}	\sum_{k\in K}  \sigma_k^l(y)\sigma_k^i(x)=\partial_{x_\gamma} Q_{i,l}(x-y), 
			\end{align*}
			which gives ( we recall that $Q$ is space-homogeneous)
			\begin{align}\label{2--}
			\sum_{k\in K}  \sigma_k^l(x) \partial_{x_\gamma}\sigma_k^i(x)=\partial_{x_\gamma} Q_{i,l}(0).
			\end{align}
			Thus $
			\displaystyle\sum_{k\in K} \sum_{l,\gamma,i=1}^2\sigma_k^l \partial_{x_\gamma}\sigma_k^i\partial_{x_l}(r_\gamma\partial_{r_i}\psi)=\sum_{l,\gamma,i=1}^2\partial_{x_\gamma} Q_{i,l}(0)\partial_{x_l}(r_\gamma\partial_{r_i}\psi).$
		Next, let us compute $I^2$				\begin{align*}
				I^2(\psi)&=\sum_{k\in K}	(\nabla \sigma_k^Nr).\nabla_r(\sigma_k^N.\nabla_x\psi)=	\sum_{k\in K} \sum_{l,\gamma,i=1}^2 \partial_{x_\gamma}\sigma_k^ir_\gamma\partial_{r_i}(\sigma_k^l\partial_{x_l} \psi)\\
				&=	\sum_{k\in K} \sum_{l,\gamma,i=1}^2 \partial_{x_\gamma}\sigma_k^i\sigma_k^l\partial_{x_l}(r_\gamma\partial_{r_i} \psi)=\sum_{l,\gamma,i=1}^2\partial_{x_\gamma} Q_{i,l}(0)\partial_{x_l}(r_\gamma\partial_{r_i}\psi).
					\end{align*}
				Now,	let us prove that the first part of $I^1$ vanishes.    Namely
					\begin{align*}	
					\sum_{k\in K} \sum_{l,\gamma,i=1}^2\sigma_k^l(\partial_{x_l} \partial_{x_\gamma}\sigma_k^i)r_\gamma\partial_{r_i}\psi=
				 \sum_{\gamma,i=1}^2 	\left(\sum_{l}^2\sum_{k\in K}\sigma_k^l\partial_{x_l} \partial_{x_\gamma}\sigma_k^i\right)r_\gamma\partial_{r_i}\psi=0.
						\end{align*}
						It is sufficient to show that  $\displaystyle\sum_{l}^2\sum_{k\in K}\sigma_k^l\partial_{x_l} \partial_{x_\gamma}\sigma_k^i=0.$ Indeed,    notice    that
							\begin{align*}	
							\sum_{l}^2\sum_{k\in K}\sigma_k^l\partial_{x_l} \partial_{x_\gamma}\sigma_k^i=\sum_{l}^2\partial_{x_\gamma}\sum_{k\in K}\sigma_k^l\partial_{x_l} \sigma_k^i-\sum_{l}^2\sum_{k\in K}\partial_{x_\gamma}\sigma_k^l\partial_{x_l} \sigma_k^i=-\sum_{l}^2\sum_{k\in K}\partial_{x_\gamma}\sigma_k^l\partial_{x_l} \sigma_k^i,
								\end{align*}
							where we used similar arguments 	 to  the    one used    to  obtain  \eqref{2--} to get 
       \begin{align*}\sum_{k\in K}  \sigma_k^l(x) \partial_{x_l}\sigma_k^i(x)=\partial_{x_l} Q_{i,l}(0) \text{ and  }  \partial_{x_\gamma}\sum_{k\in K}\sigma_k^l\partial_{x_l} \sigma_k^i=0.
       \end{align*} On the other hand, note that
							\begin{align*}			
								\sum_{l}^2\sum_{k\in K}\partial_{x_\gamma}\sigma_k^l\partial_{x_l} \sigma_k^i=	\sum_{l}^2 \partial_{x_l}\sum_{k\in K}\partial_{x_\gamma}\sigma_k^l \sigma_k^i- \sum_{k\in K}\partial_{x_\gamma} (\sum_{l}^2\partial_{x_l}\sigma_k^l) \sigma_k^i=\sum_{l}^2 \partial_{x_l}\sum_{k\in K}\partial_{x_\gamma}\sigma_k^l \sigma_k^i,
										\end{align*}
										since $\Div_x(\sigma_k)=0.$ Again, note that 
										$		
										\sum_{k\in K}\partial_{x_\gamma}\sigma_k^l \sigma_k^i=\partial_{x_\gamma} Q_{l,i}(0).
											$
												Therefore
										\begin{align*}		
												\partial_{x_l}\sum_{k\in K}\partial_{x_\gamma}\sigma_k^l \sigma_k^i=\partial_{x_l}(\partial_{x_\gamma} Q_{l,i}(0))=0.
												\end{align*}
												Summing up, we get
											$				I(\psi)= 2	\sum_{l,\gamma,i=1}^2\partial_{x_\gamma} Q_{i,l}(0)\partial_{x_l}(r_\gamma\partial_{r_i}\psi).	$
												If $Q(x)=Q(-x)$ and $Q$ is smooth function, we see that $\partial_{x_\gamma} Q_{i,l}(0)=0$ and the second part of   \autoref{lem1} follows.
\end{proof}

Moreover, the correctors have a special form.

\begin{lemma}\label{lem-matrix-r} The following equalities    hold
	\begin{align*}
\dfrac{1}{2}	\Div_x(\sum_{k\in K}(\sigma_k^N \otimes \sigma_k^N) \nabla_xf)&=
\dfrac{1}{2}\sum_{k\in K_{++}}(\theta_{\left\vert k\right\vert }^N)^2	\Delta_xf:=\alpha_N\Delta_xf,\\
\sum_{k\in K}\left((\nabla \sigma_k^Nr) \otimes (\nabla \sigma_k^Nr)\right) &= A(r)+O(\dfrac{1}{N})P(r),	\end{align*}
where $A(r)
=k_T (3 \left\vert r\right\vert^2I-2r\otimes r) =k_T (\left\vert r\right\vert^2I+2r^\perp\otimes r^\perp) , $
  $k_T=\dfrac{\pi \log(2)}{8}a^2$, $r=(r_1,r_2)$ and $P$ is a polynomial of second degree.
\end{lemma}
\begin{proof} 
Let us simplify the expressions of 
$$ \mathcal{S}(f):=\dfrac{1}{2}\Div_r(\sum_{k\in K}\left((\nabla \sigma_k^Nr) \otimes (\nabla \sigma_k^Nr)\right) \nabla_rf) \text{ and } \mathcal{B}(f):=\dfrac{1}{2}	\Div_x(\sum_{k\in K}(\sigma_k^N \otimes \sigma_k^N) \nabla_xf).$$
First,  we  consider    the term $\mathcal{B}$.
Recall that 
\begin{align*}
\sigma_{k}^N\left(  x\right)     =\theta_{\left\vert k\right\vert}^N\frac{k^{\perp}}{\left\vert
	k\right\vert }\cos k\cdot x,\quad k\in K_{+},\qquad
\sigma_{k}^N\left(  x\right)     =\theta_{\left\vert k\right\vert }^N\frac{k^{\perp}}{\left\vert
	k\right\vert }\sin k\cdot x,\quad k\in K_{-}.%
\end{align*}
We have
\begin{align*}
\sum_{k\in K}\sigma_k^N \otimes \sigma_k^N&=\sum_{k\in K_+}(\theta_{\left\vert k\right\vert }^N)^2\frac{k^{\perp}\otimes k^{\perp}}{\left\vert
	k\right\vert^2 }\cos^2 k\cdot x+\sum_{k\in K_-}(\theta_{\left\vert k\right\vert }^N)^2\frac{k^{\perp}\otimes k^{\perp}}{\left\vert
	k\right\vert^2 }\sin^2 k\cdot x\\
(k\in K_- \to -k\in K_+)&=\sum_{k\in K_+}(\theta_{\left\vert k\right\vert }^N)^2\frac{k^{\perp}\otimes k^{\perp}}{\left\vert
	k\right\vert^2 }\\
( k\in K_{+-} \to k^{\perp} \in K_{++})&=\sum_{k\in K_{++}}(\theta_{\left\vert k\right\vert }^N)^2\left(\frac{k^{\perp}\otimes k^{\perp}}{\left\vert
k\right\vert^2 }+\frac{k\otimes k}{\left\vert
k\right\vert^2 }\right)\quad  (K_+=K_{++}\cup K_{+-}).
\end{align*}
Thus $
\sum_{k\in K}\sigma_k^N \otimes \sigma_k^N=\sum_{k\in K_{++}}(\theta_{\left\vert k\right\vert }^N)^2 I; \quad I= \left(
\begin{array}
[c]{cc}%
1& 0\\
0 &  1
\end{array}
\right) $   and
therefore
\begin{align*}
\dfrac{1}{2}	\Div_x(\sum_{k\in K}(\sigma_k^N \otimes \sigma_k^N) \nabla_xf)=
\dfrac{1}{2}\sum_{k\in K_{++}}(\theta_{\left\vert k\right\vert }^N)^2	\Delta_xf:=\alpha_N\Delta_xf. \end{align*}
On  the other   hand,
let us   present some properties of $\mathcal{S}$. 
  We have (with slight abuse of notation we use $\sigma_k$ instead of $\sigma_k^N$ )
		\begin{align*}	
	&\left(\sum_{k\in K}(\nabla \sigma_kr) \otimes (\nabla \sigma_kr)\right)_{i,l}=	\sum_{j,\alpha =1}^2\sum_{k\in K}(\partial_{x_j} \sigma_k^ir_j \partial_{x_\alpha} \sigma_k^lr_\alpha)\\
	&=\sum_{k\in K_+}(\theta_{\left\vert k\right\vert }^N)^2(k\cdot r)^2\frac{k^{\perp}\otimes k^{\perp}}{\left\vert
		k\right\vert^2 }\sin^2 k\cdot x+\sum_{k\in K_-}(\theta_{\left\vert k\right\vert }^N)^2(k\cdot r)^2\frac{k^{\perp}\otimes k^{\perp}}{\left\vert
		k\right\vert^2 }\cos^2 k\cdot x\\
&	(k\in K_- \to -k\in K_+)=\sum_{k\in K_+}(\theta_{\left\vert k\right\vert }^N)^2(k\cdot r)^2\frac{k^{\perp}\otimes k^{\perp}}{\left\vert
		k\right\vert^2 }=a^2\sum_{\substack {k\in K_{+}\\N \leq \left\vert k\right\vert \leq  2N}}\frac{1}{\left\vert
		k\right\vert^6 }(k\cdot r)^2k^{\perp}\otimes k^{\perp}.
	\end{align*}
Now, let us compute $\displaystyle\sum_{\substack {k\in K_{+}\\N \leq \left\vert k\right\vert \leq  2N}}\frac{1}{\left\vert
	k\right\vert^6 }(k\cdot r)^2k^{\perp}\otimes k^{\perp}.$ Note that 
	\begin{align*}
\sum_{\substack {k\in K_{+}\\N \leq \left\vert k\right\vert \leq  2N}}\frac{1}{\left\vert
	k\right\vert^6 }(k\cdot r)^2k^{\perp}\otimes k^{\perp}=
\frac{1}{N^2}\sum_{\substack {k\in K_{+}\\1 \leq \frac{\left\vert k\right\vert}{N} \leq  2}}\frac{N^6}{\left\vert
k\right\vert^6 }\frac{1}{N^2}(k\cdot r)^2\frac{1}{N^2}(k^{\perp}\otimes k^{\perp}).
	\end{align*}
	Note that $h_r(x)=\dfrac{1}{x^6}(x\cdot r)^2(x^{\perp}\otimes x^{\perp})$ is smooth function   for $1 \leq \left\vert x\right\vert\leq  2.$ By using Riemann sum, we get
	\begin{align*}
	\frac{1}{N^2}\sum_{\substack {k\in K_{+}\\1 \leq \frac{\left\vert k\right\vert}{N} \leq  2}}\frac{N^6}{\left\vert
		k\right\vert^6 }\frac{1}{N^2}(k\cdot r)^2\frac{1}{N^2}(k^{\perp}\otimes k^{\perp})=\int_{D}h_r(x)dx+ O(\frac{1}{N})P(r),
	\end{align*}
	where $D=\{x=(\vert x\vert\cos(\varphi),\vert x\vert\sin(\varphi)): 1\leq \vert x\vert \leq 2  \text{ and }  \varphi \in D_\pi:=]0, \dfrac{\pi}{2}]\cup ]\dfrac{3\pi}{2}, 2\pi] \}$ and $P$ is a polynomial of second degree.   On the other hand, we  have
	\begin{align*}
\int_{D}h_r(x)dx=\int_1^2\dfrac{1}{z}\int_{D_\pi}( r_1\cos(\varphi)+r_2\sin(\varphi))^2 \left(
\begin{array}
[c]{cc}%
\sin^2(\varphi) & -\sin(\varphi)\cos(\varphi)\\
-\sin(\varphi)\cos(\varphi) &  \cos^2(\varphi)
\end{array}
\right)d\varphi   dz.
\end{align*}
Let us compute the following integral
	\begin{align*}
I(r):=\int_{D_\pi}( r_1\cos(\varphi)+r_2\sin(\varphi))^2 \left(
\begin{array}
	[c]{cc}%
	\sin^2(\varphi) & -\sin(\varphi)\cos(\varphi)\\
	-\sin(\varphi)\cos(\varphi) &  \cos^2(\varphi)
\end{array}
\right)d\varphi.
\end{align*}
A   standard    integration with    respect to  $\varphi$   gives
\begin{align*}I(r)=\frac{\pi }{8} \left(
\begin{array}
[c]{cc}%
3 \left\vert r\right\vert^2-2r_1^2& -2r_1r_2\\
-2r_1r_2 &  3 \left\vert r\right\vert^2-2r_2^2
\end{array}
\right) ; r=(r_1,r_2). \end{align*} 
Since $\displaystyle\int_1^2\dfrac{1}{z} dz=\log(2),$ we get the final expression of $A(r)$.
\end{proof}

\begin{remark}
    When $\theta^N_{k}=\dfrac{a}{\left\vert k\right\vert^2 }$ if $ N \leq \left\vert k\right\vert \leq  2N, N\in \mathbb{N}^* $ and $\theta^N_{k}=0$ elsewhere,  we get   
\begin{align}\label{alpha-N}
0<\dfrac{\pi}{64}\dfrac{a^2}{N^3}  \leq \alpha_N \leq  \dfrac{\pi}{4}\dfrac{a^2}{N^3} <a^2  \text{  and }   \alpha_N\to 0.\end{align}
\end{remark}

Based on these two lemmata, the It\^o form (still formulated only fomally) of the  stochastic   Fokker Planck equation     \eqref{Stra-FP},
by  assuming that  the noise is space-homogeneous and  satisfies mirror symmetry property,
 is given by\footnote{Recall that $f$ depends on $t,x,r,\omega$ and $N$ but we don't stress the dependence on the above variables for the simplicity of notation, that is, with slight abuse of notation $f^N:=f^N(t,x,r,\omega)$.}

	\begin{align}\label{Ito-FP}
\begin{cases}&df^N+\Div_x(u_L(x,t)f^N)dt+\Div_r((\nabla u_L(t,x)r-\dfrac{1}{\beta}r)f^N) dt\\	&=	\sigma^2\Delta_rf^Ndt-\sum_{k\in K}\sigma_k^N.\nabla_xf^NdW^k -\sum_{k\in K}(\nabla \sigma_k^Nr).\nabla_rf^N dW^k\\
&\quad +\alpha_N\Delta_x f^Ndt+ \dfrac{1}{2}\Div_r(\sum_{k\in K}\left((\nabla \sigma_k^Nr) \otimes (\nabla \sigma_k^Nr)\right) \nabla_rf^N)dt\\
&f^N|_{t=0}=f_0.
\end{cases}\end{align}
where 
$$\sum_{k\in K}\left((\nabla \sigma_k^Nr) \otimes (\nabla \sigma_k^Nr)\right) = A(r)+O(\dfrac{1}{N})P(r).$$

\subsubsection{Assumptions on large scale component}
Let $T>0$. In  the following,  we  assume  that   $u_L\in C([0,T],C^2(\T^2;\R^2))$    such    that    $\Div_x(u_L)=0.$  

\section{Main   results}\label{section-main-result}

Following \cite{NevesOlivera,FedrizziNevesOlivera,FlaOli2018}.
We introduce the concept of   \textit{"quasi-regular weak solution"}     to  \eqref{Ito-FP},  where   we    prove   the well-posedness. Notice that  uniqueness in this class (sometimes called Wiener uniqueness)  is weaker     than    pathwise uniqueness.

\begin{definition}(Quasi-regular weak solution)\label{Def-sol}
 Let    $f_0\in H$    and   $N\in   \mathbb{N}$.   We  say that  $f^N$  is   quasi-regular weak solution   to  \eqref{Ito-FP}  if  
	$f^N$    is $(\mathcal{F}_t)_t$-adapted and
	\begin{enumerate}
		\item  $f^N \in L^2_{w-*}(\Omega;L^\infty([0, T];H)), \nabla_rf^N\in L^2(\Omega\times [0, T];H),$
		\item   P a.s. in  $\Omega:$  $f^N \in C_w([0, T];H)$\footnote{$C_w([0, T];H)$    denotes 	the	Bochner space of weakly continuous functions with values in $H.$},
		\item P-a.s: for any $t\in ]0, T]$  the following equation holds: 
		\begin{align*}&\int_{\mathbb{T}^2}\int_{\mathbb{R}^2}f^N(t) \phi drdx-\int_{\mathbb{T}^2}\int_{\mathbb{R}^2}f_0 \phi drdx\\&-\int_0^t\int_{\mathbb{T}^2}\int_{\mathbb{R}^2}f^N(s)\left(u_L(s)\cdot \nabla_x \phi + (\nabla u_L(s)r-\dfrac{1}{\beta}r) \cdot\nabla_r \phi \right)dr dxds\\	=&-	\sigma^2\int_0^t\int_{\mathbb{T}^2}\int_{\mathbb{R}^2}\nabla_rf^N(s) \cdot\nabla_r \phi drdxds+\alpha_N\int_0^t\int_{\mathbb{T}^2}\int_{\mathbb{R}^2}f^N(s) \cdot\Delta_x\phi   drdxds\\&+\sum_{k\in K}\int_0^t\int_{\mathbb{T}^2}\int_{\mathbb{R}^2}\big(f^N(s) \sigma_k^N.\nabla_x\phi +f^N(s) (\nabla \sigma_k^Nr).\nabla_r\phi\big) drdxdW^k(s)\\
    		&- \dfrac{1}{2}\int_0^t\int_{\mathbb{T}^2}\int_{\mathbb{R}^2}(\sum_{k\in K}\left((\nabla \sigma_k^Nr) \otimes (\nabla \sigma_k^Nr)\right) \nabla_rf^N(s))\cdot\nabla_r \phi drdxds,\quad   \text{  for any } \phi  \in U.\end{align*}
		
  \item (Regularity in Mean) For all    $n\in   \mathbb{N^*}$   and each function   $g\in   G_n$,   the deterministic   function    $V^N(t,x,r)=\E[f^N(t,x,r)e_g(t)]$ is  a   measurable  function,   which   belongs to  $ L^\infty([0, T];H)\cap    C_w([0, T];H)$  and $\nabla_rV^N\in L^2([0, T];H)$    and satisfies   the following   equation
\begin{align}
\dfrac{d   V^N}{dt}
&+\Div_x([u_L-h_n]V^N)+\Div_r([(\nabla u_Lr)-y_n]-\dfrac{1}{\beta}r)V^N)\notag\\	
&=	\sigma^2\Delta_rV^N+\alpha_N	\Delta_xV^N+ \dfrac{1}{2}\Div_r(\sum_{k\in K}\left((\nabla \sigma_k^Nr) \otimes (\nabla \sigma_k^Nr)\right) \nabla_rV^N),\notag
\end{align}
  in  a  weak    sense   (see    \autoref{prop-exisence-mean}),   where  
  \begin{align*}
  \displaystyle\sum_{k\in K_n}g_k\sigma_k^N=h_n   \text{  and } \displaystyle\sum_{k\in K_n}g_k(\nabla \sigma_k^Nr)=y_n,    \text{  where } K_n=\{k\in K: \min(n,N) \leq    \vert   k\vert  \leq    \max(2N,n)\} . 
      \end{align*}

	\end{enumerate}
	
\end{definition}
\begin{remark}
  The   point   $(3)$   in  \autoref{Def-sol}   is  satisfied   for larger  class   of  test    functions,  namely   $\phi  \in V$  thanks  to  the regularity  properties  of  $f^N.$
\end{remark}
The main   results  are  given   by  the following   theorems. 
\begin{theorem}\label{TH1}  There exists  at    least   one     solution $f^N$ to  \eqref{Ito-FP} in the sense of  \autoref{Def-sol}. Moreover, $(f^N)_N$  and $(\nabla_rf^N)_N$ are bounded in 
	$  L^2_{w-*}(\Omega;L^\infty([0, T];H))$     and $  L^2(\Omega\times [0, T];H) $ respectively.
\end{theorem}
\begin{proof}
  See    \autoref{Subsection-exist-proof}.
\end{proof}

\begin{theorem}\label{TH2}
Under   the assumptions  of  \autoref{TH1},  let $f^N_i,i=1,2,$  be  two quasi-regular    weak    solutions   of  \eqref{Ito-FP}  with    the same    initial data    $f_0$.  Assume  that    $(f^N_i(t),\varphi)$    is  $\overline{\mathcal{G}}_t$-adapted, for both    $i=1,2,$ for    any $\varphi\in V$. Then   $f^N_1=f^N_2.$
\end{theorem}
\begin{proof}
    See \autoref{Subsection-uniq-quasi-weak}.
\end{proof}
Concerning  the scaling limit   as  $N\to   +\infty$,  we  have    
\begin{theorem} \label{main-thm-2}	There  exists  a   new probability space,	denoted by  the same    way  (for simplicity)    $(\Omega,\mathcal{F},P)$,  $\overline{f}\in L^2_{w-*}(\Omega ;L^\infty([0, T];H))),  \nabla_r\overline{f}\in L^2(\Omega ;L^2([0, T];H)))   $    such    that  the following   convergence holds   (up to  a   sub-sequence):  $ f^N\rightarrow \overline{f}  \text{ in } C([0,T];U^\prime) \quad P-a.s.$  and
\begin{align*}
    f^N &\rightharpoonup \overline{f} \text{		in	}  L^2_{w-*}(\Omega ;L^\infty([0, T];H))),\quad
         \nabla_rf^N \rightharpoonup \nabla_r\overline{f} \text{		in	}  L^2(\Omega ;L^2([0, T];H)).\end{align*} Moreover   $\overline{f}$ is   the unique  solution    of  the following   problem:    	$P$-a.s. \begin{align}&\int_{\mathbb{T}^2}\int_{\mathbb{R}^2}\overline{f}(x,r,t)\phi(x)\psi(r) drdx-\int_{\mathbb{T}^2}\int_{\mathbb{R}^2}f_0(x,r)\phi(x)\psi(r) drdx\label{equ-final}\\&= \int_0^t\int_{\mathbb{T}^2}\int_{\mathbb{R}^2}\overline{f}(x,r,s)\left(u_L(x,s)\cdot \nabla_x \phi(x)\psi(r)+ (\nabla u_L(s,x)r-\dfrac{1}{\beta}r) \cdot\nabla_r\psi(r)\phi(x)\right) dr dxds\notag\\	&-	\sigma^2\int_0^t\int_{\mathbb{T}^2}\int_{\mathbb{R}^2}\nabla_r\overline{f}(x,r,s) \cdot\nabla_r\psi(r)\phi(x)  dr dxds-\dfrac{1}{2}\int_0^t\int_{\mathbb{T}^2}\int_{\mathbb{R}^2}A(r) \nabla_r\overline{f}(x,r,s))\cdot\nabla_r\psi(r)\phi(x) dr dxds,\notag\end{align}
for any $\phi  \in C^\infty(\mathbb{T}^2)$ and  $\psi  \in C^\infty_c(\mathbb{R}^2)$	and   $A(r)$  is  given   in  \autoref{lem-matrix-r}.\end{theorem} 
 \begin{proof}
     See    \autoref{section-scaling-PF}    for the existence proof  and \autoref{lemma-uniqunes-limit}  for the uniqueness.
 \end{proof}
 \begin{remark}
        By taking into    account the regularity  of  $\overline{f}$, \eqref{equ-final} can   be  written  as 	\begin{align}\label{Limit-FP-eqn}
\begin{cases}\partial_t\overline{f}(x,r,t)&+\Div_x(u_L(x,t)\overline{f}(x,r,t))+\Div_r((\nabla u_L(t,x)r-\dfrac{1}{\beta}r)\overline{f}(x,r,t))\\	&=	\sigma^2\Delta_r\overline{f}(x,r,t)+\dfrac{1}{2}\Div_r( A(r)\nabla_r \overline{f}(x,r,t))\\
f|_{t=0}&=f_0,
\end{cases}\end{align}
in $\mathcal{Y}^\prime$-sense,    where   $\mathcal{Y}=\{\varphi\in H:    \nabla_r\varphi\in H,  \nabla_x\varphi\in  L^2(\T^2\times\R^2)\}$.  Moreover,
    since  $\overline{f}\in L^\infty(0, T;H)$    and $\overline{f}  \in C([0,T];U^\prime)$, we  get $\overline{f}  \in C_w([0,T];H)$    $P$-a.s.$($    see       \cite[Lem. 1.4 p. 263]{Temam77}$)$.\\
        $\bullet$   Since   \eqref{Limit-FP-eqn}    is  Fokker-Planck equation  of  \eqref{FP_SDE-limit} then    $\overline{f}\geq   0$  and $\dfrac{d}{dt}\int_{\T^2\times   \R^2}   \overline{f}(t,x,r)dxdr=0$ if $f_0\geq 0$ and  $f_0\in L^1(\mathbb{T}^2\times \mathbb{R}^2).$
  \end{remark}
   
\begin{remark}\label{Rmq-SDE-FP}$($The  limit   FP,  Lagrangian  description  and macroscopic equation$)$ \\
   $\bullet$ From    \autoref{lem-matrix-r},   we  get
 for each $r\in \R^2$
   \begin{align*}
   A(r)=\frac{a^2\pi\log 2 }{8} \lvert r\rvert^2\frac{r\otimes r}{\lvert r\rvert^2}+\frac{
   2a^2\pi\log 2 }{8}\lvert r\rvert^2 \left(\frac{r^\perp\otimes r^\perp }{\lvert r\rvert^2}\right).
   \end{align*}
   Therefore $ A(r)=\mathcal{Q}(r)\mathcal{Q}(r)$, where 
   $
   \mathcal{Q}(r)=\frac{a\sqrt{\pi\log 2} }{2\sqrt{2}}\left( \dfrac{r\otimes r}{\lvert r\rvert}+\sqrt{2} \dfrac{r^\perp\otimes r^\perp }{\lvert r\rvert}\right).
   $
Notice    that  $\mathcal{Q}(r)$ is a symmetric, non negative matrix for each $r\in \R^2$ and the function $r\rightarrow  \mathcal{Q}(r)$ is Lipschitz with linear growth.\\
$\bullet$   Since for each $i\in \{1,2\},\ \sum_{j=1}^2 \partial_jA_{i,j}(r)=0$, if $\widetilde{W}$ and $\widetilde{\widetilde{W}}$     are   2D    independent Brownian motions, then the stochastic differential equation (SDE)   \begin{align}\label{FP_SDE-limit}
\begin{cases}
d\overline{X}_t&=u_L(t,\overline{X}_t)dt;\quad    \overline{X}_0=x\\
d\overline{R}_t&=(\nabla u_L(t,\overline{X}_t)\overline{R}_t-\dfrac{1}{\beta}\overline{R}_t)dt+\mathcal{Q}(\overline{R}_t)d\widetilde{\widetilde{W}}_t+\sqrt{2}\sigma    d\widetilde{W}_t;\quad   
\overline{R}_0=r,
\end{cases}
\end{align}
has \eqref{Limit-FP-eqn} as Fokker-Planck equation. \\
$\bullet$    Define  $\overline{\mathbf{T}}=\E_r(\overline{R}_t\otimes \overline{R}_t),$   a   standard    computations    gives that   $\overline{\mathbf{T}}$ satisfies
    \begin{align}\label{II}
\begin{cases}
&\partial_{t}\overline{\mathbf{T}}+u_L\cdot\nabla \overline{\mathbf{T}}   =(\nabla u_L) \overline{\mathbf{T}}+\overline{\mathbf{T}}(\nabla u_L)^t -\dfrac{2}{\beta}(\overline{\mathbf{T}}-kT\mathbf{I}) +\E A(\overline{R}_t)\\
&\quad \mathbf{T}|_{t=0} =\mathbf{T}_0,
\end{cases}
\end{align}
which   means   that    $\mathbf{T}^N$  solution    to  \eqref{macro-intro} converges   weakly  to  $\overline{\mathbf{T}}$ in  appropriate (weak)  sense and    the turbulent   velocity    generates the term    $\E A(\overline{R}_t)$, as  $N\to   +\infty$,  at  the macroscopic level.

Let us briefly interpret the previous results. The additional term $\mathcal{Q}(\overline{R}_t)d\widetilde{\widetilde{W}}_t$ in the stochastic equation contributes to a higher dispersion of the values of $\overline{R}_t$, increasing the variance; and not in a "Gaussian" way, since the diffusion matrix depends on $\overline{R}_t$ and its quadratic form evaluated in the $\overline{R}_t$ direction increases with $\overline{R}_t$. This is coherent with the power law result for the stationary solution illustrated in the next subsection.
Similarly, the additional term $\E_r A(\overline{R}_t)$ in the equation for the tensor $\E_r(\overline{R}_t\otimes \overline{R}_t)$ is non-negative definite, mostly positive definite, hence it contributes to a higher value of $\partial_{t}\overline{\mathbf{T}}$ hence to an increase of $\overline{R}_t$, coherent with the picture above: turbulent stretching statistically increases polymer length.
\end{remark} \begin{remark}\label{rmq-FENE}
 It is worth mentioning that we  consider the Hookean model. If we consider  the "Finitely Extensible Nonlinear Elastic (FENE)" model, namely replace $\frac{1}{\beta}R_t$ in  \eqref{Intro1} by 
 $\frac{1}{\beta}\frac{R_t}{1-\vert R_t\vert^2/b},$
    where $b$ is a  parameter denote the maximal  length of the polymer. Then, the analysis must be considered in a bounded domain with respect to $r$-variable with the appropriate boundary conditions, see e.g. \cite{Masmoudi2013}. It is interesting to extend our result to the FENE case. It seems that the scaling limit equation should have similar features  as the current one. However, since it requires significant changes,  this will be considered in a future work.
  \end{remark}
\subsection{Stationary solutions and  power-law tail}\label{Section-stationry-power-law}
The aim of this section is to discuss "formally" that  the limit PDE \eqref{equ-final}  exhibits interesting features of the power-law tail predicted in the physical literature, see for example \cite{Balk}. More rigorous results and discussions will be the subject of a future work.
 We consider the limit equation \eqref{equ-final}  in the simplest case $u_L=0$ and look for stationary solutions $\overline{f}(r)$, independent of $x$.   We  look for rotation-invariant and  non-negative solutions $\overline{f}\left(  r\right)    =g\left(  \left\vert r\right\vert \right)$, so that $\overline{f}$ should satisfy
\[
k_{T}\operatorname{div}_r\left( \dfrac{1}{2} \left(  3\left\vert r\right\vert ^{2}%
Id-2r\otimes r\right)  \nabla \overline{f}\left(  r\right)  +\frac{\sigma^{2}}{ k_{T}}\nabla
\overline{f}\left(  r\right)+ \frac{r}{\beta k_{T}}\overline{f}\left(
r\right) \right)   =0.
\]

Then, since $\nabla \overline{f}\left(  r\right)=\frac{r}{\vert r\vert}g^\prime\left(
 \left\vert r\right\vert \right)$ and $k_{T}>0$ we get

\begin{align}\label{eqn-station}
 \operatorname{div}_r\left( (\dfrac{\left\vert r\right\vert ^{2}}{2} +\frac{\sigma^{2}}{ k_{T}})
\nabla \overline{f}\left(  r\right)  + \frac{r}{\beta k_{T}}\overline{f}\left(
r\right) \right)   =0.   
\end{align}

Set $\frac{1}{\beta k_{T}}=\alpha$ and $p(r)=\frac{\left\vert r\right\vert ^{2}}{2} +\frac{\sigma^{2}}{ k_{T}}
.$ Then we can write   \eqref{eqn-station} as follows
\begin{align}\label{eqn-sta-den}
 \operatorname{div}_r\left( p^{1-\alpha}(r)\nabla [ p^{\alpha}(r)\overline{f}\left(  r\right)] \right)   =0.  
\end{align}
If $\overline{f}$ is smooth function, we can multiply \eqref{eqn-sta-den} by $p^{\alpha}(r)\overline{f}$ and integrate over $\mathbb{R}^2$ 
to get 
\begin{align*}
 \int_{\mathbb{R}^2}p^{1-\alpha}(r)\vert\nabla ( p^{\alpha}(r)\overline{f}\left(  r\right))\vert^2 dr=0.  
\end{align*}
Since $\dfrac{\sigma^{2}}{ k_{T}}>0$, the last equality implies the existence of  $C\in \mathbb{R}$ such that 
$p^{\alpha}(r)\overline{f}\left(  r\right)=C.$
Thus, since we are looking for positive nontrivial solution, we obtain 
\begin{align*}
\overline{f}\left(  \left\vert r\right\vert\right)=Cp^{-\alpha}(r)=C\left(\dfrac{\left\vert r\right\vert ^{2}}{2} +\dfrac{\sigma^{2}}{ k_{T}}\right)^{-\frac{1}{\beta k_{T}}}, C>0.
\end{align*}
In relation to physical literature, e.g. \cite[eqn.(11)]{Balk}. Let us just write the last formula only in $\vert r\vert$, namely
\begin{align}\label{formula-comment}
    \overline{f}\left(  \left\vert r\right\vert\right) \sim\left\vert r\right\vert^{-\frac{2}{\beta k_{T}}}= \left\vert r\right\vert^{-(\frac{2}{\beta k_{T}}-1)-1}=\left\vert r\right\vert^{-\gamma-1},\quad \gamma=\frac{2}{\beta k_{T}}-1.
\end{align}
Notice that  the power $\gamma$  depends on the interpretable constants, the relaxation time  of the polymer $\beta$ and the number $k_{T}$ proportional to the square-intensity of the turbulent eddies. Let us make some comments about the  formula \eqref{formula-comment}.
\begin{itemize}
\item \underline{The case $\gamma=\frac{2}{\beta k_{T}}-1<0 \Leftrightarrow 2<\beta k_{T}:$}  $\overline{f}$ is not integrable  and larger values of $\left\vert r\right\vert$ are more probable, which means that  the most of polymers have large size and the polymers are in \textit{stretched state}. This state holds either when   the relaxation time $\beta$ is large or  the intensity $k_{T}$ of the turbulent fluid is large.\\
\item  \underline{The case $\gamma=\frac{2}{\beta k_{T}}-1>0 \Leftrightarrow \beta k_{T}<2:$}  $\overline{f}$ is integrable and thus the normalized  $\overline{f}$  is PDF,  in other words,  most polymers have an equilibrium size and the Polymers are in \textit{coil state}.  the last expression of $\overline{f}$ confirms that  when either the relaxation time is small (namely the polymer is fast in recovering its equilibrium position) or  the intensity $k_{T}$ of the turbulent fluid is small then the polymers tends to be in \textit{coil state}.\\
\item \underline{The case $\gamma=\frac{2}{\beta k_{T}}-1=0 \Leftrightarrow \beta k_{T}=2:$}  The power $\gamma$ changes the sign and as discussed in  \cite{Balk}, this can be interpreted as the criterion for the coil-stretch transition.\\
\item As mentioned in Introduction,  we recall that the power tail is related to Lyapunov exponent $\lambda$ of the turbulent flow, which is not trivial to compute in general but one would like also to predict it. From \cite{Balk}, one   extracts the following link between $\lambda$ and $\gamma$:
\begin{enumerate}
    \item[i.]  If $\lambda$ increases then $\gamma$ decreases.$\qquad\qquad$  ii.  $\gamma$ changes its sign at $\lambda=\frac{1}{\beta}$.
    \item[iii.]  As $\lambda$ tends to zero, $\gamma$ tends to infinity.\\
\end{enumerate}
From (ii), we get that $\gamma$ changes the sign at $\frac{k_T}{2}=\frac{1}{\beta}=\lambda$. Moreover, as $\lambda$ is related to the turbulent flow and by combining (i) and (iii), we can  say  that the Lyapunov exponent $\lambda$ is our model of turbulent flow is $\lambda=\frac{K_T}{2}.$ However, we do not prove   such  equality "just heuristic!". Finally, in the case of  the point (iii), we get strong suppression of the tail, which   is quite natural since in a weak flow the molecules are only weakly stretched.
\end{itemize}

\section{About proofs and mathematical challenges}\label{section-outlines}
For the convenience of the reader, let us summarize the outline of the proof and the main technical difficulties we face in proving the main results. The details will  be  provided    in     \autoref{Section-exitence}  and \autoref{Section-diffusion}.   Note    that    \eqref{Ito-FP}  is  stochastic  Fokker  Planck,    has a hyperbolic
nature  with    respect to  the space   variable    $x$, it cannot regularize the initial condition.    Moreover,  its    well-posedness   is  not a   standard    result   and we  need    to  construct   a solution  in  an  appropriate sense,   see \autoref{Def-sol}.\\

The first   step    concerns    the construction    of  weak    solutions   (see    \autoref{Def-sol},   point   (3)) by    using   Galerkin    approximation scheme.  We proceed as  follows: we   construct   an  orthonormal basis  of $H$  (the  weight used to  give    sense to     the    terms   including  the coefficient   with $r$-variable,  $r\in   \R^2$).  It's    worth   to  mention we  consider    a general setting   to  prove   the existence   of  solution    and up  to  a cosmetic    modifications,   the same    result  follows by  using   other   weights e.g.    $(1+\vert   r\vert^2)^m,    m>1$ which corresponds to     $L^2_{r,m}(\T^2\times   \R^2)\hookrightarrow L^1(\T^2\times   \R^2)$ if  $  m>2$.    The first   problem  concerns    the proof   of  some   a   priori  estimates,  for that    we  add the extra   term    $\mathcal{Y}^m$ given   by \eqref{extra-ope-projection}     to  control    the term   coming  from    the  interaction between  the two part    of  the stochastic  integral,  in order   to  get the desired estimates.  Then, we    may pass    to  the limit   as  $m\to   +\infty$  and  construct   solution    to   our problem, after    showing    that    $\mathcal{Y}^m$ vanishes    as  $m\to   +\infty$ (we  exploit  that    $Q$ is  space-homogeneous   and has mirror  symmetry  property).   Consequently,    we  construct   weak    solution,   see \autoref{TH1}. \\

After   that,   we  face    the problem of  uniqueness  due to   lack    of  regularity  with    respect to  the space   variable    $x$ and the presence    of  noise,  which   motivates    the introduction    of  the    notion  \textit{"quasi-regular    weak    solution", }  which is characterized by the      point   (4) of  \autoref{Def-sol}   to  show uniqueness  in  this    subclass    of  solution,  see \autoref{TH2}   ( we  take    the advantage   of  the linearity   of  \eqref{Ito-FP}).  On   the other   hand,     we    should  use commutators estimates   and techniques  associated  with   hyperbolic   equations to prove in the beginning the uniqueness of solution to   the equation satisfied by some appropriate mean (\autoref{Def-sol},   point   $(4)$)  and we combine that with some density arguments in the $L^2$-space of  random  variables   to  deduce  the well-posedness   in  the sense   of  \autoref{Def-sol}. We   recall  that    this notion of uniqueness is weaker than the classical notion of    pathwise uniqueness,    see \cite[Rmk.    7]{FlaOli2018}.   Another key point   concerns    the uniform estimate    with    respect to  $N$,    which   follows directly    from    the construction    of  the solution    since  It\^o   formula  is    classical        at  the level   of  Galerkin    approximation    in   contrast with    infinite    dimension   where the solution    $f^N$ is    not smooth  and one  needs   to   be  careful    in  the application of  It\^o   formula.\\      

Finally, after obtaining some  estimates regarding    to $N$,    we  can pass    to  the limit   as  $N\to   +\infty$  (in weak    sense)  by  considering  the new regime when the noise covariance goes to zero but a suitable covariance built  on derivatives of the noise converges to a non zero limit  and we  obtain  a    degenerate  PDE with    respect to  the variable    $r$,  see \autoref{main-thm-2}   where    the stationary  solution    of  the limit   equation    \eqref{equ-final}    has a   power-law,   with    an  explicit    power   depending   on  the friction    parameter   $\beta$ and the number  $k_T$  (  related to  the turbulent   kinetic energy),  see \autoref{Section-stationry-power-law}.

\subsubsection{Formal computations    and estimates}
Let us  present some    formal  computations    related to   \eqref{Ito-FP}. Namely, 
we  compute   $\Vert  f^N\Vert^2$,    by  applying    It\^o   formula (formally), we  get
    
    \begin{align*}\displaystyle&\Vert  f^N(t)\Vert^2-\Vert  f_0\Vert^2=2\int_0^t(f^N(s),u_L\cdot \nabla_x f^N(s)+ (\nabla_x u_Lr-\dfrac{1}{\beta}r) \cdot\nabla_rf^N(s) )ds\notag\\	&-	2\sigma^2\displaystyle\int_0^t \Vert\nabla_rf^N(s)\Vert^2ds-2\alpha_N\int_0^t\Vert\nabla_xf^N(s)\Vert^2ds-\int_0^t (A_k^N(x,r)\nabla_rf^N(s),\nabla_rf^N(s)) ds\notag\\&+2\sum_{k\in K}\int_0^t[(f^N(s),\sigma_k^N.\nabla_xf^N(s))+(f^N(s),(\nabla \sigma_k^Nr).\nabla_rf^N(s))]dW^k(s)\notag
\\&+\sum_{k\in K}\int_0^t\Vert  \sigma_k^N.\nabla_xf^N(s)\Vert^2+\Vert    (\nabla\sigma_k^Nr).\nabla_rf^N(s)\Vert^2ds+2\sum_{k\in K}\int_0^t(\sigma_k^N.\nabla_xf^N(s),(\nabla\sigma_k^Nr).\nabla_rf^N(s))ds.
  \notag\end{align*}

  Since   $\Div_x(u_L)=\Div_r(\nabla_x u_Lr)=\Div_x(\sigma_k^N)=\Div_r(\nabla \sigma_k^Nr)=0,$    one has
  \begin{align*}
  (f^N,u_L\cdot \nabla_x f^N)=(f^N,\sigma_k^N.\nabla_xf^N)=0,\quad    (f^N,(\nabla_x u_Lr)\cdot\nabla_rf^N)=(f^N,(\nabla \sigma_k^Nr).\nabla_rf^N)=0.
  \end{align*}
  On    the other   hand,  since \begin{align*}\sum_{k\in K}\int_0^t\Vert  \sigma_k^N.\nabla_xf^N(s)\Vert^2ds&=\sum_{k\in K}\int_0^t(  \sigma_k^N.\nabla_xf^N(s),\sigma_k^N.\nabla_xf^N(s))ds\\&=\int_0^t(  \sum_{k\in K}\sigma_k^N(x)\otimes\sigma_k^N(x)\nabla_xf^N(s),\nabla_xf^N(s))ds\\&=\sum_{k\in K_{++}}(\theta_{\left\vert k\right\vert }^N)^2\int_0^t( \nabla_xf^N(s),\nabla_xf^N(s))ds=2\alpha_N\int_0^t\Vert\nabla_xf^N(s)\Vert^2ds.\end{align*}
 Hence
 $ -2\displaystyle\alpha_N\int_0^t\Vert\nabla_xf_m(s)\Vert^2ds +\sum_{k\in K}\int_0^t\Vert  \sigma_k^N.\nabla_xf_m(s)\Vert^2ds=0.$   Similarly,  we  have
   \begin{align*}
  \sum_{k\in K}\int_0^t\Vert   (\nabla\sigma_k^Nr).\nabla_rf^N(s)\Vert^2ds=\int_0^t (A_k^N\nabla_rf^N(s),\nabla_rf^N(s)) ds,
  \end{align*}
  where $A^N_k$ is  given   by 
  \begin{align}
       A_k^N&:=\sum_{k\in K}\left((\nabla \sigma_k^Nr) \otimes (\nabla \sigma_k^Nr)\right);\label{matrix-A}
  \end{align}   thus
  $-\displaystyle\int_0^t (A_k^N\nabla_rf^N(s),\nabla_rf^N(s)) ds+\sum_{k\in K}\int_0^t\Vert    (\nabla\sigma_k^Nr).\nabla_rf^N(s)\Vert^2ds= 0.$\\
  Summing   up, we  get
    \begin{align*}&\displaystyle\Vert  f^N(t)\Vert^2+2\sigma^2\displaystyle\int_0^t \Vert\nabla_rf^N(s)\Vert^2ds-\Vert  f_0\Vert^2 \\\leq&-\dfrac{2}{\beta}\int_0^t (f^N(s),r\cdot\nabla_rf^N(s) )ds+2\sum_{k\in K}\int_0^t(\sigma_k^N.\nabla_xf^N(s),(\nabla\sigma_k^Nr).\nabla_rf^N(s))ds.\end{align*}
On  the other   hand,   notice    that
\begin{align}\label{Cancelation-noise-covariance}
\sum_{k\in K}(\sigma_k^N.\nabla_xf^N(s),(\nabla\sigma_k^Nr).\nabla_rf^N(s))=-(\displaystyle\sum_{k\in K}	(\nabla \sigma_k^Nr).\nabla_r(\sigma_k^N.\nabla_xf^N(s)),f^N(s))=0,
\end{align}
        since   the covariance  matrix  $Q$   is space-homogeneous   and has mirror symmetry property,   see \autoref{lem1}.    By  using   \eqref{Cancelation-noise-covariance},   we  obtain

    \begin{align*}&\displaystyle\Vert  f^N(t)\Vert^2+2\sigma^2\displaystyle\int_0^t \Vert\nabla_rf^N(s)\Vert^2ds\leq    \Vert  f_0\Vert^2 -\dfrac{2}{\beta}\int_0^t (f^N(s),r\cdot\nabla_rf^N(s) )ds.\end{align*}
    Let us  compute the last    term    of  the RHS.    We  have
     \begin{align*} (f^N,r\cdot\nabla_rf^N )=\sum_{i=1}^2 \int_{\T^2\times  \R^2}\hspace{-0.5cm}f^N(x,r)r_i\cdot\partial_{r_i}f^N(x,r) drdx=\dfrac{1}{2}\sum_{i=1}^2 \int_{\T^2\times  \R^2}\hspace{-0.5cm}r_i\partial_{r_i}(f^N)^2(x,r) drdx=-\Vert  f^N\Vert^2.\end{align*}
 Consequently,  $\displaystyle\Vert  f^N(t)\Vert^2+2\sigma^2\displaystyle\int_0^t \Vert\nabla_rf^N(s)\Vert^2ds\leq    \Vert  f_0\Vert^2 +\dfrac{2}{\beta}\int_0^t \Vert  f^N(s)\Vert^2ds.$
     Grönwall  lemma   ensures the following:   P-a.s.
        \begin{align}\label{Estimate-formal}
        \forall t\geq   0:  \quad   \displaystyle\Vert  f^N(t)\Vert^2+2\sigma^2\displaystyle\int_0^t \Vert\nabla_rf^N(s)\Vert^2ds\leq    \Vert  f_0\Vert^2e^{\frac{2}{\beta}t}.\end{align}
            Note    that    the last    inequality  \eqref{Estimate-formal} is  sufficient, \textit{a   priori, }     to  construct  a   weak    solution       and also    ensure  the uniqueness, since   \eqref{Ito-FP}  is  a    linear  equation. Unfortunately,  \eqref{Estimate-formal} is not rigorous. Indeed, if  $f^N$   satisfies   \eqref{Estimate-formal} only    then    we  are not able    to  apply   directly    It\^o   formula to  \eqref{Ito-FP}  and we  need    either  to  approximate or  regularize \eqref{Ito-FP}  in  appropriate  way and prove the existence   of  $f^N$   and \eqref{Estimate-formal},    after   passing to  the limit   with    respect to  the approximation   or  regularization  parameters.  It turns   out that    it   is  not sufficient    to  consider  $f_0    \in L^2(\T\times\R^2)$ and  we  need    more    regular initial data,   namely  $f_0\in H$  to  construct   a   weak    solution,   see \autoref{Def-sol}.  Uniqueness is more delicate,    since   \eqref{Ito-FP}  has a hyperbolic character with respect to  $x$-variable    and we will    prove   uniqueness   in  particular  class   of  solution,   see \autoref{TH2}. For the convenience of the reader, let us explain the arguments that failed when we tried to obtain  \eqref{Estimate-formal} in  rigorous    way.    
\subsubsection{A  priori  estimates    via  Galerkin    approximation}
The first   step    concerns    the projection  of  \eqref{Ito-FP}  onto    a   finite  dimensional space,    see \eqref{approx-linear*SDE}.  Then,   we    apply   finite  dimensional It\^o   formula which   leads   to  the presence    of the  term
\begin{align}\label{new-term}
  \displaystyle\sum_{k\in K}(\sigma_k^N\cdot\nabla_xP_m[(\nabla\sigma_k^Nr).\nabla_rf_m],w_j))_H    \quad   (f_m=P_mf^N)
\end{align} instead of     $\displaystyle\sum_{k\in K}(\sigma_k^N.\nabla_xf^N(s),(\nabla\sigma_k^Nr).\nabla_rf^N(s))$   and   \eqref{Cancelation-noise-covariance}    is  not valid   anymore. We remedy this issue by subtraction this term \eqref{new-term} at the level of Galerkin approximation to get an appropriate estimate,   see   \autoref{Lemma9}  and we  need    to  show    that    $\displaystyle\lim_{m}(\sigma_k^N\cdot\nabla_xP_m[(\nabla\sigma_k^Nr).\nabla_rf_m],w_j))_H=0$  to  recover the original    problem \eqref{Ito-FP}. It's    worth   to  mention that    since   $r\in   \R^2$   and it  acts    as  coefficient  then    we  need    to  use initial data    $f_0\in H$   to  prove  that    the extra   term   \eqref{new-term} vanishes as  $m\to   +\infty$,  see  \eqref{limit-covarnaince-0}  and one  obtains  the desired estimates    of     \autoref{Lemma9}. We  refer   to  \autoref{Section-exitence}  for more    details.
\subsubsection{Uniqueness   via commutators and quasi-regular  weak  solution}
Once    the existence of    weak    solution    satisfying the  points  (1),(2) and (3) of  \autoref{Def-sol}  is  established, one seeks   to  prove   the uniqueness   of  this    class   of  solutions.  Since   \eqref{Ito-FP}  has an  hyperbolic  nature, one uses    the commutators estimates   in  order   to  prove   pathwise    uniqueness. Let us  explain why using   this    technique   does          not   give    uniqueness  of  this    class   of  solution,   for that    we  mention only    the terms   that    cause   problems. Let $\delta>0$   and $\varphi\in C^\infty_c(\T^2\times  \R^2)$, if  we  denote  $X=(x,r)\in \T^2\times  \R^2$, $\rho_{\delta}(X)=\rho_\delta(x)$  and  $\varphi_{\delta}:=\rho_{\delta}*\varphi$     then
   we  need    to    pass    to  the limit  as $\delta\to  0$   in  the   equality
   \begin{align*}&\dfrac{1}{2} \E\Vert [f^N(t)]_{\delta}\Vert^2+\E\int_0^t\langle [u_L(s)\cdot \nabla_xf^N(s)]_{\delta}+[\Div_r(\nabla u_L(s)r-\dfrac{1}{\beta}r)f^N(s)]_{\delta}, [f^N(s)]_{\delta}\rangle   ds\\	=&\E\int_0^t\langle \sigma^2[\Delta_rf^N(s)]_{\delta}+\alpha_N[\Delta_xf^N(s)]_{\delta}, [f^N(s)]_{\delta}\rangle   ds\\ &+ \dfrac{1}{2}\E\int_0^t\langle [\Div_r\sum_{k\in K}\left((\nabla \sigma_k^Nr) \otimes (\nabla \sigma_k^Nr)\right) \nabla_rf^N(s))]_{\delta},[f^N(s)]_{\delta}\rangle    ds\\&+\dfrac{1}{2}\sum_{k\in K}\E\int_0^t\Vert[\sigma_k^N.\nabla_xf^N(s)]_{\delta}+[(\nabla \sigma_k^Nr).\nabla_rf^N(s)]_{\delta}\Vert^2ds.\end{align*}
   We   can estimate    all the terms   in  convenient  way except  the terms
   \begin{align}
    &\E\int_0^t\langle \alpha_N[\Delta_xf^N(s)]_{\delta}, [f^N(s)]_{\delta}\rangle   ds+\dfrac{1}{2}\sum_{k\in K}\E\int_0^t\Vert[\sigma_k^N.\nabla_xf^N(s)]_{\delta}\Vert^2ds\label{Problem-terms}\\ &+\sum_{k\in K}\E\int_0^t(\sigma_k^N.\nabla_xf^N(s)]_{\delta},[(\nabla \sigma_k^Nr).\nabla_rf^N(s)]_{\delta})ds.\notag
\end{align}
The last    two terms   come    from    the term
$\dfrac{1}{2}\displaystyle\sum_{k\in K}\E\int_0^t\Vert[\sigma_k^N.\nabla_xf^N(s)]_{\delta}+[(\nabla \sigma_k^Nr).\nabla_rf^N(s)]_{\delta}\Vert^2ds.$
A  \textit{ priori,} we    expect   that    the sum of  the first   two terms    in  \eqref{Problem-terms}   vanishes    as    $\delta\to  0$    (balance    of  energy  of  It\^o   Stratonovich    corrector)  but in  order   to  get that,   we  need  more  regularity  of  $f^N,   $   namely  $\nabla_xf^N\in L^2(\T^2\times\R^2)$,   which   is  not our case. Concerning    the last    term    in  \eqref{Problem-terms},   it is expected to disappear due to the space  homogeneity  and the mirror   symmetry   of  the covariance  $Q$ but also    it  requires      $\nabla_xf^N\in L^2(\T^2\times\R^2)$   as well.  Consequently,  we  take    the advantage   of  the linearity   of  the equation    \eqref{Ito-FP}  and we  prove   uniqueness  in  weaker  sense,  see \autoref{TH2}.

\section{Existence   and uniqueness  of  solutions to   \eqref{Ito-FP}} \label{Section-exitence}
Our aim is  this    section is  to  prove   the existence   and uniqueness  of  solution    to  \eqref{Ito-FP}  in  the sense   of     \autoref{Def-sol},   namely  \autoref{TH1}   and \autoref{TH2}. We divide the proof into several steps.   First,  we  introduce   the Galerkin    approximation.  Then,   we  prove   some      estimates    in  the appropriate spaces. Next,    we  show   the existence   of  analytically  weak    solution    to  \eqref{Ito-FP}.    Finally,    we prove the uniqueness   of  the class   of  quasi-regular    weak    solutions    after   showing the existence   and uniqueness  of  the solution    to  an  appropriate mean    equation    associated  with    \eqref{Ito-FP}, see \autoref{prop-exisence-mean}    and \autoref{lem-uniq-Vn}.

\subsubsection{Galerkin    basis   and approximation}
  We need to construct an orthonormal basis of $H$ such that all the terms in our approximation scheme are meaningful.   For   that,    note   that    $V\underset{cont.}{\hookrightarrow}	H.$    On  the other   hand,   $(V,(\cdot,\cdot)_V)$ is   a separable Hilbert space. By  using  \cite[Lem. C.1.]{Brez2013}, there   exists  a   Hilbert space   $(U , (\cdot|\cdot)_U)$ such that $U\hookrightarrow V$ , $U$ is dense in $V$ and the embedding $U\hookrightarrow V$ is compact.    
		Thus	the embedding    $U\hookrightarrow	H$ is    also    compact and 		we		can   construct	an	orthonormal	basis	for	$H$	by	using	the	eigenvectors	of	the	compact	embeeding	operator.	More	precisely,	
	there exists an orthonormal basis
	$\{w_i\}_{i\in	\mathbb{N}}$	of	$H$	such	that	$w_i\in	U$	and	
		 satisfies
	\begin{align}\label{basis1.2}
	(v,w_i)_{U}=\lambda_i(v,w_i)_H, \quad \forall v \in U, \quad i \in \mathbb{N},
	\end{align}
	where the sequence $\{\lambda_i\}_{i\in\mathbb{N}}$ of the corresponding eigenvalues fulfils the  properties: $\lambda_i >0, \forall i\in \mathbb{N}.$  Note that $\{\widetilde{w}_i=\frac{1}{\sqrt{\lambda_i}}w_i\}$ is an orthonormal basis for $U$,    see     \cite[Subsect.  2.3]{Brez2013} for similar argument of construction.	
 Now,	let $m\in   \mathbb{N}^*$ and
		denote	by	$H_m=\text{span}\{w_1,\cdots,w_m\}$	and	the		operator	$P_m$	defined from	$U^\prime$	to	$H_m$	defined	by
	$P_m:U^\prime\to	H_m;\quad	u\mapsto	P_mu=\displaystyle\sum_{i=1}^m\langle	u,w_i\rangle_{U^\prime,U}	w_i.	$
	In	particular,	the	restriction	of	$P_m$	to	$H$,	denoted	by	the	same	way,	is	the	$(\cdot,\cdot)$-orthogonal	projection	from	$H$	to	$H_m$	and	given	by
	\begin{align}\label{def-Projection}
	P_m:H\to	H_m;\quad	u\mapsto	P_mu=\displaystyle\sum_{i=1}^m 
(	u,w_i)_H	w_i.		\end{align}
	
	We 	notice	that	$\Vert	P_mu\Vert_H\leq\Vert	u\Vert_H$, $\forall u\in H$,
	then 	$\Vert	P_m\Vert_{L(H,H)}\leq	1$.
	\begin{remark}\label{rmq-proj-U}
It	is	worth	to	mention	that	the	restriction	of	$P_m$	to	$U$	is	also	an	orthogonal	projection,	thanks	to	\eqref{basis1.2}	and	thus	$\Vert	P_m\Vert_{L(U,U)}\leq	1$.
	\end{remark}
We	have	the	following	continuous	embedding
$	U\hookrightarrow V	\hookrightarrow H\hookrightarrow L^2(\T^2\times	\R^2).$
Since	$U$	is	dense	subset	of		$L^2(\T^2\times	\R^2)$,	we	can	consider	the	following	compact	Lions-Gelfand	triple,	namely
\begin{align*}
	U\underset{dense}{\hookrightarrow}	L^2(\T^2\times	\R^2)\equiv		L^2(\T^2\times	\R^2)\hookrightarrow	U^\prime.
\end{align*}
To simplify the notation, the duality between	$U$	and	$U^\prime$	will	be	denoted	$\langle	\cdot,\cdot\rangle$	instead	of		$\langle	\cdot,\cdot\rangle_{U^\prime,U}$.	Thus,	we	have	the	following	equality
\begin{align}\label{Lions-Gelfand-triple}
	\langle	f,u\rangle=(f,u),	\quad	\forall	f\in	L^2(\T^2\times	\R^2),	\forall	u\in	U.
\end{align} 
 We use   Faedo-Galerkin method,      we introduce the approximation $f_m(t)=\displaystyle\sum_{j=1}^mg_{mj}(t)w_j$ 
        and set $f_m(0)=P_mf_0\in H_m$.		We	consider	the	following	finite	dimensional	SDE
        
\begin{align}&(f_m(t),w_j)_H-(f_{0},w_j)_H-\int_0^t(f_m(s),u_L\cdot \nabla_x w_j)_Hds-\int_0^t(f_m(s),\mathcal{Y}^mw_j)_Hds; \quad 1\leq j \leq m\label{approx-linear*SDE}\\=&-\int_0^t((\nabla_x u_Lr)\cdot\nabla_rf_m(s),w_j)_H+ \dfrac{2}{\beta}(f_m(s) ,w_j )_H+\dfrac{1}{\beta}(r\cdot\nabla_rf_m(s) ,w_j )_Hds\notag\\&-	\int_0^t[\sigma^2 (\nabla_rf_m(s),\nabla_rw_j)_H+2\sigma^2 (\nabla_rf_m(s),rw_j)+\alpha_N(\nabla_xf_m(s),\nabla_xw_j)_H]ds\notag\\
&- \dfrac{1}{2}\int_0^t (A_k^N\nabla_rf_m(s),\nabla_rw_j)_H ds-\int_0^t (A_k^N\nabla_rf_m(s),rw_j) ds,\notag\\&+\sum_{k\in K}\int_0^t((f_m(s),\sigma_k^N.\nabla_xw_j)_H-((\nabla \sigma_k^Nr)\cdot\nabla_rf_m(s),w_j)_H)dW^k(s),
\notag\end{align}
 where  
 \begin{align}
 \mathcal{Y}^m\phi&:=\sum_{k\in K}P_m\left[(\nabla\sigma_k^Nr).\nabla_r(P_m(\sigma_k^N.\nabla_x\phi))+(2\dfrac{(\nabla\sigma_k^Nr)\cdot	r}{1+\vert	r\vert^2}P_m(\sigma_k^N.\nabla_x\phi)\right], \forall\phi\in U.\label{extra-ope-projection}
 \end{align}

Note	that	\eqref{approx-linear*SDE}   is   a   linear system  of  SDE, by  classical result    (see    e.g.    \cite[Chapter   V]{Oksendal2013}),   we  get the existence and uniqueness of $\mathcal{F}_t$-adapted   solution    $f_m\in C([0,T],L^2(\Omega;H_m)).$
\begin{remark}\label{RMQ-proj}	Note	that
	\begin{align*}
	(f_m,\mathcal{Y}^mw_j)_H&=\sum_{k\in K}(f_m,P_m\left[(\nabla\sigma_k^Nr).\nabla_r(P_m(\sigma_k^N.\nabla_xw_j))+2\dfrac{(\nabla\sigma_k^Nr)\cdot	r}{1+\vert	r\vert^2}P_m(\sigma_k^N.\nabla_xw_j)\right])_H\\
	&=\sum_{k\in K}(f_m,(\nabla\sigma_k^Nr).\nabla_r(P_m(\sigma_k^N.\nabla_xw_j)))_H+2(f_m,\dfrac{(\nabla\sigma_k^Nr)\cdot	r}{1+\vert	r\vert^2}P_m(\sigma_k^N.\nabla_xw_j))_H\\
	&=\sum_{k\in K}(\sigma_k^N\cdot\nabla_xP_m[(\nabla\sigma_k^Nr).\nabla_rf_m],w_j))_H.
	\end{align*}
\end{remark}

\subsection{A   priori  estimates} \label{Subsection-estimate-fm}
    We  apply   the  finite  dimensional It\^o formula to  the process $f_m$    to  get
\begin{align}\label{Ito-form-1-}&\displaystyle\Vert  f_m(t)\Vert^2_H-\Vert  P_mf_0\Vert_H^2 =2\int_0^t(f_m(s),u_L\cdot \nabla_x f_m(s))_Hds\\&-2\int_0^t((\nabla_x u_Lr)\cdot\nabla_rf_m(s),f_m(s))_H+ \dfrac{2}{\beta}\Vert  f_m(s)\Vert^2_H+\dfrac{1}{\beta}(r\cdot\nabla_rf_m(s) ,f_m(s) )_Hds\notag\\&\hspace*{-1.5cm}+2\int_0^t(f_m(s),\mathcal{Y}^mf_m)_Hds-	2\int_0^t[\sigma^2 \Vert\nabla_rf_m(s)\Vert_H^2+2\sigma^2 (\nabla_rf_m(s),rf_m(s))+\alpha_N\Vert\nabla_xf_m(s)\Vert_H^2]ds\notag\\
&-\int_0^t (A_k^N\nabla_rf_m(s),\nabla_rf_m(s))_H ds-2\int_0^t (A_k^N\nabla_rf_m(s),rf_m(s)) ds\notag\\&+2\sum_{k\in K}\int_0^t((f_m(s),\sigma_k^N.\nabla_xf_m(s))_H-((\nabla \sigma_k^Nr)\cdot\nabla_rf_m(s),f_m(s))_H)dW^k(s)
\notag\\&+\sum_{k\in K}\int_0^t\Vert  P_m\sigma_k^N.\nabla_xf_m(s)\Vert_H^2+\Vert   P_m (\nabla\sigma_k^Nr).\nabla_rf_m(s)\Vert^2_Hds\notag
\\&+2\sum_{j=1}^m\sum_{k\in K}\int_0^t(\sigma_k^N.\nabla_xf_m(s),e_j)_H((\nabla\sigma_k^Nr).\nabla_rf_m(s),e_j)_Hds.\notag
\end{align}
  Since   $\Div_x(u_L)=\Div_x(\sigma_k^N)=0,$    one has
 $
  (f_m,u_L\cdot \nabla_x f_m)_H=(f_m,\sigma_k^N.\nabla_xf_m)_H=0.
  $
  On    the other   hand,  note	that \begin{align*}\sum_{k\in K}\int_0^t\Vert  \sigma_k^N.\nabla_xf_m(s)\Vert^2_Hds=\sum_{k\in K_{++}}(\theta_{\left\vert k\right\vert }^N)^2\int_0^t( \nabla_xf_m(s),\nabla_xf_m(s))_Hds=2\alpha_N\int_0^t\Vert\nabla_xf_m(s)\Vert^2_Hds,\end{align*}
  where we  used    that    $(  \sigma_k^N.\nabla_xf_m,\sigma_k^N.\nabla_xf_m)_H=(  \sigma_k^N\otimes\sigma_k^N\nabla_xf_m,\nabla_xf_m)_H.$
Therefore
   \begin{align*}
  -2\alpha_N\int_0^t\Vert\nabla_xf_m(s)\Vert^2_Hds &+\sum_{k\in K}\int_0^t\Vert  P_m\sigma_k^N.\nabla_xf_m(s)\Vert^2_H \leq 0,\end{align*}
  since $\Vert  P_m\sigma_k^N.\nabla_xf_m(s)\Vert^2_H\leq \Vert  \sigma_k^N.\nabla_xf_m(s)\Vert^2_H$
  due   to   that    $P_m$   is  	the	$(\cdot,\cdot)_H$-orthogonal	projection	from	$H$	to	$H_m$.   Similarly,  we  have
   \begin{align*}
  \sum_{k\in K}\int_0^t\Vert   (\nabla\sigma_k^Nr).\nabla_rf_m(s)\Vert^2_Hds=\int_0^t (A_k^N\nabla_rf_m(s),\nabla_rf_m(s))_H ds,
  \end{align*}
  where $A^N_k$ is  given   by  \eqref{matrix-A},   thus
  \begin{align*}
-\int_0^t (A_k^N\nabla_rf_m(s),\nabla_rf_m(s))_H ds+\sum_{k\in K}\int_0^t\Vert   P_m (\nabla\sigma_k^Nr).\nabla_rf_m(s)\Vert^2_Hds\leq 0.\end{align*}
On  the other   hand,   note    that
\begin{align*}
&\sum_{j=1}^m\sum_{k\in K}\int_0^t(\sigma_k^N.\nabla_xf_m(s),e_j)_H((\nabla\sigma_k^Nr).\nabla_rf_m(s),e_j)_Hds\\&=\sum_{k\in K}\int_0^t(\sum_{j=1}^m(\sigma_k^N.\nabla_xf_m(s),e_j)_He_j,(\nabla\sigma_k^Nr).\nabla_rf_m(s))_Hds\\&=-\sum_{k\in K}\int_0^t(f_m(s),P_m\left[(\nabla\sigma_k^Nr).\nabla_r(P_m(\sigma_k^N.\nabla_xf_m(s))+2\dfrac{(\nabla\sigma_k^Nr)\cdot	r}{1+\vert	r\vert^2}P_m(\sigma_k^N.\nabla_xf_m(s))\right])_Hds\\&=-\int_0^t(f_m(s),\mathcal{Y}^mf_m(s))ds.
\end{align*}
  Summing   up, we  get
    \begin{align}\label{estimate-stage-1}&\displaystyle\Vert  f_m(t)\Vert^2_H-\Vert  P_mf_0\Vert_H^2+2\sigma^2\int_0^t \Vert\nabla_rf_m(s)\Vert_H^2ds =-2\int_0^t((\nabla_x u_Lr)\cdot\nabla_rf_m(s),f_m(s))_H\\&+ \dfrac{2}{\beta}\Vert  f_m(s)\Vert^2_H+\dfrac{1}{\beta}(r\cdot\nabla_rf_m(s) ,f_m(s) )_Hds-	4\sigma^2\int_0^t (\nabla_rf_m(s),rf_m(s))ds\notag\\&-2\int_0^t (A_k^N\nabla_rf_m(s),rf_m(s)) ds-2\sum_{k\in K}\int_0^t((\nabla \sigma_k^Nr)\cdot\nabla_rf_m(s),f_m(s))_HdW^k(s).\notag
    \end{align}
    Let	us	estimate	the	terms	of	the	right	hand	side    in  \eqref{estimate-stage-1}.	First,	note	that
   \begin{align}
\int_0^t((\nabla_x u_Lr)\cdot\nabla_rf_m(s),f_m(s))_Hds&=-\int_0^t(f_m(s),f_m(s)(\nabla_x u_Lr)\cdot	r)ds\notag\\&\leq	\Vert	\nabla_x u_L\Vert_{\infty}\int_0^t\Vert	f_m(s)\Vert_H^2ds.\label{reg-p-0}
\end{align}
A   standard    integration by  parts   gives   	
 $$
\displaystyle\int_0^t\dfrac{1}{\beta}(r\cdot\nabla_rf_m(s) ,f_m(s) )_Hds=-\dfrac{1}{\beta}\int_0^t(\Vert	f_m(s)\Vert_H^2 +(\vert	r\vert^2f_m(s) ,f_m(s) )ds,$$
and $
\displaystyle\vert\int_0^t\dfrac{1}{\beta}(r\cdot\nabla_rf_m(s) ,f_m(s) )_Hds\vert\leq\dfrac{2}{\beta}\int_0^t\Vert	f_m(s)\Vert_H^2ds.
$
Let	us		estimate	
$\displaystyle\int_0^t (\nabla_rf_m(s),rf_m(s))ds$.	Note  that
$
\displaystyle\int_0^t (\nabla_rf_m(s),rf_m(s))ds=-\int_0^t (f_m(s),f_m(s))ds,
$
thus
\begin{equation}\label{reg-p-2--}
-4\sigma^2\int_0^t (\nabla_rf_m(s),rf_m(s))ds=4\sigma^2\int_0^t \Vert	f_m(s)\Vert^2ds.
\end{equation}
Next,	the	term	$-2\displaystyle\int_0^t (A_k^N\nabla_rf_m(s),rf_m(s)) ds$.	Since	$\Div_r[\nabla \sigma_k^Nr]=0$,	we	get	
\begin{align*}
&\int_0^t (A_k^N\nabla_rf_m(s),rf_m(s)) ds=\sum_{k\in K}\int_0^t (  (\nabla \sigma_k^Nr)\cdot\nabla_rf_m(s),(\nabla \sigma_k^Nr)\cdot	rf_m(s)) ds\\
&=-\sum_{k\in K}\int_0^t\Big((  (\nabla \sigma_k^Nr)\cdot\nabla_rf_m(s),(\nabla \sigma_k^Nr)\cdot	rf_m(s)) + ((\nabla \sigma_k^Nr)f_m(s),f_m(s)\nabla_r[(\nabla \sigma_k^Nr)\cdot	r]) \Big)ds,
\end{align*}
and
$
2\displaystyle\int_0^t (A_k^N\nabla_rf_m(s),rf_m(s)) ds=-\sum_{k\in K}\int_0^t ((\nabla \sigma_k^Nr)f_m(s),f_m(s)\nabla_r[(\nabla \sigma_k^Nr)\cdot	r]) ds.$
On	the	other	hand,	note	that
\begin{align*}
&\sum_{k\in K}\int_0^t ((\nabla \sigma_k^Nr)f_m(s),f_m(s)\nabla_r[(\nabla \sigma_k^Nr)\cdot	r]) ds\\
&=\sum_{k\in K_+}\int_0^t(\theta_{\vert	k\vert}^N)^2\vert	k\vert^2(\dfrac{k}{\vert	k\vert}\cdot	r\dfrac{k^\perp}{\vert	k\vert}\sin k\cdot xf_m(s),f_m(s)\sin k\cdot x\nabla_r[\dfrac{k}{\vert	k\vert}\cdot	r\dfrac{k^\perp}{\vert	k\vert}\cdot	r]	)ds\\&\quad+\sum_{k\in K_-}\int_0^t(\theta_{\vert	k\vert}^N)^2\vert	k\vert^2(\dfrac{k}{\vert	k\vert}\cdot	r\dfrac{k^\perp}{\vert	k\vert}\cos k\cdot xf_m(s),f_m(s)\cos k\cdot x\nabla_r[\dfrac{k}{\vert	k\vert}\cdot	r\dfrac{k^\perp}{\vert	k\vert}\cdot	r]	)ds\\
&=\sum_{k\in K_+}\int_0^t(\theta_{\vert	k\vert}^N)^2\vert	k\vert^2(\dfrac{k}{\vert	k\vert}\cdot	r\dfrac{k^\perp}{\vert	k\vert}f_m(s),f_m(s)\nabla_r[\dfrac{k}{\vert	k\vert}\cdot	r\dfrac{k^\perp}{\vert	k\vert}\cdot	r]	)ds\\
&=\sum_{k\in K_+}\int_0^t(\theta_{\vert	k\vert}^N)^2\vert	k\vert^2\left((\dfrac{k}{\vert	k\vert}\cdot	r\dfrac{k^\perp\cdot	k}{\vert	k\vert^2}f_m(s),f_m(s)\dfrac{k^\perp}{\vert	k\vert}\cdot	r)+(\dfrac{k}{\vert	k\vert}\cdot	r\dfrac{k^\perp\cdot	k^\perp}{\vert	k\vert^2}f_m(s),f_m(s)\dfrac{k}{\vert	k\vert}\cdot	r)\right)ds\\
&\leq	2\sum_{k\in K_+}\int_0^t(\theta_{\vert	k\vert}^N)^2\vert	k\vert^2	\Vert	\vert	r\vert	f_m(s)\Vert^2	ds\leq	\sum_{\substack {k\in		K_+\\	N\leq		\vert	k\vert		\leq	2N}}\dfrac{2}{\vert	k\vert^2}\int_0^t	\Vert		f_m(s)\Vert_H^2	ds\leq	C\int_0^t	\Vert		f_m(s)\Vert_H^2	ds,
\end{align*}
thus
\begin{align}\label{reg-p-3--}
&\vert\sum_{k\in K}\int_0^t ((\nabla \sigma_k^Nr)f_m(s),f_m(s)\nabla_r[(\nabla \sigma_k^Nr)\cdot	r]) ds\vert\leq	C\int_0^t	\Vert		f_m(s)\Vert_H^2	ds,
\end{align}

where	$C>0$	independent	of	$N$	satisfies	$\sum_{\substack {k\in		K_+\\	N\leq		\vert	k\vert		\leq	2N}}\dfrac{2}{\vert	k\vert^2}\leq	C.$    Concerning  the stochastic	integral,
note	that   $
((\nabla \sigma_k^Nr)\cdot\nabla_rf_m(s),f_m(s))_H
=-(f_m(s),f_m(s)(\nabla \sigma_k^Nr)\cdot	r)).
$
Recall	that	
\begin{align}\label{nabla-sigma}
(\nabla \sigma_k^Nr)(x)= - k\cdot	r\theta_{\vert	k\vert}^N\frac{k^{\perp}}{\left\vert
	k\right\vert }\sin k\cdot x,\quad k\in K_{+};
(\nabla \sigma_k^Nr)(x)     =k\cdot	r\theta_{\vert	k\vert}^N\frac{k^{\perp}}{\left\vert
	k\right\vert }\cos k\cdot x,\quad k\in K_{-}
\end{align}
and
\begin{align*}
\sum_{k\in K}\int_0^t(-((\nabla \sigma_k^Nr)\cdot\nabla_rf_m(s),f_m(s))_H)dW^k(s)=2\sum_{k\in K}\int_0^t((f_m(s),f_m(s)(\nabla \sigma_k^Nr)\cdot	r)dW^k(s).
\end{align*}
By	using	 the Burkholder-Davis-Gundy inequality,	we get
\begin{align*}2\E\sup_{q\in[0,t]}\vert\int_0^q\sum_{k\in	K}(f_m(s),f_m(s)(\nabla \sigma_k^Nr)\cdot	r)dW^k(s)\vert&\leq	2\E[\int_0^t\sum_{k\in	K}\vert(f_m(s),f_m(s)(\nabla \sigma_k^Nr))\cdot	r)\vert^2ds]^{1/2},
\end{align*}
On	the	other	hand,	we	have
\begin{align*}
&\int_0^t\sum_{k\in K}\vert(f_m(s),f_m(s)(\nabla \sigma_k^Nr))\cdot	r)\vert^2ds\\
&=\sum_{k\in K_+}(\theta_{\vert	k\vert}^N)^2\vert	k\vert^2\int_0^t(-\dfrac{k}{\vert	k\vert}\cdot	\dfrac{r}{(1+\vert	r\vert^2)}(1+\vert	r\vert^2)^{1/2}f_m(s), r\cdot\frac{k^{\perp}}{\left\vert
	k\right\vert }\sin k\cdot x	(1+\vert	r\vert^2)^{1/2}f_m)^2ds\\&+\sum_{k\in K_-}(\theta_{\vert	k\vert}^N)^2\vert	k\vert^2\int_0^t(\dfrac{k}{\vert	k\vert}\cdot	\dfrac{r}{(1+\vert	r\vert^2)}(1+\vert	r\vert^2)^{1/2}f_m(s)	, r\cdot\frac{k^{\perp}}{\left\vert
	k\right\vert }\cos k\cdot x	(1+\vert	r\vert^2)^{1/2}f_m(s)	)^2ds\\
&\leq	\sum_{k\in K}(\theta_{\vert	k\vert}^N)^2\vert	k\vert^2\int_0^t\Vert(1+\vert	r\vert^2)^{1/2}f_m(s)\Vert^2 \Vert	(1+\vert	r\vert^2)^{1/2}f_m\Vert^2ds\\
&\leq	\sum_{k\in K}(\theta_{\vert	k\vert}^N)^2\vert	k\vert^2\int_0^t\Vert(1+\vert	r\vert^2)^{1/2}f_m(s)\Vert^4 ds=\sum_{\substack {k\in		K\\	N\leq		\vert	k\vert		\leq	2N}}(\theta_{\vert	k\vert}^N)^2\vert	k\vert^2\int_0^t\Vert	f_m(s)\Vert_H^2 \Vert	f_m(s)\Vert_H^2ds\\
&\leq	\sum_{\substack {k\in		K\\	N\leq		\vert	k\vert		\leq	2N}}\dfrac{1}{\vert	k\vert^2}\int_0^t\Vert	f_m(s)\Vert_H^2 \Vert	f_m(s)\Vert_H^2ds\leq	C\int_0^t\Vert	f_m(s)\Vert_H^2 \Vert	f_m(s)\Vert_H^2ds\\
&\leq	C\sup_{q\in[0,t]}\Vert	f_m(q)\Vert_H^2\int_0^t \Vert	f_m(s)\Vert_H^2ds,
\end{align*}
which   gives
\begin{align}\label{reg-p-5}
\int_0^t\sum_{k\in K}\vert(f_m(s),f_m(s)(\nabla \sigma_k^Nr))\cdot	r)\vert^2ds
\leq	C\sup_{q\in[0,t]}\Vert	f_m(q)\Vert_H^2\int_0^t \Vert	f_m(s)\Vert_H^2ds,
\end{align}
where	we	used	$\displaystyle\sum_{\substack {k\in		K\\	N\leq		\vert	k\vert		\leq	2N}}\dfrac{1}{\vert	k\vert^2}\leq	C,$  with	$C>0$	independent	of	$N$.	Therefore
\begin{align*}2\E\sup_{q\in[0,t]}\vert\int_0^q\sum_{k\in	K}(f_m(s),f_m(s)(\nabla \sigma_k^Nr)\cdot	r)dW^k(s)\vert&\leq	2\E[\int_0^t\sum_{k\in	K}\vert(f_m(s),f_m(s)(\nabla \sigma_k^Nr))\cdot	r)\vert^2ds]^{1/2}\\
&\leq	2\E[C\sup_{q\in[0,t]}\Vert	f_m(q)\Vert_H^2\int_0^t \Vert	f_m(s)\Vert_H^2ds]^{1/2}\\
&\leq	\epsilon	\E\sup_{q\in[0,t]}\Vert	f_m(q)\Vert_H^2+	\dfrac{C}{\epsilon}\E\int_0^t \Vert	f_m(s)\Vert_H^2ds,
\end{align*}
for	any	$\epsilon>0$	(to be chosen later).		By	using	\eqref{estimate-stage-1}	and	gathering	the	previous	estimates,	we	get
 \begin{align*}&(1-\epsilon)\displaystyle\E\sup_{q\in[0,t]}\Vert  f_m(q)\Vert^2_H+2\sigma^2\E\int_0^t \Vert\nabla_rf_m(s)\Vert_H^2 \\&\leq	\Vert  P_mf_0\Vert_H^2+ [2\Vert	\nabla_x u_L\Vert_{\infty}+\dfrac{8}{\beta}+4\sigma^2+2C+	\dfrac{C}{\epsilon}]\E\int_0^t\Vert  f_m(s)\Vert^2_Hds.
\end{align*}
By	chosing	$\epsilon=\dfrac{1}{2}$	and	setting		\begin{align}\label{const-Gronwall}
\lambda=2[2\Vert	\nabla_x u_L\Vert_{\infty}+\dfrac{8}{\beta}+4\sigma^2+4C],
\end{align}	we	obtain
\begin{align*}&\displaystyle\E\sup_{q\in[0,t]}\Vert  f_m(q)\Vert^2_H+4\sigma^2\E\int_0^t \Vert\nabla_rf_m(s)\Vert_H^2 \leq	2\Vert  f_0\Vert_H^2+ \lambda\E\int_0^t\Vert  f_m(s)\Vert^2_Hds.
\end{align*}
Finally,	Grönwall	lemma	ensures
\begin{align*}\forall t\geq   0:\displaystyle\E\sup_{q\in[0,t]}\Vert  f_m(q)\Vert^2_H+4\sigma^2\E\int_0^t \Vert\nabla_rf_m(s)\Vert_H^2 \leq	2\Vert  f_0\Vert_H^2e^{\lambda	t}.
\end{align*}
  As a   conclusion, we  get
 \begin{lemma}\label{Lemma9}
 For every  $m\in   \mathbb{N}^*$, there   exists  a   unique   solution $f_m\in C([0,T],L^2(\Omega;H_m))$  to    \eqref{approx-linear*SDE}, which   is  adapted to  the filtration  and satisfy
\begin{align}\label{borne}
    \forall t\geq   0:\displaystyle\E\sup_{q\in[0,t]}\Vert  f_m(q)\Vert^2_H+4\sigma^2\E\int_0^t \Vert\nabla_rf_m(s)\Vert_H^2 \leq	2\Vert  f_0\Vert_H^2e^{\lambda	t},
 \end{align}
 where	$\lambda$	is	given	by	\eqref{const-Gronwall}.
\end{lemma}
  In  order   to  handle  the stochastic  integral    as  $N\to   +\infty$,   it  is  convenient   to  show    the following. 
\begin{lemma}\label{Lemma10}
    The    unique   solution $f_m\in C([0,T],L^2(\Omega;H_m))$  to    \eqref{approx-linear*SDE} satisfies
    \begin{align}\label{borne-p}
    \forall t\in[0,T] :\displaystyle\E\sup_{q\in[0,t]}\Vert  f_m(q)\Vert^{2p}_H+(2\sigma^2)^p\E(\int_0^t \Vert\nabla_rf_m(s)\Vert_H^2)^p \leq	\mathbf{C}\Vert  f_0\Vert_H^{2p}e^{\mathbf{C}t},  \forall 1<p<+\infty
 \end{align}
 where	$\mathbf{C}>0$	is    independent of  $N$ and $m$.
\end{lemma}

\begin{proof}
From    \eqref{estimate-stage-1},   by  using   \eqref{reg-p-0},   \eqref{reg-p-2--} and \eqref{reg-p-3--} we  get
  \begin{align*}\displaystyle\Vert  f_m(t)\Vert^2_H+2\sigma^2\int_0^t \Vert\nabla_rf_m(s)\Vert_H^2 ds
  &\leq \Vert  f_0\Vert_H^2+ \dfrac{\lambda}{2}\int_0^t\Vert  f_m(s)\Vert^2_Hds\\&+2\vert   \sum_{k\in K}\int_0^t(-((\nabla \sigma_k^Nr)\cdot\nabla_rf_m(s),f_m(s))_H)dW^k(s)\vert.
    \end{align*}    
   Let  $p>1$,   we have
    \begin{align*}\displaystyle\Vert  f_m(t)\Vert^{2p}_H+(2\sigma^2)^p(\int_0^t \Vert\nabla_rf_m(s)\Vert_H^2 ds)^p
  &\leq C_p[\Vert  f_0\Vert_H^{2p}+ [\dfrac{\lambda}{2}]^pt^{p-1}\int_0^t\Vert  f_m(s)\Vert^{2p}_Hds\\&+2^pC_p\vert   \sum_{k\in K}\int_0^t(-((\nabla \sigma_k^Nr)\cdot\nabla_rf_m(s),f_m(s))_H)dW^k(s)\vert^p.
    \end{align*}    
    Now,    by	using	  Burkholder-Davis-Gundy inequality    and \eqref{reg-p-5},	we get
\begin{align*}2^pC_p&\E\sup_{q\in[0,t]}\vert\int_0^q\sum_{k\in	K}(f_m(s),f_m(s)(\nabla \sigma_k^Nr)\cdot	r)dW^k(s)\vert^p\\&\leq	2^pC_p\E[\int_0^t\sum_{k\in	K}\vert(f_m(s),f_m(s)(\nabla \sigma_k^Nr))\cdot	r)\vert^2ds]^{p/2}\\
&\leq    2^pC_pC^{p/2}t^{(p-1)/2} \E[\sup_{q\in[0,t]}\Vert	f_m(q)\Vert_H^{2p}\int_0^t \Vert	f_m(s)\Vert_H^{2p}ds]^{1/2}\\&\leq	\epsilon	\E\sup_{q\in[0,t]}\Vert	f_m(q)\Vert_H^{2p}+	\dfrac{(2^pC_p)^2C^pt^{p-1}}{\epsilon}\E\int_0^t \Vert	f_m(s)\Vert_H^{2p}ds,
\end{align*}
for	any	$\epsilon>0$	(to be chosen later).	Consequently,    there   exists  $\mathbf{C}>0$  independent of  $m$ and $N$ such    that
  \begin{align*}(1-\epsilon)	\E\sup_{q\in[0,t]}\Vert	f_m(q)\Vert_H^{2p}+(2\sigma^2)^p\E(\int_0^t \Vert\nabla_rf_m(s)\Vert_H^2 ds)^p
  &\leq \mathbf{C}[\Vert  f_0\Vert_H^{2p}+ (\mathbf{C}+\dfrac{\mathbf{C}}{\epsilon})\E\int_0^t\Vert  f_m(s)\Vert^{2p}_Hds.    \end{align*}  
Hence, \eqref{borne-p}   follows  by    choosing $\epsilon=\dfrac{1}{2}$ and applying        Grönwall	lemma.
\end{proof}

\subsection{Existence of solution to \eqref{Ito-FP}} \label{Subsection-exist-proof}
We  have  the    following   equation of $f_m(t)=f_m(x,r,t)$:  \begin{align}
df_m(t)&+P_m\Div_x(u_Lf_m(t))dt+P_m\Div_r((\nabla u_Lr-\dfrac{1}{\beta}r)f_m(t)) dt\label{approx-eqn.}\\	&=\sum_{k\in K}P_m(\sigma_k^N\cdot\nabla_xP_m[(\nabla\sigma_k^Nr).\nabla_rf_m])dt+	\sigma^2P_m\Delta_rf_m(t)dt\notag\\&-\sum_{k\in K}P_m\sigma_k^N.\nabla_xf^N(t)dW^k -\sum_{k\in K}P_m(\nabla \sigma_k^Nr).\nabla_rf_m(t) dW^k\notag\\
&+\alpha_N	P_m\Delta_xf_m(t)dt+ \dfrac{1}{2}P_m\Div_r(\sum_{k\in K}\left((\nabla \sigma_k^Nr) \otimes (\nabla \sigma_k^Nr)\right) \nabla_rf_m(t))dt,\notag\\
f_m|_{t=0}&=P_mf_0.\notag
\end{align}
  
  From \eqref{approx-eqn.} we  get  
  for   $1\leq  j\leq   m:$
		\begin{align}&\int_{\mathbb{T}^2}\int_{\mathbb{R}^2}f_m(t) w_j drdx-\int_{\mathbb{T}^2}\int_{\mathbb{R}^2}P_mf_0 w_j drdx\label{eqn-app-key1}\\&-\int_0^t\int_{\mathbb{T}^2}\int_{\mathbb{R}^2}f_m(s)\left(u_L(s)\cdot \nabla_x w_j + (\nabla u_L(s)r-\dfrac{1}{\beta}r) \cdot\nabla_r w_j \right)dr dxds\notag\\
  &=-\sum_{k\in K}\int_0^t\int_{\mathbb{T}^2}\int_{\mathbb{R}^2}   \nabla_rf_m(s),(\nabla\sigma_k^Nr)P_m[\sigma_k^N\cdot\nabla_xw_i ]drdx  ds\notag\\
  	&-	\sigma^2\int_0^t\int_{\mathbb{T}^2}\int_{\mathbb{R}^2}\nabla_rf_m(s) \cdot\nabla_r w_j drdxds+\alpha_N\int_0^t\int_{\mathbb{T}^2}\int_{\mathbb{R}^2}f_m(s) \cdot\Delta_xw_j drdxds\notag\\&+\sum_{k\in K}\int_0^t\int_{\mathbb{T}^2}\int_{\mathbb{R}^2}[f_m(s) \sigma_k^N.\nabla_xw_j+f_m(s) (\nabla \sigma_k^Nr).\nabla_r w_j] drdxd    W^k(s)\notag\\
		&- \dfrac{1}{2}\int_0^t\int_{\mathbb{T}^2}\int_{\mathbb{R}^2}(\sum_{k\in K}\left((\nabla \sigma_k^Nr) \otimes (\nabla \sigma_k^Nr)\right) \nabla_rf_m(s))\cdot\nabla_r w_j drdxds.\notag\end{align}
  Let   us     pass    to  the limit   in  the last    equation.
  \subsubsection{$1^{st} step$}
By  using   \autoref{Lemma9},      we  are able    to  get the following   convergences  (up to a subsequence)
  \begin{align}
 	f_m \rightharpoonup \widetilde{f} &\text{		in	}  L^2(\Omega ;L^2([0, T];H)),\label{cv-appro-3}\\
 	 f_m \rightharpoonup^* \widetilde{f} &\text{		in	}  L^2_{w-*}(\Omega ;L^\infty([0, T];H))),\label{cv-appro-4}\\
   \nabla_rf_m \rightharpoonup \nabla_r\widetilde{f} &\text{		in	}  L^2(\Omega ;L^2([0, T];H)),\label{cv-appro-5}
 	\end{align}
   On   the other   hand,   from  \autoref{Lemma10},  we  obtain    also
      \begin{align}\label{borne-p-limit}
    \forall t\in[0,T] :\displaystyle\E\sup_{q\in[0,t]}\Vert  \widetilde{f}(q)\Vert^{2p}_H+(2\sigma^2)^p\E(\int_0^t \Vert\nabla_r\widetilde{f}(s)\Vert_H^2)^p \leq	\mathbf{C}\Vert  f_0\Vert_H^{2p}e^{\mathbf{C}t},  \forall 1<p<+\infty
 \end{align}
 where	$\mathbf{C}>0$	is    independent of  $N$ and $m$.
  	\subsubsection{$2^{nd} step$}
For  $\phi \in U$, note that    $\displaystyle\int_{\mathbb{T}^2}\int_{\mathbb{R}^2}f_m(x,r,t)\phi(x,r)dxdr$	is adapted with    respect to  $(\mathcal{F}_t)_t$    and recall  that    the space of 
adapted processes is a closed convex    subspace of $L^2 (\Omega \times [0,T])$, hence weakly
closed. Therefore $\displaystyle\int_{\mathbb{T}^2}\int_{\mathbb{R}^2}\widetilde{f}(x,r,t)\phi(x,r)dxdr$ is also adapted  and  its    It\^o   integral    is  well    defined and  bounded.   Now, let us consider the  following mapping
\begin{align*}
\mathcal{L}: L^2(\Omega\times[0,T]; H) &\to L^2(\Omega\times[0,T]; \mathbb{R})\\
f_m &\mapsto \sum_{k\in K}\int_0^t\int_{\mathbb{T}^2}\int_{\mathbb{R}^2}f_m(s)( \sigma_k^N.\nabla_x\phi +(\nabla \sigma_k^Nr).\nabla_r \phi) drdxdW^k(s)\notag,
\end{align*}
which is linear and bounded. Therefore, by  using
 \eqref{cv-appro-3}, we infer that 
 \begin{align*}
 \sum_{k\in K}&\int_0^t\int_{\mathbb{T}^2}\int_{\mathbb{R}^2}f_m(s)( \sigma_k^N.\nabla_x\phi +(\nabla \sigma_k^Nr).\nabla_r \phi) drdxdW^k(s) \\&\rightharpoonup \sum_{k\in K}\int_0^t\int_{\mathbb{T}^2}\int_{\mathbb{R}^2}\widetilde{f}(s)( \sigma_k^N.\nabla_x\phi +(\nabla \sigma_k^Nr).\nabla_r \phi) drdxdW^k(s) \text{  in } L^2(\Omega\times [0,T]).\notag
 \end{align*}
 \subsubsection{$3^{rd} step$} 
For $\phi\in H_m$ and  $t\in [0,T]$, let us  set 
\begin{align*}
B_m(t) :=\int_{\mathbb{T}^2}\int_{\mathbb{R}^2}f_m(t) \phi drdx  -\sum_{k\in K}\int_0^t\int_{\mathbb{T}^2}\int_{\mathbb{R}^2}f_m(s)( \sigma_k^N.\nabla_x\phi +(\nabla \sigma_k^Nr).\nabla_r \phi) drdxdW^k(s).
\end{align*}
From    \eqref{eqn-app-key1},   we  write   (in distributional sense    with    respect to  $t$)
		\begin{align*}&\dfrac{d}{dt}B_m=\int_{\mathbb{T}^2}\int_{\mathbb{R}^2}f_m\left(u_L\cdot \nabla_x \phi + (\nabla u_Lr-\dfrac{1}{\beta}r) \cdot\nabla_r \phi \right)dr dx\notag\\
  &-\sum_{k\in K}\int_{\mathbb{T}^2}\int_{\mathbb{R}^2}   \nabla_rf_m(\nabla\sigma_k^Nr)P_m[\sigma_k^N\cdot\nabla_x\phi ]drdx  -	\sigma^2\int_{\mathbb{T}^2}\int_{\mathbb{R}^2}\nabla_rf_m \cdot\nabla_r \phi drdxds\notag\\
		&+\alpha_N\int_{\mathbb{T}^2}\int_{\mathbb{R}^2}f_m \Delta_x\phi drdx- \dfrac{1}{2}\int_{\mathbb{T}^2}\int_{\mathbb{R}^2}(\sum_{k\in K}\left((\nabla \sigma_k^Nr) \otimes (\nabla \sigma_k^Nr)\right) \nabla_rf_m)\cdot\nabla_r \phi drdx.\notag\end{align*}

 Let $A\in \mathcal{F}$ and $\xi\in  \mathcal{D}(0,T)$\footnote{$\mathcal{D}(0,T)$ denotes the space of $\mathcal{C}^\infty$-functions with compact support in $]0,T[$.}, by multiplying the    last    equation    by    $\mathbb{I}_A\xi$ and integrating  over $\Omega \times [0,T]$ we derive 
\begin{align*}
&-\int_A\int_0^T\bigl[B_m\dfrac{d\xi}{ds}\bigr]dsdP\\
&=\int_A\int_0^T\int_{\mathbb{T}^2}\int_{\mathbb{R}^2}f_m(s)\left(u_L(s)\cdot \nabla_x \phi + (\nabla u_L(s)r-\dfrac{1}{\beta}r) \cdot\nabla_r \phi \right)\xi(s) dr dxdsdP \\
&-\int_A\int_0^T\sum_{k\in K}\int_{\mathbb{T}^2}\int_{\mathbb{R}^2}   \nabla_rf_m(s)(\nabla\sigma_k^Nr)P_m[\sigma_k^N\cdot\nabla_x\phi ] \xi(s) dr dxdsdP
\\
&-\int_A\int_0^T \sigma^2\int_{\mathbb{T}^2}\int_{\mathbb{R}^2}\nabla_rf_m(s) \cdot\nabla_r \phi \xi(s) dr dxdsdP+\int_A\int_0^T\alpha_N\int_{\mathbb{T}^2}\int_{\mathbb{R}^2}f_m(s) \cdot\Delta_x\phi \xi(s) dr dxdsdP\\
&-\int_A\int_0^T\dfrac{1}{2}\int_{\mathbb{T}^2}\int_{\mathbb{R}^2}(\sum_{k\in K}\left((\nabla \sigma_k^Nr) \otimes (\nabla \sigma_k^Nr)\right) \nabla_rf_m(s))\cdot\nabla_r \phi\xi(s) dr dxdsdP , \quad   \forall \phi \in H_m.
\end{align*}

  Now,    let us  prove the following.
  \begin{align}\label{limit-covarnaince-0}
      \int_A\int_0^T\sum_{k\in K}\int_{\mathbb{T}^2}\int_{\mathbb{R}^2}   \nabla_rf_m(s)(\nabla\sigma_k^Nr)P_m[\sigma_k^N\cdot\nabla_x\phi ] \xi(s) dr dxdsdP\to0 \text{    as  }   m\to    +\infty.
  \end{align}
  Indeed,   it's    sufficient  to  pass    to  the limit   with $(w_i)_{i\in\mathbb{N}}$   as  test    functions.  Thus, for $1\leq  i\leq   m$,    we  recall
        \begin{align*}
            -&\int_A\int_0^T\sum_{k\in K}\int_{\mathbb{T}^2}\int_{\mathbb{R}^2}   \nabla_rf_m(s)(\nabla\sigma_k^Nr)P_m[\sigma_k^N\cdot\nabla_x\phi ] \xi(s) dr dxdsdP\\=  &-\sum_{k\in K}\E\int_0^T(   \nabla_rf_m(s),(\nabla\sigma_k^Nr)P_m[\sigma_k^N\cdot\nabla_xw_i ])1_A\xi(s)  ds\\
              &=\E\int_0^T\langle   \sum_{k\in K}P_m(\sigma_k^N\cdot\nabla_xP_m[(\nabla\sigma_k^Nr).\nabla_rf_m(s)]),w_i \rangle1_A\xi(s) ds
                 \end{align*}
            
        where   $ \E$ denotes the expectation.
       On   the other   hand,   the following   convergence holds
        \begin{align*}
          \sum_{k\in K}  (\nabla\sigma_k^Nr)P_m[\sigma_k^N\cdot\nabla_xw_i ]    \to \sum_{k\in K}  (\nabla\sigma_k^Nr)\sigma_k^N\cdot\nabla_xw_i   \text{  in  }   L^2(\T^2\times\R^2). 
        \end{align*}
        Indeed, denote  by  $\Vert  \cdot\Vert_2$   the norm    in  $L^2(\T^2\times\R^2)$, we  get
        \begin{align*}
             \Vert \sum_{k\in K}  (\nabla\sigma_k^Nr)(P_m-I)[\sigma_k^N\cdot\nabla_xw_i ]  \Vert_2^2&\leq    
              \sum_{k\in K} (\theta_{\vert   k\vert}^N)^2\vert   k\vert^2 \Vert\vert   r\vert  (P_m-I)[\sigma_k^N\cdot\nabla_xw_i ]  \Vert_2^2\\    &\leq    
              \sum_{k\in K} (\theta_{\vert   k\vert}^N)^2\vert   k\vert^2 \Vert(P_m-I)[\sigma_k^N\cdot\nabla_xw_i ]  \Vert_H^2\\
              &\leq 
              \sum_{k\in K} (\theta_{\vert   k\vert}^N)^4\vert   k\vert^2 \Vert P_m-I\Vert_{L(H,H)}^2\Vert  \nabla_xw_i   \Vert_H^2\\
              &\leq  \sum_{k\in K} (\theta_{\vert   k\vert}^N)^4\vert   k\vert^2 \Vert P_m-I\Vert_{L(H,H)}^2\Vert w_i   \Vert_U^2\\
              &\leq  C \Vert P_m-I\Vert_{L(H,H)}^2\Vert w_i   \Vert_U^2 \to 0,
              \end{align*}
              where $C>0$ \footnote{
             Recall that     $\displaystyle\sum_{k\in K} (\theta_{\vert   k\vert}^N)^4\vert   k\vert^2   = \sum_{k\in K} \dfrac{1}{\vert   k\vert^6}  \leq    C$.}   and since   $P_m$   is  an  orthogonal projection  on  $H$. On  the other   hand,   $\nabla_rf_m$   converges   weakly  to  $\nabla_r\widetilde f$   in $L^2(\Omega\times[0,T]\times\T^2\times\R^2)$,  thanks  to \eqref{cv-appro-5}. Therefore
                   \begin{align*}
    \lim_m        -&\int_A\int_0^T\sum_{k\in K}\int_{\mathbb{T}^2}\int_{\mathbb{R}^2}   \nabla_rf_m(s)(\nabla\sigma_k^Nr)P_m[\sigma_k^N\cdot\nabla_x\phi ] \xi(s) dr dxdsdP\\
    &=-\sum_{k\in K}\int_A\int_0^T(   \nabla_r\widetilde f(s),(\nabla\sigma_k^Nr)\sigma_k^N\cdot\nabla_xw_i )\xi(s)   dsdP\\
    &=\int_A\int_0^T(  \widetilde f(s),\sum_{k\in K}(\nabla\sigma_k^Nr)\cdot\nabla_r\sigma_k^N\cdot\nabla_xw_i)\xi(s)   dsdP=0\quad   \forall i\in    \mathbb{N}.
        \end{align*}
        Indeed, for given  function    $\psi$ we  have    $\displaystyle\sum_{k\in K}	(\nabla \sigma_k^Nr).\nabla_r(\sigma_k^N.\nabla_x\psi)=	\sum_{k\in K} \sum_{l,\gamma,i=1}^2 \partial_{x_\gamma}\sigma_k^ir_\gamma\partial_{r_i}(\sigma_k^l\partial_{x_l} \psi)$
        and 
        \begin{align*}
					\sum_{k\in K} \sum_{l,\gamma,i=1}^2 \partial_{x_\gamma}\sigma_k^ir_\gamma\partial_{r_i}(\sigma_k^l\partial_{x_l} \psi)=	\sum_{k\in K} \sum_{l,\gamma,i=1}^2 \partial_{x_\gamma}\sigma_k^i\sigma_k^l\partial_{x_l}(r_\gamma\partial_{r_i} \psi)=\sum_{l,\gamma,i=1}^2\partial_{x_\gamma} Q_{i,l}(0)\partial_{x_l}(r_\gamma\partial_{r_i}\psi).
					\end{align*}
    Since   the covariance  matrix  $Q$ satisfies   $Q(x)=Q(-x)$ then   $\partial_{x_\gamma} Q_{i,l}(0)=0$. As a result we get
     \begin{align*}
&-\int_A\int_0^T\bigl[B\dfrac{d\xi}{ds}\bigr]dsdP\\
&=\int_A\int_0^T\int_{\mathbb{T}^2}\int_{\mathbb{R}^2}\widetilde    f(s)\left(u_L(s)\cdot \nabla_x \phi + (\nabla u_L(s)r-\dfrac{1}{\beta}r) \cdot\nabla_r \phi \right)\xi(s) dr dxdsdP \\
&-\int_A\int_0^T \sigma^2\int_{\mathbb{T}^2}\int_{\mathbb{R}^2}\nabla_r\widetilde    f(s) \cdot\nabla_r \phi \xi(s) dr dxdsdP+\int_A\int_0^T\alpha_N\int_{\mathbb{T}^2}\int_{\mathbb{R}^2}\widetilde    f(s) \cdot\Delta_x\phi \xi(s) dr dxdsdP\\
&-\int_A\int_0^T\dfrac{1}{2}\int_{\mathbb{T}^2}\int_{\mathbb{R}^2}(\sum_{k\in K}\left((\nabla \sigma_k^Nr) \otimes (\nabla \sigma_k^Nr)\right) \nabla_r\widetilde    f(s))\cdot\nabla_r \phi\xi(s) dr dxdsdP , \quad   \forall \phi \in U,
\end{align*}
where   \begin{align*}
B(t) :=\int_{\mathbb{T}^2}\int_{\mathbb{R}^2}\widetilde    f(t) \phi drdx  -\sum_{k\in K}\int_0^t\int_{\mathbb{T}^2}\int_{\mathbb{R}^2}\widetilde    f(s)( \sigma_k^N.\nabla_x\phi +(\nabla \sigma_k^Nr).\nabla_r \phi) drdxdW^k(s).
\end{align*}

Then, taking into account the regularity    of $\widetilde    f$ , we infer that the distributional derivative $\dfrac{dB}{dt}$ belongs to the space $L^2(\Omega\times[0,T]).$ Recalling that $B\in L^2(\Omega\times[0,T])$, we conclude that 
\begin{align*} B(\cdot)\in   L^2(\Omega;C([0,T]).
\end{align*}

Considering the properties of  It\^o's integral, we deduce
 $$ \int_{\mathbb{T}^2}\int_{\mathbb{R}^2}\widetilde    f(t) \phi drdx\in L^2(\Omega;C([0,T]),$$
 which  means   that   $\widetilde    f\in  L^2(\Omega;C([0,T];U^\prime)$ and   therefore 
 $\widetilde    f\in  L^2(\Omega;C_w([0,T];H),$ thanks to \eqref{cv-appro-4}    and \cite[Lemma. 1.4 p. 263]{Temam77}. We   finish  the proof   by  showing some    continuous  convergence in  time.   Indeed,
let  $\xi \in \mathcal{C}^\infty([0,t])$ for $t\in ]0,T]$ and note that the	following	integration	by	parts	formula	holds
\begin{align}\label{IPPtimez-1}
\int_0^t	\dfrac{dB}{ds}(s)\xi(s) ds&=-\int_0^tB(s)\dfrac{d\xi}{ds}ds+B(t)\xi(t)-\int_{\mathbb{T}^2}\int_{\mathbb{R}^2}    f_0 \phi drdx\xi(0).
\end{align}
Now,    by  standard   arguments  (see e.g.  \cite[proof    of  Prop. 3.]{vallet2019well} ) we  get
	for any $t\in ]0,T]$, $$\int_{\mathbb{T}^2}\int_{\mathbb{R}^2}\widetilde    f_m(t) \phi drdx \rightharpoonup \int_{\mathbb{T}^2}\int_{\mathbb{R}^2}\widetilde    f(t) \phi drdx \text{ in } L^2(\Omega,H), \text{ as } m\to \infty.$$
 and    $\widetilde f(0)=f_0$   in  $H$-sense.
In  conclusion,     there   exists  a   solution    in  the sense   of  \autoref{Def-sol} ($\widetilde  f=f^N$  to  stress  the dependence  $N$,    since   we  will    pass    to  the limit   as  $N\to   +\infty $   in  \autoref{Section-diffusion}.)  
\begin{align*}&\int_{\mathbb{T}^2}\int_{\mathbb{R}^2}f^N(t) \phi drdx-\int_{\mathbb{T}^2}\int_{\mathbb{R}^2}f_0 \phi drdx\\&-\int_0^t\int_{\mathbb{T}^2}\int_{\mathbb{R}^2}f^N(s)\left(u_L(s)\cdot \nabla_x \phi+ (\nabla u_L(s)r-\dfrac{1}{\beta}r) \cdot\nabla_r \phi \right)dr dxds\\	=&-	\sigma^2\int_0^t\int_{\mathbb{T}^2}\int_{\mathbb{R}^2}\nabla_rf^N(s) \cdot\nabla_r \phi drdxds+\alpha_N\int_0^t\int_{\mathbb{T}^2}\int_{\mathbb{R}^2}f^N(s) \cdot\Delta_x\phi   drdxds\\&+\sum_{k\in K}\int_0^t\int_{\mathbb{T}^2}\int_{\mathbb{R}^2}[f^N(s) \sigma_k^N.\nabla_x\phi+f^N(s) (\nabla \sigma_k^Nr).\nabla_r\phi ]drdxdW^k(s)\\
		&- \dfrac{1}{2}\int_0^t\int_{\mathbb{T}^2}\int_{\mathbb{R}^2}(\sum_{k\in K}\left((\nabla \sigma_k^Nr) \otimes (\nabla \sigma_k^Nr)\right) \nabla_rf^N(s))\cdot\nabla_r \phi drdxds,\end{align*}
		for any $\phi  \in U$.    In  particular, $f^N$   is  adapted with    respect to   the given   filtration.
  \subsection{Uniqueness of   quasi-regular weak    solution}\label{Subsection-uniq}
 In order   to  prove  (4) in \autoref{Def-sol}   and the  uniqueness  of  "\textit{quasi-regular   weak    solution}"    to  \eqref{Ito-FP},   we  need    first   to  prove   the existence   and uniqueness  of  solution   $V^N$ to  an  appropriate mean    equation.   Namely, we  prove   the following. 
\begin{proposition}\label{prop-exisence-mean}
  For   any $t\in[0,T],$   there   exists  $V^N(t)=\E[  f^N(t)e_g(t)]$  such    that
\begin{enumerate}
    \item $V^N \in  L^\infty([0, T];H),  $     $    \nabla_rV^N \in     L^2([0, T];H)$  and $V^N\in  C_w([0,T];H)$.
    \item   For any $t\in[0,T],$    it  holds
    \begin{align*}
&\int_{\mathbb{T}^2}\int_{\mathbb{R}^2}   V^N(t)\phi   dxdr-\int_{\mathbb{T}^2}\int_{\mathbb{R}^2}   f_0\phi dxdr\\
&=\int_0^t\int_{\mathbb{T}^2}\int_{\mathbb{R}^2}  V^N(s)\left([u_L(s)-h_n]\cdot \nabla_x \phi + ([\nabla u_L(s)r-y_n]-\dfrac{1}{\beta}r) \cdot\nabla_r \phi \right) dr dxds \\
&-\int_0^t \sigma^2\int_{\mathbb{T}^2}\int_{\mathbb{R}^2}\nabla_r V^N(s) \cdot\nabla_r \phi  dr dxds+\int_0^T\alpha_N\int_{\mathbb{T}^2}\int_{\mathbb{R}^2}V^N(s) \cdot\Delta_x\phi  dr dxds\\
&-\int_0^t\dfrac{1}{2}\int_{\mathbb{T}^2}\int_{\mathbb{R}^2}(\sum_{k\in K}\left((\nabla \sigma_k^Nr) \otimes (\nabla \sigma_k^Nr)\right) \nabla_r V^N(s))\cdot\nabla_r \phi dr dxds , \quad   \forall \phi \in V,
\end{align*}
   \end{enumerate}
   \end{proposition}
    \begin{remark}
       A    priori, the last    point   holds   for any $\phi   \in U$  and by  taking  into    account the regularity  of  the solution    $V^N$,    it  holds   as  well    for all $\phi   \in V.$
   \end{remark}
 
  \subsubsection*{Proof of  \autoref{prop-exisence-mean}}
  Let $g\in   G_n$,   by  using   It\^o formula to the
product,    we  get
\begin{align*}
d[e_g(t)f_m(t)]&+e_g(t)P_m\Div_x(u_Lf_m(t))dt+e_g(t)P_m\Div_r((\nabla u_Lr-\dfrac{1}{\beta}r)f_m(t)) dt\\	&=\sum_{k\in K}e_g(t)P_m(\sigma_k^N\cdot\nabla_xP_m[(\nabla\sigma_k^Nr).\nabla_rf_m])dt+	\sigma^2e_g(t)P_m\Delta_rf_m(t)dt\notag\\&-\sum_{k\in K}e_g(t)P_m\sigma_k^N.\nabla_xf^N(t)dW^k -\sum_{k\in K}e_g(t)P_m(\nabla \sigma_k^Nr).\nabla_rf_m(t) dW^k\notag\\
&+\alpha_N	e_g(t)P_m\Delta_xf_m(t)dt+ \dfrac{1}{2}e_g(t)P_m\Div_r(\sum_{k\in K}\left((\nabla \sigma_k^Nr) \otimes (\nabla \sigma_k^Nr)\right) \nabla_rf_m(t))dt,\notag\\
&+\sum_{k\in M_n}f_m(t)g_k(t)e_g(t)dW^k(t)\\
&- [\sum_{k\in M_n}g_k(t)e_g(t)dW^k(t),\sum_{k\in K}e_g(t)P_m\sigma_k^N.\nabla_xf^N(t)+e_g(t)P_m(\nabla \sigma_k^Nr).\nabla_rf_m(t) dW^k]\\
e_g(t)f_m|_{t=0}&=P_mf_0.\notag
\end{align*}
Denote  $K_n=\{k\in K: \min(n,N) \leq    \vert   k\vert  \leq    \max(2N,n)\}$  and
set $V_m(t)=\E(f_m(t)e_g(t))$   then    
\begin{align*}
d[V_m(t)]&+P_m\Div_x(u_LV_m(t))dt+P_m\Div_r((\nabla u_Lr-\dfrac{1}{\beta}r)V_m(t)) dt\\	&=\sum_{k\in K}P_m(\sigma_k^N\cdot\nabla_xP_m[(\nabla\sigma_k^Nr).\nabla_rV_m(t)])dt+	\sigma^2P_m\Delta_rV_m(t)dt\notag\\
&+\alpha_N	P_m\Delta_xV_m(t)dt+ \dfrac{1}{2}P_m\Div_r(\sum_{k\in K}\left((\nabla \sigma_k^Nr) \otimes (\nabla \sigma_k^Nr)\right) \nabla_rV_m(t))dt,\notag\\
&- \sum_{k\in K_n}P_mg_k\sigma_k^N.\nabla_xV_m(t)+P_mg_k(\nabla \sigma_k^Nr).\nabla_rV_m(t)\\
V_m|_{t=0}&=P_mf_0.\notag
\end{align*}
  Denote    $\displaystyle\sum_{k\in K_n}g_k\sigma_k^N=h_n$    and $\displaystyle\sum_{k\in K_n}g_k(\nabla \sigma_k^Nr)=y_n$.  
Then    $V_m$   satisfies 
\begin{align}
\dfrac{d   V_m}{dt}
&+P_m\Div_x([u_L-h_n]V_m(t))+P_m\Div_r([(\nabla u_Lr)-y_n]-\dfrac{1}{\beta}r)V_m(t))\label{eqn-Vm}\\	
&=\sum_{k\in K}P_m(\sigma_k^N\cdot\nabla_xP_m[(\nabla\sigma_k^Nr).\nabla_rV_m])+	\sigma^2P_m\Delta_rV_m(t)\notag\\
&+\alpha_N	P_m\Delta_xV_m(t)+ \dfrac{1}{2}P_m\Div_r(\sum_{k\in K}\left((\nabla \sigma_k^Nr) \otimes (\nabla \sigma_k^Nr)\right) \nabla_rV_m(t)),\notag\\
V_m|_{t=0}&=P_mf_0.\notag
\end{align}
 \eqref{eqn-Vm}  is  linear  system  of  ODE,    by  using   a   classical   results  (see    e.g.    \cite[Chapter   V]{Oksendal2013})  we  get
\begin{lemma}
    There   exists  a   unique  $V_m\in C([0,T];H_m)$  to   \eqref{eqn-Vm}.
\end{lemma}
 \begin{lemma}\label{Lemma9-1}
 For every  $m\in   \mathbb{N}^*$, there   exists  a   unique   solution $V_m=\E[e_gf_m]\in C([0,T];H_m)$  to    \eqref{eqn-Vm}, which    satisfy    
\begin{align}\label{borne-uniq}
    \forall t\geq   0:\displaystyle\sup_{q\in[0,t]}\Vert  V_m(q)\Vert^2_H+4\sigma^2\int_0^t \Vert\nabla_rV_m(s)\Vert_H^2 \leq	2\Vert  f_0\Vert_H^2e^{\overline\lambda(t)	},
 \end{align}
 where
\begin{align*} \overline\lambda(t)=[2\Vert	\nabla_x u_L\Vert_{\infty}+\dfrac{8}{\beta}+4\sigma^2+2C]t+\int_0^t(\Vert  g(s)\Vert^2+1)ds<+\infty .
\end{align*}
\end{lemma}
\begin{proof}The proof consists of arguments analogous to the proof of  \autoref{Lemma9}    but for the reader's convenience, let   us sketch it.
By  applying    It\^o formula with  $\Vert  \cdot\Vert_H^2$  to  the process $V_m$    satisfying  \eqref{eqn-Vm},	 we  obtain  
\begin{align*}&\hspace*{-1cm}\displaystyle\Vert  V_m(t)\Vert^2_H-\Vert  P_mf_0\Vert_H^2 \\&\hspace*{-1cm}=2\int_0^t(V_m(s),u_L\cdot \nabla_x V_m(s))_Hds-2\int_0^t(V_m(s),h_n\cdot \nabla_x V_m(s))_Hds-2\int_0^t((\nabla_x u_Lr)\cdot\nabla_rV_m(s),V_m(s))_H\\&+2\int_0^t(V_m(s),y_n\cdot \nabla_r V_m(s))_Hds
+ \int_0^t\dfrac{2}{\beta}\Vert  V_m(s)\Vert^2_H ds+\dfrac{1}{\beta}\int_0^t(r\cdot\nabla_rV_m(s) ,V_m(s) )_Hds\notag\\&
+2\sum_{k\in K}\int_0^t(V_m(s),\sigma_k^N\cdot\nabla_xP_m[(\nabla\sigma_k^Nr))\nabla_rV_m(s)])_Hds\\&-	2\int_0^t[\sigma^2 \Vert\nabla_rV_m(s)\Vert_H^2+2\sigma^2 (\nabla_rV_m(s),rV_m(s))+\alpha_N\Vert\nabla_xV_m(s)\Vert_H^2]ds\notag\\
&-\int_0^t (A_k^N(x,r)\nabla_rV_m(s),\nabla_rV_m(s))_H ds-2\int_0^t (A_k^N(x,r)\nabla_rV_m(s),rV_m(s)) ds\notag
\end{align*}
The last    
equation    has        similar terms as  \eqref{Ito-form-1-} (without the stochastic  integrals and It\^o   correctors)    but with    the   two  new  terms   
\begin{align*}
    2\int_0^t(V_m(s),h_n\cdot \nabla_x V_m(s))_Hds+2\int_0^t(V_m(s),y_n\cdot \nabla_r V_m(s))_Hds.
\end{align*}
By  noticing  that  
$\Div_x(h_n)=0$,    the first   term    vanishes.   Concerning  the second  one,   note that    $
     \vert  y_n\vert      \leq  \vert   r\vert    \displaystyle\sum_{    \vert   k\vert  \leq    n}\vert   g_k\vert=\vert   r\vert  \Vert  g\Vert
$
and 
  \begin{align*}
      \int_0^t(V_m(s),y_n\cdot \nabla_r V_m(s))_Hds=-\displaystyle\sum_{k\in K_n}  (V_m(s),g_k(\nabla \sigma_k^Nr)\cdot r V_m(s))ds,
  \end{align*}
    which   ensures that
   $
      \displaystyle\int_0^t\vert    (V_m(s),y_n\cdot \nabla_r V_m(s))_H\vert    ds\leq \int_0^t(\Vert  g(s)\Vert^2+1)\Vert	V_m(s)\Vert_H^2ds.   
    $  Thus, the other   terms   can be  estimated   by  similar arguments as in    
\autoref{Subsection-estimate-fm}  and   we  obtain \autoref{Lemma9-1}.
\end{proof}
By  using   \autoref{Lemma9-1},      we  are able    to  get the following   convergences (up to a subsequence)
  \begin{align}
 	V_m \rightharpoonup \widetilde{V} &\text{		in	}  L^2([0, T];H),\label{cv-appro-3-2}\\
 	 V_m \rightharpoonup^* \widetilde{V} &\text{		in	}  L^\infty([0, T];H),\label{cv-appro-4-2}\\
   \nabla_rV_m \rightharpoonup \nabla_r\widetilde{V} &\text{		in	}  L^2([0, T];H),\label{cv-appro-5-2}
 	\end{align}
  moreover, by  using   the linearity   of  the expectation we  get $\widetilde{V}=\E[e_g\widetilde{f}]$   (recall that    $V_m=\E(e_gf_m)$    and $e_g\in L^2(\Omega)$).   Now,    we  have    all the ingredients  in  hand    to  argue  as  in  \autoref{Subsection-exist-proof}    and obtain   \autoref{prop-exisence-mean}.
  
  \begin{lemma}\label{lem-uniq-Vn}
   Let   $N\in   \mathbb{N}^*$.   Then, the   solution    $V^N$ given   by  \autoref{prop-exisence-mean}    is  unique.
  \end{lemma}
  \subsubsection*{Proof of  \autoref{lem-uniq-Vn}}
  Let $V^N_1$   and $V^N_2$ be two solutions,  with  the same    initial data, given   by  \autoref{prop-exisence-mean} and denote  by   $V^N$  their difference.   Then
    for any $t\in ]0, T]$   and $\phi\in    V$, we  have
		\begin{align}
&\int_{\mathbb{T}^2}\int_{\mathbb{R}^2}   V^N(t)\phi   dxdr\label{eqn-regularize}\\
&=\int_0^t\int_{\mathbb{T}^2}\int_{\mathbb{R}^2}  V^N(x,r,s)\left([u_L(x,s)-h_n]\cdot \nabla_x \phi + ([\nabla u_L(s,x)r-y_n]-\dfrac{1}{\beta}r) \cdot\nabla_r \phi \right) dr dxds \notag\\
&-\int_0^t \sigma^2\int_{\mathbb{T}^2}\int_{\mathbb{R}^2}\nabla_r V^N(x,r,s) \cdot\nabla_r \phi  dr dxds+\int_0^T\alpha_N\int_{\mathbb{T}^2}\int_{\mathbb{R}^2}V^N(x,r,s) \cdot\Delta_x\phi  dr dxds\notag\\
&-\int_0^t\dfrac{1}{2}\int_{\mathbb{T}^2}\int_{\mathbb{R}^2}(\sum_{k\in K}\left((\nabla \sigma_k^Nr) \otimes (\nabla \sigma_k^Nr)\right) \nabla_r V^N(x,r,s))\cdot\nabla_r \phi dr dxds.\notag
\end{align}
  Since   the above   equation    is  understood    in  weak    form,   we  need    first   to  consider    an  appropriate  regularization  of   $V^N$,  denoted by  $[V^N]_\delta$.  Then,   take    the $L^2$-inner   product of  the above   equation    with  $[V^N]_\delta$   and finally    pass    to  the limit   with    respect to  the regularization  parameters $\delta$.  
    \subsubsection*{Step  1:  Regularization}

Let  $\varphi\in C^\infty_c(\T^2\times  \R^2)$, if  we  denote  $X=(x,r)\in \T^2\times  \R^2$   and $\rho_{\delta}(X)=\rho_\delta(x)$,
then    $\varphi_{\delta}:=\rho_{\delta}*\varphi$   is  an  appropriate test    function    in  \eqref{eqn-regularize},  namely  we  get
	\begin{align*}
&( V^N(t),\varphi_{\delta})\\
&=\int_0^t\int_{\mathbb{T}^2}\int_{\mathbb{R}^2}  V^N(s)\left([u_L(s)-h_n]\cdot \nabla_x \varphi_{\delta} + ([\nabla u_L(s)r-y_n]-\dfrac{1}{\beta}r) \cdot\nabla_r \varphi_{\delta} \right) dr dxds \notag\\
&-\int_0^t \sigma^2\int_{\mathbb{T}^2}\int_{\mathbb{R}^2}\nabla_r V^N(s) \cdot\nabla_r \varphi_{\delta}  dr dxds+\int_0^T\alpha_N\int_{\mathbb{T}^2}\int_{\mathbb{R}^2}V^N(s) \cdot\Delta_x\varphi_{\delta}  dr dxds\notag\\
&-\int_0^t\dfrac{1}{2}\int_{\mathbb{T}^2}\int_{\mathbb{R}^2}(\sum_{k\in K}\left((\nabla \sigma_k^Nr) \otimes (\nabla \sigma_k^Nr)\right) \nabla_r V^N(s))\cdot\nabla_r \varphi_{\delta} dr dxds.\notag
\end{align*}
Since   $\rho$  is  radially    symmetric   then    the operator of convolution with
 $\rho_{\delta}$  is  self-adjoint    on $L^2 (\T^2 \times \R^2)$.    Thus  
 \begin{align*}
&( [V^N(t)]_\delta,\varphi)+\int_0^t  \left([(u_L(s)-h_n)\cdot \nabla_xV^N(s)]_\delta+ +[\Div_r\big( (\nabla u_L(s)r-y_n-\dfrac{1}{\beta}r) V^N(s)\big)]_\delta, \varphi \right) ds \notag\\
&=\int_0^t \sigma^2([\Delta_r V^N(s)]_\delta,  \varphi)  ds+\int_0^T\alpha_N([\Delta_xV^N(s)]_\delta,\varphi) ds\notag\\
&+\int_0^t\dfrac{1}{2}[\Div_r(\sum_{k\in K}\left((\nabla \sigma_k^Nr) \otimes (\nabla \sigma_k^Nr)\right) \nabla_r V^N(s))]_\delta, \varphi)ds=\int_0^t([d(s)]_\delta,\varphi)ds.\notag
\end{align*} Consider
    the following   space   $X:=\{\varphi\in    H;   \nabla_r\varphi\in    H \}$    and note    that    $X\hookrightarrow L^2(\T^2\times\R^2)\hookrightarrow    X^\prime  $ is  Gelfand triple.   Since   $$V^N \in L^\infty([0, T];H), \nabla_rV^N\in L^2( [0, T];H),$$    and by  using   the regularization  properties  of  $\rho$, one gets    that  $[d(\cdot)]_{\delta}\in    L^2([0, T];X^\prime)$. therefore,  we  can set 
 $\varphi=[V^N(\cdot)]_\delta$ to  get 
   \begin{align}&\dfrac{1}{2} \Vert [V^N(t)]_{\delta}\Vert^2+\int_0^t  \left([(u_L(s)-h_n)\cdot \nabla_xV^N(s)]_\delta+ [\Div_r\big( (\nabla u_L(s)r-y_n-\dfrac{1}{\beta}r) V^N(s)\big)]_\delta, [V^N(s)]_\delta \right) ds \label{Vn-regularized-uniq}\\
&=\int_0^t \sigma^2([\Delta_r V^N(s)]_\delta,  [V^N(s)]_\delta)  ds+\int_0^T\alpha_N([\Delta_xV^N(s)]_\delta,[V^N(s)]_\delta) ds\notag\\
&+\int_0^t\dfrac{1}{2}[\Div_r(\sum_{k\in K}\left((\nabla \sigma_k^Nr) \otimes (\nabla \sigma_k^Nr)\right) \nabla_r V^N(s))]_\delta, [V^N(s)]_\delta)ds.\notag
\end{align}
   Now,  let   us    pass    to  the limit   in  the last   equality  \eqref{Vn-regularized-uniq}.
   The  proof   is  based   on  commutator  estimates   and using   the properties  of  $(\sigma_k)_{k\in   K}.$

\subsubsection*{Step    2:  Passage  to  the limit   as  $\delta\to0$}
   First, recall    that   $  V^N(\cdot)\in   L^2(\T^2\times  \R^2)$, by  properties  of  convolution product,    we  get    $\displaystyle\lim_{\delta\to    0}[V^N(r,\cdot)]_{\delta}= V^N(r,\cdot)$ in  $L^2(\T^2)$    uniformly   for a.e.    $r\in   \R^2$    and we  get  $\displaystyle\lim_{\delta\to   0}\Vert [V^N(t)]_{\delta}\Vert^2=\Vert V^N(t)\Vert^2.$
   Next, we  will    prove   the following
    \begin{align}\label{eqn.delta-1-2-*}
\displaystyle\lim_{\delta}\int_0^t\langle [u_L(s)\cdot \nabla_xV^N(s)]_{\delta}, [V^N(s)]_{\delta}\rangle   ds=0.
    \end{align}

Since   $\Div_xu_L=0$,  we  get
\begin{align*}
    &\int_0^t\langle [u_L(s)\cdot \nabla_xV^N(s)]_{\delta}, [V^N(s)]_{\delta}\rangle   ds\\
    &=\int_0^t\langle   \rho_{\delta}*[u_L(s)\cdot \nabla_xV^N(s)]- u_L(s)\cdot \nabla_x\rho_{\delta}*[V^N(s)], \rho_{\delta}*[V^N(s)]\rangle   ds.
\end{align*}
Let us  introduce   the  commutator 
$
r_{\delta}(s)= \rho_{\delta}*[u_L(s)\cdot \nabla_xV^N(s)]-u_L(s)\cdot \nabla_x\rho_{\delta}*[V^N(s)].$  Thus,   
\eqref{eqn.delta-1-2-*} is  a   consequence  of the following:
   a.e.  $ s\in[0,T]$
\begin{align}
&
\left\Vert r_{\delta}(s)\right\Vert _{L^{2}(\T^2\times\R^2)}\leq C\left\Vert \nabla u_L(s)\right\Vert
_{\infty}\left\Vert V^N(s)\right\Vert _{L^{2}(\T^2\times\R^2)},   \label{lem-17-eqn1-2-*}%
\\
&\lim_{\delta}r_{\delta}(s)=0\text{ in    } L^{2}(\T^2\times\R^2)  \label{lem-17-eqn2-2-*},%
\end{align}
where   $C>0$   independent of  $\delta$.   Indeed, let us show \eqref{lem-17-eqn1-2-*}, note    that
\[
r_{\delta}\left(s,  x,r\right)  =-\int_{\R^2}\left(  u_L\left(  x,s\right)  -u_L\left(  y,s\right)
\right)  \cdot\nabla_{x}\rho_{\delta}\left(  x-y\right)  V^N\left(s,  y,r\right)  dy
\]%

Consider    the following   change  of  variables   $z=\dfrac{x-y}{\delta}$  to  get 
\begin{align*}
    r_{\delta}\left(s,  x,r\right)  &=-\int_{\R^2}  \dfrac{u_L\left(  y+\delta  z,s\right)  -u_L\left(  y,s\right)
}{\delta}  \cdot\nabla_{x}\rho\left(  z\right)  V^N\left(s,  y,r\right)  dz\\
&=-\int_{\R^2}  \int_0^1\nabla  u_L(y+\alpha \delta    z,s)zd\alpha  \cdot\nabla_{x}\rho\left(  z\right)  V^N\left(s,  y,r\right)  dz.
\end{align*}
Thus    $
  \vert  r_{\delta}\left(s,  x,r\right)\vert\leq \Vert    \nabla   u_L(s)\Vert_{\infty}\int_{\R^2}\vert z\vert  \vert\nabla_{x}\rho\left(  z\right)\vert  \vert    V^N\left(s,  y,r\right)\vert  dz.   
$
Since   $x=y+\delta z$,   we  get 
\begin{align*}
    \Vert   r_{\delta}(s)\Vert_{L^{2}(\T^2\times\R^2)}^2&\leq \Vert    \nabla   u_L(s)\Vert_{\infty}^2\int_{\R^2}\int_{\T^2}  (\int_{\R^2}\vert z\vert  \vert\nabla_{x}\rho\left(  z\right)\vert  \vert    V^N\left(s,  y,r\right)\vert  dz)^2dydr\\
    &\leq   \Vert    \nabla u_L(s)\Vert_{\infty}^2(\int_{\R^2}\int_{\T^2}  (\int_{\R^2} z  \cdot\nabla_{x}\rho\left(  z\right)  V^N\left(s,  y,r\right)  dz   )^2dydr\\&\leq   \Vert    \nabla  u_L(s)\Vert_{\infty}^2  \int_{\R^2} \vert   z\vert^2  \vert\nabla_{x}\rho\left(  z\right)\vert^2 dz (\int_{\R^2}\int_{\T^2}\vert V^N\left(s,  y,r\right)\vert^2     dydr,\\
    &\leq  C^2 \Vert    \nabla  u_L(s)\Vert_{\infty}^2\Vert V^N(s)\Vert_{L^{2}(\T^2\times\R^2)}^2,
\end{align*}
since    $\text{supp}[\rho]\subset    B(0,1)$  and    denoted by  $C^2=\int_{\R^2} \vert   z\vert^2  \vert\nabla_{x}\rho\left(  z\right)\vert^2 dz<+\infty.$  Concerning    \eqref{lem-17-eqn2-2-*},  we  have
\begin{equation*}
L^2(\T^2\times\R^2)\mbox{-}\lim_{\delta\rightarrow0}r_\delta(s)=-V^N(s)\left(\int_{\R^2}\nabla    u_L(s)z\cdot\nabla_x\rho(z)dz\right).
\end{equation*} Indeed, we  have
\begin{align*}
&\int_{\R^2}\int_{T^2}\vert \int_{\R^2}  \int_0^1\nabla u_L(y+\alpha \delta    z,s)zd\alpha  \cdot\nabla_{x}\rho\left(  z\right)  V^N\left(s,  y,r\right)  dz\\&\qquad\qquad\qquad\qquad-    V^N(s,x,r)\left(\int_{\R^2}\nabla  u_L(s,x)z\cdot\nabla_x\rho(z)dz\right)\vert^2dxdr\\
&=\int_{\R^2}\int_{T^2}\vert \int_{\R^2}  \int_0^1\nabla    u_L(y+\alpha \delta    z,s)zd\alpha  \cdot\nabla_{x}\rho\left(  z\right)  V^N\left(s,  y,r\right)  dz\\&\qquad\qquad\qquad\qquad-\int_{\R^2}V^N(s,x,r)\nabla    u_L(s,x)z\cdot\nabla_x\rho(z)dz\vert^2dxdr\\
&\leq\int_{\R^2}\int_{T^2} \int_{\R^2}  \int_0^1\vert   \nabla  u_L(y+\alpha \delta    z,s)z  \cdot\nabla_{x}\rho\left(  z\right)  V^N\left(s,  y,r\right)  dz-V^N(s,x,r)\nabla u_L(s,x)z\cdot\nabla_x\rho(z)\vert^2d\alpha  dzdxdr\\
&\leq\int_{\R^2} \int_{\R^2} \int_{T^2} \int_0^1\vert   \nabla  u_L(y+\alpha \delta    z,s)  V^N\left(s,  y,r\right)  -V^N(s,y+\delta    z,r)\nabla u_L(s,y+\delta z)\vert^2d\alpha  dydr    \vert   z\vert^2\vert\nabla_x\rho(z)\vert^2dz\\
&\leq   2\int_{\R^2} \int_{\R^2} \int_{T^2} \int_0^1\vert   \nabla  u_L(y+\alpha \delta    z,s)  V^N\left(s,  y,r\right)  -V^N(s,y,r)\nabla u_L(s,y)\vert^2d\alpha  dydr    \vert   z\vert^2\vert\nabla_x\rho(z)\vert^2dz\\
&+2\int_{\R^2} \int_{\R^2} \int_{T^2} \int_0^1\vert   V^N(s,y,r)\nabla   u_L(s,y)  -V^N(s,y+\delta    z,r)\nabla u_L(s,y+\delta z)\vert^2d\alpha  dydr    \vert   z\vert^2\vert\nabla_x\rho(z)\vert^2dz\\&=I^1_\delta+I^2_\delta.
\end{align*}
Using the continuity of translations in $L^2(\T^2)$ for the function $V^N\nabla u_L$, we get  $\displaystyle\limsup_{\delta\to 0}        I^2_\delta=0$
Concerning  $I^1_\delta$,   note    that    
\begin{align*}
    I^1_\delta\leq  2\int_{\R^2} \int_{\R^2} \int_{T^2} \int_0^1\vert   \nabla u_L(y+\alpha \delta    z,s)    -\nabla    u_L(s,y)\vert^2\vert V^N\left(s,  y,r\right)\vert^2d\alpha  dydr    \vert   z\vert^2\vert\nabla_x\rho(z)\vert^2dz.
\end{align*}
On  the other   hand,   by   mean-value theorem we  get
\begin{align*}
    \vert   \nabla  u_L(y+\alpha \delta    z,s)    -\nabla  u_L(s,y)\vert \leq        \alpha \delta\vert       z\vert \Vert   u_L\Vert_{C^2}
\end{align*}
Thus    $
    \displaystyle   I^1_\delta\leq  
    \delta^2 \Vert   u_L\Vert_{C^2}^2\int_{\R^2}  \int_{T^2} \vert V^N\left(s,  y,r\right)\vert^2  dydr    \int_{\R^2}\vert   z\vert^4\vert\nabla_x\rho(z)\vert^2dz \to 0.$
Finally,    since   $\rho$  is  smooth density of a probability measure,  we  get
  $\int_{\R^2}z_i\partial_j\rho(z)dz=-\delta_{ij}$  and so 
$$\int_{\R^2}\nabla u_L(s)z\cdot\nabla_x\rho(z)dz=-\Div  u_L=0,$$
which   gives   \eqref{lem-17-eqn2-2-*}.    The next    step    is  proving the following
 \begin{align}\label{eqn.delta-1-2-*-hn}
\displaystyle\lim_{\delta}\int_0^t\langle [h_n(s)\cdot \nabla_xV^N(s)]_{\delta}, [V^N(s)]_{\delta}\rangle   ds=0.
    \end{align}
  We    recall    that    $\displaystyle\sum_{k\in K_n}g_k\sigma_k^N=h_n$
  and   $\Div_xh_n=0$,  hence
\begin{align*}
    &\int_0^t\langle [h_n(s)\cdot \nabla_xV^N(s)]_{\delta}, [V^N(s)]_{\delta}\rangle   ds\\&=\int_0^t\langle   \rho_{\delta}*[h_n(s)\cdot \nabla_xV^N(s)]- h_n(s)\cdot \nabla_x\rho_{\delta}*[V^N(s)], \rho_{\delta}*[V^N(s)]\rangle   ds
\end{align*}
Let us  introduce   the  commutator 
\[
r_{\delta}^h(s)= \rho_{\delta}*[h_n(s)\cdot \nabla_xV^N(s)]-u_L(s)\cdot \nabla_x\rho_{\delta}*[V^N(s)].
\]
\eqref{eqn.delta-1-2-*-hn} is  a   consequence  of the following:
   a.e.  $ s\in[0,T]$
\begin{align}
&
\left\Vert r_{\delta}^h(s)\right\Vert _{L^{2}(\T^2\times\R^2)}\leq C\left\Vert g\right\Vert
\left\Vert V^N(s)\right\Vert _{L^{2}(\T^2\times\R^2)},   \label{lem-17-eqn1-2-*-hn}%
\\
&\lim_{\delta}r_{\delta}^h(s)=0\text{ in    } L^{2}(\T^2\times\R^2)  \label{lem-17-eqn2-2-*-hn},%
\end{align}
where   $C>0$   independent of  $\delta$.   Indeed, similarly    to \eqref{lem-17-eqn1-2-*}, we  get
\[
r_{\delta}^h\left(s,  x,r\right)  =-\int_{\R^2}\left(  h_n\left(  x,s\right)  -h_n\left(  y,s\right)
\right)  \cdot\nabla_{x}\rho_{\delta}\left(  x-y\right)  V^N\left(s,  y,r\right)  dy
\]%
thus,   we  obtain  ($\Vert \cdot\Vert_\infty$  denotes the $L^\infty$-norm with    respect to  the $x$-variable)
\begin{align*}
    \Vert   r_{\delta}^h(s)\Vert_{L^{2}(\T^2\times\R^2)}^2
    &\leq  C^2 \Vert    \nabla  h_n(s)\Vert_{\infty}^2\Vert V^N(s)\Vert_{L^{2}(\T^2\times\R^2)}^2,
\end{align*}
since    $\text{supp}[\rho]\subset    B(0,1)$  and    denoted by  $C^2=\int_{\R^2} \vert   z\vert^2  \vert\nabla_{x}\rho\left(  z\right)\vert^2 dz<+\infty.$
On  the other   hand,   note    that    $$\Vert\nabla_xh_n  \Vert_\infty\leq    \displaystyle\sum_{k\in K_n}\vert   g_k\vert    \Vert\nabla\sigma_k^N\Vert_\infty\leq   \displaystyle\sum_{k\in K_n}\dfrac{1}{\vert  k\vert}\vert   g_k\vert    \leq   \displaystyle\sum_{  \vert   k\vert  \leq  n} \vert   g_k\vert :=\Vert   g\Vert \in  L^2(0,T).$$
 The    proof   \eqref{lem-17-eqn2-2-*-hn}   is analogous   to  the proof \eqref{lem-17-eqn2-2-*}   and we omit this    detail. Let us  prove   that       $\displaystyle\lim_{\delta}\int_0^t\langle[\Div_r(\nabla u_L(s)r)V^N(s)]_{\delta}, [V^N(s)]_{\delta}\rangle   ds=0.$
Since   $\Div_r(\nabla u_L(s)r)=0$, we  get 
\begin{align*}
    &\int_0^t\langle[\Div_r(\nabla u_L(s)r)V^N(s)]_{\delta}, [V^N(s)]_{\delta}\rangle   ds\\&=\int_0^t\langle\Div_r\rho_\delta*(\nabla u_L(s)r)V^N(s)-\Div_r(\nabla u_L(s)r)\rho_\delta *   V^N(s), \rho_\delta *   V^N(s)\rangle   ds\\
    &=-\int_0^t\langle\rho_\delta*(\nabla u_L(s)r)V^N(s)-(\nabla u_L(s)r)\rho_\delta *   V^N(s), \rho_\delta *   \nabla_rV^N(s)\rangle   ds.
\end{align*}
On  the other   hand,   note    that
\begin{align*}
    &\big(\rho_\delta*(\nabla u_L(\cdot)r)V^N\big)(s,x,r)-(\nabla u_L(s,x)r)(\rho_\delta *   V^N)(s,x,r)\\
    &=\int_{\R^2}[\nabla    u_L(x-y,s)-\nabla   u_L(x,s)]r\rho_\delta(y)V^N(s,x-y,r)dy.
\end{align*}
By   mean-value theorem we  get
$    \vert   \nabla u_L(x-y,s)    -\nabla    u_L(x,s)\vert \leq        \vert       y\vert \Vert   u_L\Vert_{C^2}$
and 
\begin{align*}
    &\vert  \big(\rho_\delta*(\nabla u_L(\cdot)r)V^N\big)(s,x,r)-(\nabla u_L(s,x)r)(\rho_\delta *   V^N)(s,x,r)\vert    \\
    &\leq  \Vert   u_L\Vert_{C^2} \int_{\R^2}\vert    y\vert  \vert   r\vert  \rho_\delta(y)\vert   V^N(s,x-y,r)\vert       dy\leq  \delta\Vert   u_L\Vert_{C^2} \int_{\R^2}   \rho_\delta(y)\vert   r\vert \vert   V^N(s,x-y,r)\vert   dy,
\end{align*}
since   $\text{supp}[    \rho]\subset B(0,1).$    Therefore,  we  get
\begin{align*}
    &\int_0^t\vert\langle\rho_\delta*(\nabla u_L(s)r)V^N(s)-(\nabla u_L(s)r)\rho_\delta *   V^N(s), \rho_\delta *   \nabla_rV^N(s)\rangle \vert  ds\\
    &\leq    \delta\Vert   u_L\Vert_{C^2}    \int_0^t\Vert \int_{\R^2}   \rho_\delta(y)\vert   r\vert \vert   V^N(s,x-y,r)\vert   dy \Vert \Vert\rho_\delta *   \nabla_rV^N(s)\Vert ds\\
    &\leq    \delta\Vert   u_L\Vert_{C^2}    \int_0^t\Vert \rho_\delta *\vert   r\vert \vert   V^N(s)\vert    \vert    \Vert \Vert\rho_\delta *   \nabla_rV^N(s)\Vert ds\\
    &\leq    \delta\Vert   u_L\Vert_{C^2}    \int_0^t\Vert    V^N(s)    \vert    \Vert_H \Vert   \nabla_rV^N(s)\Vert ds\to    0   \text{  as    } \delta  \to 0.
\end{align*}
Concerning  the term        $\displaystyle\int_0^t\langle[\Div_r(y_nV^N(s))]_{\delta}, [V^N(s)]_{\delta}\rangle   ds.$   We  
recall that $\displaystyle\sum_{k\in K_n}g_k(\nabla \sigma_k^Nr)=y_n.$
Since   $\Div_r(y_n)=0$, we  have
\begin{align*}
    &\int_0^t\langle[\Div_ry_nV^N(s)]_{\delta}, [V^N(s)]_{\delta}\rangle   ds=\int_0^t\langle[\Div_ry_nV^N(s)]_{\delta}-\Div_ry_n[V^N(s)]_{\delta}, [V^N(s)]_{\delta}\rangle   ds\\
    &=\int_0^t\langle\Div_r\rho_\delta*y_nV^N(s)-\Div_ry_n\rho_\delta *   V^N(s), \rho_\delta *   V^N(s)\rangle   ds\\
    &=-\int_0^t\langle\rho_\delta*y_nV^N(s)-y_n\rho_\delta *   V^N(s), \rho_\delta *   \nabla_rV^N(s)\rangle   ds.
\end{align*}
On  the other   hand,   note    that
\begin{align*}
    &\big(\rho_\delta*y_nV^N\big)(s,x,r)-y_n(\rho_\delta *   V^N)(s,x,r)=\int_{\R^2}[y_n(x-y,s)-y_n(x,s)]\rho_\delta(y)V^N(s,x-y,r)dy.
\end{align*}
By   mean-value theorem we  get
\begin{align*}
    \vert   y_n(x-y,s)    -y_n(x,s)\vert &=\vert\displaystyle\sum_{k\in K_n}(g_k(\nabla \sigma_k^Nr)(x-y,s)-g_k(\nabla \sigma_k^Nr)(x,s))\vert\\
     &\leq   \displaystyle\sum_{k\in K_n}\vert   g_k\vert    \Vert   D^2 \sigma_k^N\Vert_\infty\vert  y\vert\vert r\vert\leq   \vert  y\vert\vert r\vert\Vert  g\Vert  \text{  since   }   \Vert   D^2 \sigma_k^N\Vert_\infty\leq  1,
\end{align*}
and 
\begin{align*}
    &\vert  \big(\rho_\delta*(\nabla u_L(\cdot)r)V^N\big)(s,x,r)-(\nabla u_L(s,x)r)(\rho_\delta *   V^N)(s,x,r)\vert    \\
    &\leq  \Vert   g\Vert \int_{\R^2}\vert    y\vert  \vert   r\vert  \rho_\delta(y)\vert   V^N(s,x-y,r)\vert       dy\leq  \delta\Vert   g\Vert \int_{\R^2}   \rho_\delta(y)\vert   r\vert \vert   V^N(s,x-y,r)\vert   dy,
\end{align*}
since   $\text{supp}[    \rho]\subset B(0,1).$    Therefore,  we  get
\begin{align*}
    &\int_0^t\vert\langle\rho_\delta*(\nabla u_L(s)r)V^N(s)-(\nabla u_L(s)r)\rho_\delta *   V^N(s), \rho_\delta *   \nabla_rV^N(s)\rangle \vert  ds\\
     &\leq    \delta\int_0^t\Vert   g\Vert   \Vert \rho_\delta *\vert   r\vert \vert   V^N(s)\vert    \vert    \Vert \Vert\rho_\delta *   \nabla_rV^N(s)\Vert ds\\
    &\leq    \delta\int_0^t\Vert   g\Vert    \Vert    V^N(s)        \Vert_H \Vert   \nabla_rV^N(s)\Vert ds\to    0   \text{  as    } \delta  \to 0.
\end{align*}

Now,    notice    that      $[\Div_rrV^N(s)]_{\delta}=\Div_rr[V^N(s)]_{\delta}$. Hence
\begin{align*}
    \dfrac{1}{\beta}\int_0^t\langle[\Div_rrV^N(s)]_{\delta}, [V^N(s)]_{\delta}\rangle   ds=\dfrac{1}{\beta}\int_0^t\Vert  [V^N(s)]_{\delta}\Vert^2ds\leq   \dfrac{1}{\beta}\int_0^t\Vert  V^N(s)\Vert^2ds.
\end{align*}
Next,   we  have    $\langle \sigma^2[\Delta_rV^N(s)]_{\delta}, [V^N(s)]_{\delta}\rangle   =\langle \sigma^2\Delta_r[V^N(s)]_{\delta}, [V^N(s)]_{\delta}\rangle   =-\sigma^2\Vert \nabla_r[V^N(s)]_{\delta}\Vert^2,   $ thus
\begin{align*}
    \int_0^t\langle \sigma^2[\Delta_rV^N(s)]_{\delta}, [V^N(s)]_{\delta}\rangle   ds=-\sigma^2\int_0^t\Vert [\nabla_rV^N(s)]_{\delta}\Vert^2   ds   \to -\sigma^2\int_0^t\Vert \nabla_rV^N(s)\Vert^2   ds  \text{  as  }   \delta  \to 0,
\end{align*}
since   $\nabla_rV^N\in L^2(0,T;H).$    We  have        $([\Delta_xV^N(s)]_\delta,[V^N(s)]_\delta) =
    (\Delta_x[V^N(s)]_\delta,[V^N(s)]_\delta)$  which   gives
\begin{align*}
    \int_0^t\alpha_N([\Delta_xV^N(s)]_\delta,[V^N(s)]_\delta) ds=-  \int_0^t\alpha_N\Vert\nabla_x[V^N(s)]_\delta\Vert^2 ds\leq 0.
\end{align*}

Concerning  the last    term,   first   we  prove   the following
\begin{align*}
     \int_0^t\langle [\Div_r\sum_{k\in K}\left((\nabla \sigma_k^Nr) \otimes (\nabla \sigma_k^Nr)\right) \nabla_rV^N(s))]_{\delta},[V^N(s)]_{\delta}\rangle    ds+\sum_{k\in K}\int_0^t\Vert[(\nabla \sigma_k^Nr).\nabla_rV^N(s)]_{\delta}\Vert^2ds\underset{\delta\to0}{\to}   0.
\end{align*}
Indeed, we  have  
\begin{align*}
     &\int_0^t\langle [\Div_r\sum_{k\in K}\left((\nabla \sigma_k^Nr) \otimes (\nabla \sigma_k^Nr)\right) \nabla_rV^N(s))]_{\delta},[V^N(s)]_{\delta}\rangle    ds+\sum_{k\in K}\int_0^t\Vert[(\nabla \sigma_k^Nr).\nabla_rV^N(s)]_{\delta}\Vert^2ds\\
     &=-\sum_{k\in K}\int_0^t\langle [(\nabla \sigma_k^Nr)  (\nabla \sigma_k^Nr)\cdot \nabla_rV^N(s)]_{\delta},\nabla_r[V^N(s)]_{\delta}\rangle    ds+\sum_{k\in K}\int_0^t\Vert[(\nabla \sigma_k^Nr).\nabla_rV^N(s)]_{\delta}\Vert^2ds\\
     &=-\sum_{k\in K}\int_0^t\langle [(\nabla \sigma_k^Nr)  (\nabla \sigma_k^Nr)\cdot \nabla_rV^N(s)]_{\delta}-(\nabla \sigma_k^Nr)  [(\nabla \sigma_k^Nr)\cdot \nabla_rV^N(s)]_{\delta},\nabla_r[V^N(s)]_{\delta}\rangle    ds\\
     &\qquad-\sum_{k\in K}\int_0^t\langle (\nabla \sigma_k^Nr)  [(\nabla \sigma_k^Nr)\cdot \nabla_rV^N(s)]_{\delta},\nabla_r[V^N(s)]_{\delta}\rangle    ds\\&\qquad+\sum_{k\in K}\int_0^t([(\nabla \sigma_k^Nr).\nabla_rV^N(s)]_{\delta},[(\nabla \sigma_k^Nr).\nabla_rV^N(s)]_{\delta})ds\\
     &=-\sum_{k\in K}\int_0^t\langle [(\nabla \sigma_k^Nr)  (\nabla \sigma_k^Nr)\cdot \nabla_rV^N(s)]_{\delta}-(\nabla \sigma_k^Nr)  [(\nabla \sigma_k^Nr)\cdot \nabla_rV^N(s)]_{\delta},\nabla_r[V^N(s)]_{\delta}\rangle    ds\\&\qquad-\sum_{k\in K}\int_0^t\langle   [(\nabla \sigma_k^Nr)\cdot \nabla_rV^N(s)]_{\delta},(\nabla \sigma_k^Nr)\cdot\nabla_r[V^N(s)]_{\delta}  -[(\nabla \sigma_k^Nr).\nabla_rV^N(s)]_{\delta})\rangle  ds\\
     &=J^1_\delta+J^2_\delta.\end{align*}
     Let    us  prove   that    $\displaystyle\limsup_{\delta\to  0}  \vert   J^1_\delta\vert=0.$   Recall  that    $(\nabla \sigma_k^Nr)$  has the form    $\theta_{\vert  k\vert}^N\vert  k\vert\dfrac{k}{\vert  k\vert  }\cdot  r    \dfrac{k^\perp}{\vert  k\vert  }   h(k\cdot x)$\footnote{$  h(\cdot
     )=\cos(\cdot)$ or  $    h(\cdot
     )=-\sin(\cdot)$,   see  \eqref{nabla-sigma}.}   for any $k\in       \mathbb{Z}^2_0$.  Thus,  $\vert   \nabla \sigma_k^Nr\vert\leq \theta_{\vert  k\vert}^N\vert  k\vert\vert  r\vert$. On the other   hand,   we  have 
     \begin{align*}
        \vert   J^1_\delta\vert \leq     \sum_{k\in K}\int_0^t\vert\langle [(\nabla \sigma_k^Nr)  (\nabla \sigma_k^Nr)\cdot \nabla_rV^N(s)]_{\delta}-(\nabla \sigma_k^Nr)  [(\nabla \sigma_k^Nr)\cdot \nabla_rV^N(s)]_{\delta},\nabla_r[V^N(s)]_{\delta}\rangle \vert    ds
     \end{align*}
     and   
     \begin{align*}   &\big([(\nabla \sigma_k^Nr)  (\nabla \sigma_k^Nr)\cdot \nabla_rf^N(s)]_{\delta}-(\nabla \sigma_k^Nr)  [(\nabla \sigma_k^Nr)\cdot \nabla_rV^N(s)]_{\delta}\big)(x,r)\\
     &=-\int_{\R^2}[(\nabla \sigma_k^N(x)r)-(\nabla \sigma_k^N(x-y)r) ]  (\nabla \sigma_k^N(x-y)r)\cdot \nabla_rV^N(s,x-y,r)\rho_\delta(y)dy.
     \end{align*}
  By mean-value theorem,     we  get 
   \begin{align*}
       \vert(\nabla \sigma_k^N(x)r)-(\nabla \sigma_k^N(x-y)r) \vert\leq  \theta_{\vert  k\vert}^N\vert  k\vert^2\cdot  \vert   r\vert    \vert y\vert\leq  2 N\vert k\vert\theta_{\vert  k\vert}^N\cdot  \vert   r\vert    \vert y\vert.
          \end{align*}
          Therefore
          \begin{align*}
               &\vert\big([(\nabla \sigma_k^Nr)  (\nabla \sigma_k^Nr)\cdot \nabla_rV^N(s)]_{\delta}-(\nabla \sigma_k^Nr)  [(\nabla \sigma_k^Nr)\cdot \nabla_rV^N(s)]_{\delta}\big)(x,r)\vert\\
     &\leq  \int_{\R^2}  2N\vert    k\vert\theta_{\vert  k\vert}^N\cdot  \vert   r\vert    \vert y\vert  \vert    (\nabla \sigma_k^N(x-y)r)\cdot \nabla_rV^N(s,x-y,r)\vert\rho_\delta(y)dy\\
      &\leq 2\delta  N(\theta_{\vert  k\vert}^N)^2\vert  k\vert^2\cdot  \vert   r\vert \int_{\R^2}       \vert    \vert r\vert  \vert\nabla_rV^N(s,x-y,r)\vert\rho_\delta(y)dy.
          \end{align*}
          Thus, we  get
          \begin{align*}&\vert\langle [(\nabla \sigma_k^Nr)  (\nabla \sigma_k^Nr)\cdot \nabla_rV^N(s)]_{\delta}-(\nabla \sigma_k^Nr)  [(\nabla \sigma_k^Nr)\cdot \nabla_rV^N(s)]_{\delta},\nabla_r[V^N(s)]_{\delta}\rangle\vert \\
          &=\vert\int_{\R^2}\int_{\T^2}[(\nabla \sigma_k^Nr)  (\nabla \sigma_k^Nr)\cdot \nabla_rV^N(s)]_{\delta}-(\nabla \sigma_k^Nr)  [(\nabla \sigma_k^Nr)\cdot \nabla_rV^N(s)]_{\delta}\nabla_r[V^N(s)]_{\delta}dxdr\vert\\
           &\leq\int_{\R^2}\int_{\T^2}\vert(\nabla \sigma_k^Nr)  (\nabla \sigma_k^Nr)\cdot \nabla_rV^N(s)]_{\delta}-(\nabla \sigma_k^Nr)  [(\nabla \sigma_k^Nr)\cdot \nabla_rV^N(s)]_{\delta}\vert\vert\nabla_r[V^N(s)]_{\delta}\vert   dxdr\\
           &\leq2\delta  N(\theta_{\vert  k\vert}^N)^2\vert  k\vert^2\int_{\R^2}\int_{\T^2}  \vert   r\vert \int_{\R^2}       \vert    \vert r\vert  \vert\nabla_rV^N(s,x-y,r)\vert\rho_\delta(y)dy\vert\nabla_r[V^N(s)]_{\delta}\vert   dxdr\vert\\
           &\leq2\delta  N(\theta_{\vert  k\vert}^N)^2\vert  k\vert^2\Vert  \rho_\delta*\vert   r\vert  \vert   \nabla_r    V^N(s)\vert\Vert_{L^2(\T^2\times\R^2)}^2\\&\leq    2\delta  N(\theta_{\vert  k\vert}^N)^2\vert  k\vert^2\Vert  \vert   r\vert  \vert   \nabla_r    V^N(s)\vert\Vert_{L^2(\T^2\times\R^2)}^2\leq    2\delta  N(\theta_{\vert  k\vert}^N)^2\vert  k\vert^2\Vert     \nabla_r    V^N(s)\Vert_{H}^2.
                         \end{align*}
                         Hence, we  deduce
                          $
       \displaystyle \vert   J^1_\delta\vert \leq    2\delta  N \sum_{k\in K}(\theta_{\vert  k\vert}^N)^2\vert  k\vert^2\int_0^t\Vert     \nabla_r    V^N(s)\Vert_{H}^2  ds    \to 0   \text{  as  }   \delta  \to 0.
     $
     Concerning $J^2_\delta$,   we  recall  that    \begin{align*}
         \vert  J^2_\delta\vert \leq    \sum_{k\in K}\int_0^t\vert    \langle   [(\nabla \sigma_k^Nr)\cdot \nabla_rV^N(s)]_{\delta},(\nabla \sigma_k^Nr)\cdot\nabla_r[V^N(s)]_{\delta}    -[(\nabla \sigma_k^Nr).\nabla_rV^N(s)]_{\delta})\rangle\vert    ds.
     \end{align*}
     On the other   hand,   we  have  \begin{align*}
         &\vert(\nabla \sigma_k^Nr)\cdot\nabla_r[V^N(s)]_{\delta}    -[(\nabla \sigma_k^Nr).\nabla_rV^N(s)]_{\delta}\vert\\&=\vert\int_{\R^2}[(\nabla \sigma_k^N(x)r)-(\nabla \sigma_k^N(x-y)r) ]  \cdot \nabla_rV^N(s,x-y,r)\rho_\delta(y)dy\vert\\
         &\leq   \delta \vert  k\vert^2\theta_{\vert  k\vert}^N \int_{\R^2}\vert   r\vert  \vert\nabla_rV^N(s,x-y,r) \vert  \rho_\delta(y)dy,     \end{align*}
         consequently,  we  get 
         \begin{align*}
             &\vert    \langle   [(\nabla \sigma_k^Nr)\cdot \nabla_rf^N(s)]_{\delta},(\nabla \sigma_k^Nr)\cdot\nabla_r[V^N(s)]_{\delta}    -[(\nabla \sigma_k^Nr).\nabla_rV^N(s)]_{\delta})\rangle\vert\\
             &\leq  \int_{\R^2}\int_{\T^2}\vert      [(\nabla \sigma_k^Nr)\cdot \nabla_rV^N(s)]_{\delta}\vert   \vert   (\nabla \sigma_k^Nr)\cdot\nabla_r[V^N(s)]_{\delta}    -[(\nabla \sigma_k^Nr).\nabla_rV^N(s)]_{\delta})\vert   dxdr\\
             &\leq  \delta \vert  k\vert^2\theta_{\vert  k\vert}^N\int_{\R^2}\int_{\T^2}\vert    [(\nabla \sigma_k^Nr)\cdot \nabla_rV^N(s)]_{\delta}\vert  \int_{\R^2}\vert   r\vert  \vert\nabla_rV^N(s,x-y,r) \vert  \rho_\delta(y)dy\vert   dxdr\\
             &\leq  \delta \vert  k\vert^3(\theta_{\vert  k\vert}^N)^2\int_{\R^2}\int_{\T^2}\vert  \int_{\R^2}\vert   r\vert  \vert\nabla_rV^N(s,x-y,r) \vert  \rho_\delta(y)dy\vert^2   dxdr\\
             &\leq  \delta \vert  k\vert^3(\theta_{\vert  k\vert}^N)^2\Vert\vert  \int_{\R^2}\vert   r\vert  \vert\nabla_rV^N(s,x-y,r) \vert  \rho_\delta(y)dy\vert\Vert^2_{L^2(\T^2\times\R^2)}\\&\leq   \delta \vert  k\vert^3(\theta_{\vert  k\vert}^N)^2\Vert\nabla_rV^N(s) \vert  \Vert^2_{H}\leq   2\delta N\vert  k\vert^2(\theta_{\vert  k\vert}^N)^2\Vert\nabla_rV^N(s) \vert  \Vert^2_{H}
         \end{align*}
         and    
               \begin{align*}
        \vert   J^2_\delta\vert \leq    2\delta  N \sum_{k\in K}(\theta_{\vert  k\vert}^N)^2\vert  k\vert^2\E\int_0^t\Vert     \nabla_r    V^N(s)\Vert_{H}^2  ds    \to 0   \text{  as  }   \delta  \to 0.
     \end{align*}
     Finally,   we  get
     \begin{align*}
   \lim_{\delta\to0}  &\int_0^t\langle [\Div_r\sum_{k\in K}\left((\nabla \sigma_k^Nr) \otimes (\nabla \sigma_k^Nr)\right) \nabla_rV^N(s))]_{\delta},[V^N(s)]_{\delta}\rangle    ds\\&=-\sum_{k\in K}\int_0^t\Vert(\nabla \sigma_k^Nr).\nabla_rV^N(s)\Vert^2ds \leq  0,
\end{align*}
where   we  used $\nabla_rV^N\in     L^2(0,T;H))$
and \begin{align*}\sum_{k\in K}\int_0^t\Vert[(\nabla \sigma_k^Nr).\nabla_rV^N(s)]_{\delta}\Vert^2ds\to_{\delta\to0}   \sum_{k\in K}\int_0^t\Vert(\nabla \sigma_k^Nr).\nabla_rV^N(s)\Vert^2ds.
    \end{align*} 
In conclusion,  by  passing to  the limit   as  $\delta\to0$    in  \eqref{Vn-regularized-uniq}  and using   the above   estimates,   we  get
 \begin{align}\label{limit-delta-inequality}&\dfrac{1}{2} \Vert V^N(t)\Vert^2+\sigma^2\int_0^t\Vert \nabla_rV^N(s)\Vert^2   ds	\leq  \dfrac{1}{\beta}\int_0^t\Vert  V^N(s)\Vert^2ds.    \end{align}
 The    last    inequality  \eqref{limit-delta-inequality}  and Grönwall    lemma   ensure  that    $V^N\equiv  0$  in  $L^\infty(0,T;H)$-sense   and $\nabla V^N\equiv  0$  in  $L^2(0,T;H)-$sense as  well,    which   ends    the proof   of  uniqueness.

   \subsubsection{Uniqueness  of    quasi-regular
weak solutions  of  \eqref{Ito-FP}}\label{Subsection-uniq-quasi-weak}
We  will    use \autoref{lem-uniq-Vn}   to  prove   that    the solution    to  \eqref{Ito-FP}  in  the sense   of  \autoref{Def-sol}   is  unique  (in particular  we  use the point   (4)). First,  
note that the set of \textit{quasi-regular weak} solutions forms a
linear subspace of $L^2 (0, T ;H)$, since \eqref{Ito-FP}
is a    linear  equation, and the regularity conditions is a linear constraint. Therefore, it is enough
to show that a quasi-regular weak solution $f^N\equiv    0$ if the   initial data $f_0\equiv 0.$
\\

Let $g\in   G_n$,  by   using   \autoref{lem-uniq-Vn} we  proved  that    $\E[f^N(t)e_g(t)]=0$    in  $L^\infty(0,T;H)$-sense.  Our aim is  to  prove   that    $f^N\equiv       0$.    We  recall    that  from  \autoref{lem-uniq-Vn}:   for any $t\in[0,T],$    we  have
\begin{align*}
    (\E[f^N(t)e_g(t)],\varphi)=0,    \forall \varphi\in  V  \text{  and for any }   g\in    \mathcal{D}.
\end{align*}
Now, let  $G$   be  a     random  variable,        which   can be  written as  a   linear  combination of  finite  number  of  $e_g(t)$, it follows (by linearity)
\begin{align*}
    (\E[f^N(t)G],\varphi)=0,    \forall \varphi\in  V. 
\end{align*}  Next,  by  density of  $\mathcal{D}$   in  $L^2(\Omega,\overline{\mathcal{G}_t}),$ the last    equality    holds   for any $G\in   L^2(\Omega,\overline{\mathcal{G}_t}),$  namely
\begin{align*}
    \E[(f^N(t),\varphi)G]=0,    \forall \varphi\in  V, \quad \forall G\in   L^2(\Omega,\overline{\mathcal{G}_t}). 
\end{align*}
Since   $(f^N(t),\varphi)$  is  $\overline{\mathcal{G}_t}$-adapted, we  get
$    (f^N(t),\varphi)=0$, for any $\varphi\in V$.  Recall   that   $V$ is  dense  subspace   in  $H$,  thus,  we  deduce    that  $f^N\equiv  0$    and the uniqueness  holds.

    \section{Diffusion  scaling limit    as  $N\to   +\infty$}\label{Section-diffusion}
      Our aim in  this    section is  to  show    that    the unique  solution    of   stochastic  FP  equation    \eqref{Ito-FP},    in  the sense   of  \autoref{Def-sol},  converges   weakly  to  the unique   solution    of \eqref{equ-final},    under   the following   scaling    of  the noise coefficients:
   $\theta^N_{\left\vert k\right\vert}=\frac{a}{\left\vert k\right\vert^2 }$ if $ N \leq \left\vert k\right\vert \leq  2N$ and $\theta^N_{\left\vert k\right\vert}=0$ else. First,  note    that   
  \begin{align}\label{coefficient}    \lim_{N\to +\infty}\sum_{k\in K_{++}}\left(  \theta_{\left\vert k\right\vert }^{N}\right)^{2}=0;\quad 
 \lim_{N\to +\infty}\sup_{k\in K}\left(  \theta_{\left\vert k\right\vert }^{N}\right)  ^{2}= 0 \text{ and }
 \lim_{N\to +\infty}\sup_{k\in K}\left\vert k\right\vert
 ^{2}\left(  \theta_{\left\vert k\right\vert }^{N}\right)  ^{2}= 0.
 \end{align}
 
\subsection{Time regularity and tightness of laws}
 We begin  by showing the following result about regularity in time  of  $(f^N)_N$.
\begin{lemma}\label{Lemma-tight1}
        Let $2<p<+\infty$,  there   exists    $C>0$,   independent of  $N$  such    that
   \begin{align*}
       \E \Vert f^N\Vert_{\mathcal{C}^{\eta}([0,T],U^\prime)}^p  \leq    C  \text{  for   any   }  0<\eta<\min(\dfrac{p-2}{p},\dfrac{1}{2}).
   \end{align*}
\end{lemma}
\begin{proof}
Let $0<h<1$ and $t\in[0,T-h]$.    From	\autoref{Def-sol},  point   (3),    the following   equality    holds   in  $U^\prime$-sense:		
 \begin{align*}&f^N(t+h)-f^N(t)=-\int_t^{t+h}((u_L\cdot \nabla_xf^N(s))-\Div_r[(\nabla_x u_Lr-\dfrac{1}{\beta}r)f^N(s)])ds\\	&+	\int_t^{t+h}[\sigma^2 \Delta_rf^N(s)+\alpha_N(\Delta_xf^N(s)]ds+ \dfrac{1}{2}\int_t^{t+h} \Div_r(A_k^N\nabla_rf^N(s)) ds\\&-\sum_{k\in K}\int_t^{t+h}(\sigma_k^N.\nabla_xf^N(s)+(\nabla \sigma_k^Nr).\nabla_rf^N(s))dW^k(s):=I_1+I_2+I_3+I_4.\end{align*}
On the one hand, we have	P-a.s.
\begin{align*}
\Vert   I_1\Vert_{U^\prime}&\leq \int_t^{t+h}\Vert	u_L\cdot \nabla_xf^N(s)+\Div_r[(\nabla_x u_Lr-\dfrac{1}{\beta}r)f^N(s)]\Vert_{U^\prime}ds\\
&\leq	[\Vert    u_L\Vert_{\infty}+\Vert    \nabla_x u_L\Vert_{\infty}+\dfrac{1}{\beta}]\int_t^{t+h}\Vert	f^N(s)\Vert_{H}ds\\
&\leq   h[\Vert    u_L\Vert_{\infty}+\Vert    \nabla_x u_L\Vert_{\infty}+\dfrac{1}{\beta}]\sup_{t \in [0,T]}\Vert	f^N(t)\Vert_{H},
\end{align*}
since	$u_L$	is	a	smooth	function.     Concerning    $I_2$,	note	that
	\begin{align*}
\Vert   I_2\Vert_{U^\prime} &\leq	
\int_t^{t+h}\Vert	\sigma^2 \Delta_rf^N(s)+\alpha_N\Delta_xf^N(s)\Vert_{U^\prime}ds\leq(\sigma^2+\alpha_N)	\int_t^{t+h}\Vert	f^N(s)\Vert_{H}ds\\&\leq(\sigma^2+1)	\int_t^{t+h}\Vert	f^N(s)\Vert_{H}ds\leq    h(\sigma^2+1)\sup_{t \in [0,T]}\Vert	f^N(t)\Vert_{H}.
\end{align*}
Moreover    $
	\displaystyle\Vert   I_3\Vert_{U^\prime}\leq	
	\int_t^{t+h}\Vert	\Div_r(A_k^N\nabla_rf^N(s)\Vert_{U^\prime}ds,
 $  we    recall  that 
 \begin{align*}
     \Vert	\Div_r(A_k^N\nabla_rf^N(s)\Vert_{U^\prime}=\sup_{\Vert	\phi\Vert_U	\leq	1}\vert \langle \Div_r(A_k^N\nabla_rf^N(s),\phi\rangle\vert
 \end{align*}
 but    $
     \langle \Div_r(A_k^N\nabla_rf^N),\phi\rangle=( \Div_r(A_k^N\nabla_rf^N),\phi)=( f^N,\Div_r(A_k^N\nabla_r\phi)),
 $
 therefore
 \begin{align*}
 \int_t^{t+h}\Vert	\Div_r(A_k^N\nabla_rf^N(s)\Vert_{U^\prime}ds &\leq	\sum_{k\in K}\left\vert k\right\vert
	^{2}\left(  \theta_{\left\vert k\right\vert }^{N}\right)  ^{2}\int_t^{t+h}\Vert	f^N(s)\Vert_{H}ds\\
 &\leq   h\sup_{t \in [0,T]}\Vert	f^N(t)\Vert_{H}\sum_{k\in K}\left\vert k\right\vert
	^{2}\left(  \theta_{\left\vert k\right\vert }^{N}\right)  ^{2}.
		\end{align*}

Concerning	the	stochastic	integral  $I_4$,	let	$\phi\in	U	$	and	note	that		
	\begin{align*}
\sum_{k\in K}&\int_t^{t+h}\langle(\sigma_k^N.\nabla_xf^N(s)+(\nabla \sigma_k^Nr).\nabla_rf^N(s),\phi\rangle   dW^k(s)\\&=
\sum_{k\in K}\int_t^{t+h}\langle\sigma_k^N.\nabla_xf^N(s)+(\nabla \sigma_k^Nr).\nabla_rf^N(s),\phi\rangle dW^k(s)\\&=-
\sum_{k\in K}\int_t^{t+h}(f^N(s),\sigma_k^N.\nabla_x\phi+(\nabla \sigma_k^Nr).\nabla_r\phi)dW^k(s).\end{align*}
We  recall
	\begin{align*}
\Vert\sum_{k\in K}&\int_t^{t+h}(\sigma_k^N.\nabla_xf^N(s)+(\nabla \sigma_k^Nr).\nabla_rf^N(s))dW^k(s)\Vert_{U^\prime}\\&=\sup_{\Vert	\phi\Vert_U	\leq	1}\vert\sum_{k\in K}\int_t^{t+h}(f^N(s),\sigma_k^N.\nabla_x\phi+(\nabla \sigma_k^Nr).\nabla_r\phi)dW^k(s)\vert.
	\end{align*}
	Let    $\phi   \in U$  such    that
 	$\Vert	\phi\Vert_U	\leq	1$    and $1<p<+\infty$.
  Burkholder-Davis-Gundy inequality  	ensures
	\begin{align*}
&
\E\vert\sum_{k\in K}\int_t^{t+h}(f^N(s),\sigma_k^N.\nabla_x\phi+(\nabla \sigma_k^Nr).\nabla_r\phi)dW^k(s)\vert^p\\&\leq	
\E[\int_t^{t+h}\sum_{k\in K}(f^N(s),\sigma_k^N.\nabla_x\phi+(\nabla \sigma_k^Nr).\nabla_r\phi)^2ds]^{p/2}\\
&\leq	[\sum_{k\in K}\left(\vert k\right\vert
^{2}+1)\left(  \theta_{\left\vert k\right\vert }^{N}\right)^{2}]^{p/2}\E[\int_t^{t+h}(f^N(s),\nabla_x\phi+\vert	r\vert\vert\nabla_r\phi\vert)^2ds]^{p/2}
\\&\leq	C_1\E[\int_t^{t+h}\Vert	f^N(s)\Vert^2\Vert	\phi\Vert_U^2ds]^{p/2}\leq	C_1\E[\int_t^{t+h}\Vert	f^N(s)\Vert^2\Vert	\phi\Vert_U^2ds]^{p/2}\leq	C_1\E[\int_t^{t+h}\Vert	f^N(s)\Vert^2ds]^{p/2},\end{align*}
where   $C_1:=[\sum_{k\in K}\left(\vert k\right\vert
^{2}+1)\left(  \theta_{\left\vert k\right\vert }^{N}\right)^{2}]^{p/2}.$
Thus,	we	obtain
\begin{align}\label{stoch-BDG}
   \E \Vert\sum_{k\in K}\int_t^{t+h}(\sigma_k^N.\nabla_xf^N(s)+(\nabla \sigma_k^Nr).\nabla_rf^N(s))dW^k(s)\Vert_{U^\prime}^p &\leq    C_1\E[\int_t^{t+h}\Vert	f^N(s)\Vert^2ds]^{p/2}\notag\\
   &\leq    C_1h^{p/2}\E\sup_{q\in[0,T]}\Vert	f^N(q)\Vert^p.\end{align}
   Recall   the definition  of      $W^{s,p}(0,T; U^\prime),$       the   Sobolev space of all   $u\in  L^p(0,T; U^\prime) $    such    that
   $$   \int_0^T\int_0^T\dfrac{\Vert    u(t)-u(r)\Vert_{U^\prime}^p}{\vert t-r\vert^{1+sp}}dtdr<+\infty,$$
   endowed  with    the norm
   $$   \Vert   u\Vert_{W^{s,p}(0,T; U^\prime)}^p=\Vert   u\Vert_{L^{p}(0,T; U^\prime)}^p+\int_0^T\int_0^T\dfrac{\Vert    u(t)-u(r)\Vert_{U^\prime}^p}{\vert t-r\vert^{1+sp}}dtdr.$$
   Now,     denote  by  $\displaystyle  I(f^N)(\cdot)=\sum_{k\in K}\int_0^{\cdot}(\sigma_k^N.\nabla_xf^N(s)+(\nabla \sigma_k^Nr).\nabla_rf^N(s))dW^k(s).$\\

   Thanks   to  \eqref{borne-p-limit}   and Burkholder-Davis-Gundy inequality,  one has 
   \begin{align*}
       \E\Vert   I(f^N)\Vert_{L^{p}(0,T; U^\prime)}^p\leq \mathbf{M},\quad \mathbf{M}>0    \text{  independent   of  }   m    \text{   and    } N .
   \end{align*}
   Concerning   the second  part,   note    that
\begin{align*}
   \E \int_0^T\int_0^T\dfrac{\Vert    I(f^N)(t)-I(f^N)(r)\Vert_{U^\prime}^p}{\vert t-r\vert^{1+sp}}dtdr&= \int_0^T\int_0^T\dfrac{\E\Vert    I(f^N)(t)-I(f^N)(r)\Vert_{U^\prime}^p}{\vert t-r\vert^{1+sp}}dtdr\\
   &\leq     C_1\E\sup_{q\in[0,T]}\Vert	f^N(q)\Vert^p\int_0^T\int_0^T\dfrac{\vert t-r\vert^{p/2}}{\vert t-r\vert^{1+sp}}dtdr\\ &\leq     C_1\E\sup_{q\in[0,T]}\Vert	f^N(q)\Vert^p\int_0^T\int_0^T\vert t-r\vert^{p(\frac{1}{2}-s)-1}dtdr\leq  C,
\end{align*}
if   $p(\frac{1}{2}-s)>0$,  which   holds for  any $s\in]0,\frac{1}{2}[.$\\

Let $p>2$,  then    $   0<sp-1<\dfrac{p-2}{2}$. Denote  by  $\mathcal{C}^{\eta}([0,T],U^\prime)$    the space   of  $\eta$-Hölder continuous   functions    with    values  in  $U^\prime.$
and we recall that  (see e.g. \cite{Flan-Gater}) $$ W^{s,p}(0,T; U^\prime) \hookrightarrow \mathcal{C}^{\eta}([0,T],U^\prime)\quad  \text{if} \quad0<\eta < sp-1.$$
Let us take $s \in  \big[0,\dfrac{1}{2}\big[$ such  that     $sp>1$. For $\eta \in \big]0,  sp-1\big[$, it follows from    the previous   estimates
      \begin{align*}
       &   \E \Vert\sum_{k\in K}\int_0^{\cdot}(\sigma_k^N.\nabla_xf^N(s)+(\nabla \sigma_k^Nr).\nabla_rf^N(s))dW^k(s)\Vert_{W^{s,p}(0,T;U^\prime)}^p \leq C+\mathbf{M}.
      \end{align*}
   Thus $
       \E \Vert I(f^N)\Vert_{\mathcal{C}^{\eta}([0,T],U^\prime)}^p  \leq    C+\mathbf{M}.
  $
Consequently,   we  obtain  
   \begin{align*}
       \E \Vert f^N\Vert_{\mathcal{C}^{\eta}([0,T],U^\prime)}^p  \leq    C+\mathbf{M}   \text{  for   any   }  0<\eta<\min(\dfrac{p-2}{p},\dfrac{1}{2}),
   \end{align*}
which	gives	that	$(f^N)_N$	is	bounded	in	$L^p(\Omega,C^\eta([0,T],U^\prime)),	\eta\in]0,\min(\frac{p-2}{p},\frac{1}{2})[$   with    $2<p<+\infty.$

\end{proof}
As    a   consequence of   \autoref{Lemma9},  \autoref{Lemma10},    \autoref{Lemma-tight1}   and the lower   semi-continuity  of  the weak    convergence,    we  obtain      \begin{proposition}\label{Prop-boun-} 	
 	Let   $p\geq  2$, there exists $\mathbf{K}>0$ such that the unique solution $(f^N)_N$ to  \eqref{Ito-FP} is bounded    by $\mathbf{K} $ in  
    $$L^p(\Omega,L^\infty(0,T;H))\cap L^p(\Omega,C^\eta([0,T],U^\prime)),	\quad 0<\eta<\min(\dfrac{p-2}{p},\dfrac{1}{2}).$$
    Moreover,   $(\nabla_rf^N)_N $  is  bounded in  $L^2(\Omega,L^2(0,T;H)).$
   \end{proposition} 
\subsubsection{Tightness of laws}
Denote by $\mu_{f^N}$ the law of $f^N$, $\mu_W$ the law of $W:=(W^k)_{k\in\mathbb{Z}^2_0}$ and  their   joint   law $\mu_N$ defined on  $ C([0,T], U^\prime)\times  C([0,T]; H_0)$.
\begin{lemma}\label{tight-2}
	The set $\{ \mu_{f^N}; N\in \mathbb{N}\}$ is tight on $C([0,T], U^\prime)$.
\end{lemma}
\begin{proof}
First,  we have $H \underset{compact}{\hookrightarrow} U^\prime$,   since   	$U \underset{compact}{\hookrightarrow} H$.
Let $\mathbf{A}$    be  a   subset  $C([0,T]; U^\prime)$.    Following   	\cite[Thm. 3]{Simon}	(the	case	$p=\infty$),	$\mathbf{A}$  is relatively compact  in    $C([0,T]; U^\prime)$ if  the following    conditions  hold.
\begin{enumerate}
	\item	$\mathbf{A}$	is	bounded	in	$L^\infty(0,T; H)$.
	\item	Let	$h>0$,	$\Vert	f(\cdot+h)-f(\cdot)\Vert_{L^\infty(0,T-h;U^\prime)}	\to	0$	as	$h\to0$	uniformly	for	$f\in	\mathbf{A}$.
\end{enumerate}   The   following   embedding  is  compact
$$ \mathbf{Z}:=L^\infty(0,T; H)\cap \mathcal{C}^{\eta}([0,T],U^\prime) \hookrightarrow \mathcal{C}([0,T], U^\prime),    \quad   0<\eta<\min(\dfrac{p-2}{p},\dfrac{1}{2}).$$
Indeed,	 Let	$	\mathbf{A}$	be	a	bounded	set	of	$\mathbf{Z}$.	
First,	note	that	(1)	is	satisfied	by	assumptions.	Concerning	the	second	condition,	let	$h>0$	and		$f\in	\mathbf{A}$,	by	using	that			$  f\in C^{\eta}([0,T],U^\prime)$	we	infer
\begin{align*}
	\Vert	f(\cdot+h)-f(\cdot)\Vert_{L^\infty(0,T-h;U^\prime)}=\sup_{r\in [0,T-h]}\Vert	f(r+h)-f(r)\Vert_{U^\prime}\leq	Ch^{\eta}\to	0,	\text{	as	}	h\to	0,
\end{align*}
where	$C>0$  is	independent	of	$f$. 
    Let $R>0$ and  set $ B_{\mathbf{Z}}(0,R):=\{ v\in \mathbf{Z} \; \vert\; \Vert v\Vert_{\mathbf{Z}} \leq R\}$. Then $B_{\mathbf{Z}}(0,R)$ is a compact  subset of  $C([0,T], U^\prime)$, by using      \autoref{Prop-boun-}  the following  holds
\begin{align*}
	\mu_{f^N}(B_{\mathbf{Z}}(0,R))&=1-\mu_{f^N}(B_{\mathbf{Z}}(0,R)^c)=1-\int_{\{ \omega \in \Omega, \Vert f^N\Vert_{\mathbf{Z}} >R\}}1dP\\
	&\geq 1-\dfrac{1}{R^p}\int_{\{ \omega \in \Omega, \Vert f^N\Vert_{\mathbf{Z}} >R\}}\Vert f^N\Vert_{\mathbf{Z}}^pdP\\
	&\geq 1-\dfrac{1}{R^p}\E\Vert f^N\Vert_{\mathbf{Z}}^p=1-\dfrac{\mathbf{K}^p}{R^p},\quad  \text{for any}\;  R>0, \quad \text{and any} 
	\;N\in \mathbb{N}.
\end{align*}
Therefore,  for any $\delta>0$ we can find $R_\delta >0$ such that $$\mu_{f^N}(B_{\mathbf{Z}}(0,R_\delta)) \geq 1-\delta, \text{ for all   } N\in \mathbb{N}. $$
Thus the family of laws $\{ \mu_{f^N}; N\in \mathbb{N}\}$ is tight on $C([0,T], U^\prime).$
\end{proof}
Notice    that    the family of  Brownian    motions  $W:=(W^k)_{k\in\mathbb{Z}^2_0}$   can  be  seen    as  cylindrical Wiener  process   defined on the	filtred	probability	space $(\Omega,\mathcal{F},P;(\mathcal{F}_t)_t)$ with    values  in appropriate separable   Hilbert  space   $H_0$,  more    precisely  $W^k=We_k, k\in\mathbb{Z}^2_0,$ where   $(e_k)_{k\in\mathbb{Z}^2_0}$   is  complete orthonormal system in a separable Hilbert space $\mathbb{H}$ and recall that the sample paths  of $W$ take values in a larger Hilbert space $H_0$ such that   $\mathbb{H}\hookrightarrow H_0$ defines a Hilbert–Schmidt  embedding.   
Hence,  $P$-a.s. the trajectories of $W$ 
belong to the space $C([0,T],H_0),$ see e.g.  \cite[Chapter 4]{Daprato}.   By
taking into account that the law $\mu_{W}$  is a Radon measure on  $C([0,T]; H_0)$,	we	obtain\begin{lemma}\label{tight-1}
	The set  $\{\mu_{W}\}$ is tight on   $C([0,T]; H_0)$.
\end{lemma}

\subsubsection{Prokhorov  and Skorokhod's representation's theorem}\label{Subsection-Prokh-SKo}
Thanks to  \autoref{tight-2}    and \autoref{tight-1}, by Skorokhod's representation's theorem (see    e.g.       \cite[Thm. 1.10.4, p.
59]{Vaart-Well1996}), by  passing to    the limit   up  to subsequences (denoted    by  the same    way), we can find a    new  probability space,     denoted by  the same    way  "for simplicity"    $(\Omega,\mathcal{F},P)$   and processes \begin{align*}
    \left(\widetilde{f}^N,\ W^N:=\{W^{N,k}\}_{\substack{k\in \mathbb{Z}^2_0}}\right),\quad   \left(\overline{f},\ \overline{W}:=\{\overline{W}^{k}\}_{\substack{k\in \mathbb{Z}^2_0}}\right),
\end{align*} such that: 
\begin{itemize}
\item[$(a)$]    $ \mathcal{L}\left(\widetilde{f}^N,\ W^N:=\{W^{N,k}\}_{\substack{k\in \mathbb{Z}^2_0}}\right)=\mathcal{L} \left(f^N,\ W:=\{W^{k}\}_{\substack{k\in \mathbb{Z}^2_0}}\right)$\footnote{Given a random variable $\xi$ with values in space $E$,  its law
	$
	\mathcal{L}(\xi)(\Gamma)=P(\xi\in\Gamma)$  for any Borel subset $\Gamma \text{ of } E.
	$} on  $C([0,T], U^\prime)\times   C([0,T], H_0)$.
\item[$(b)$] the   following   convergences hold
\begin{align}
    &\widetilde{f}^N\rightarrow \overline{f} \quad \text{in } C([0,T];U^\prime) \quad P-a.s.\label{cv-fN-time}\\
   & W^N\rightarrow \overline{W} \quad \text{in } C([0,T];H_0 )\quad P-a.s.
\end{align}
\end{itemize}
On  the other   hand,   thanks  to  \autoref{Prop-boun-},   $\mathcal{L}(f^N)(\mathcal{X} )=1$ where $\mathcal{X}    \underset{cont.}{\hookrightarrow} C([0,T], U^\prime)$,  with
$$  \mathcal{X}=\{g: \quad  g\in   L^\infty(0,T;H)\cap  C^\eta([0,T],U^\prime);    \nabla_r  g\in L^2(0,T;H)   \}.$$
 	 By   using   the point   $(a)$   above,  one gets    $\mathcal{L}(\widetilde{f}^N)(\mathcal{X})=1 $ and      $(\widetilde{f}^N,W^N)$   satisfies   the point   (3) of  \autoref{Def-sol}.    Moreover,   $(\widetilde{f}^N)_N$ satisfies   the estimates   of \autoref{Prop-boun-}  in  the new probability space.   Thus,
we  have       the following   result.
 \begin{lemma}\label{lem-cv-1} 
 	There exists $\overline{f}\in L^2(\Omega ;L^2([0, T];H)),$    $\nabla_r\overline{f}\in L^2(\Omega ;L^2([0, T];H))$   
 	such that
 	   \begin{align*}
 	\widetilde{f}^N \rightharpoonup \overline{f} &\text{		in	}  L^2(\Omega ;L^2([0, T];H)),\\
   \nabla_r\widetilde{f}^N \rightharpoonup \nabla_r\overline{f} &\text{		in	}  L^2(\Omega ;L^2([0, T];H)).
 	\end{align*}
  Moreover,    $ \overline{f}\in  L^2_{w-*}(\Omega ;L^\infty([0, T];H))$ and $\widetilde{f}^N \rightharpoonup^* \overline{f} \text{		in	}  L^2_{w-*}(\Omega ;L^\infty([0, T];H)).$
 	\end{lemma}
  \begin{proof}
 By using       \autoref{Prop-boun-},   diagonal extraction argument    and
   Banach–Alaoglu theorem   
  in the spaces $L^2(\Omega ;L^2([0, T];H))$    and $L^2_{w-*}(\Omega ;L^\infty([0, T];H)),$
   there exist $\overline{f},\nabla_r\overline{f} \in L^2(\Omega ;L^2([0, T];H))$ such that
    \begin{align}\label{cv1}
  \widetilde{f}^N \rightharpoonup \overline{f} &\text{		in	}  L^2(\Omega ;L^2([0, T];H)),
  \\
   \nabla_r\widetilde{f}^N \rightharpoonup \nabla_r\overline{f} &\text{		in	}  L^2(\Omega ;L^2([0, T];H)).
  \end{align}
  Moreover, $\overline{f}   \in L^2_{w-*}(\Omega ;L^\infty([0, T];H)))$ and the following   convergence holds
  \begin{align}\label{cv2}
   \widetilde{f}^N \rightharpoonup^* \overline{f} &\text{		in	}  L^2_{w-*}(\Omega ;L^\infty([0, T];H)),
   \end{align}  
   since  $(\widetilde{f}^N)_N$ is bounded in $L^2(\Omega ;L^\infty([0, T];H))$ and $ L^2_{w-*}(\Omega ;L^\infty([0, T];H))\simeq \left(L^2(\Omega ;L^1([0, T];H^\prime)) \right)^\prime .$  
   
       \end{proof} 	
\subsection{Passage to the limit as $N\to+\infty$}
        In the following,   we  will    establish   some    lemmas  to  pass    to  the limit   as  $N\to   +\infty.$   For that,     consider  $\phi  \in C^\infty(\mathbb{T}^2)$,  $\psi  \in C^\infty_c(\mathbb{R}^2)$\footnote{Note that         $\phi\otimes    \psi    \in V,   $    thus    it  is  an  appropriate test    function    in  \autoref{Def-sol},  point   (3). },   and $A\in \mathcal{F}$.
      
 \begin{lemma}\label{lem-cv-2}
 	Let   $t\in]0,T[,$    the   following  convergence holds:    $$ \displaystyle\lim_N\alpha_N\int_A\int_0^t\int_{\mathbb{T}^2}\int_{\mathbb{R}^2}\nabla_x\widetilde{f}^N(x,r,s) \cdot\nabla_x\phi(x)\psi(r) dr dxdsdP=0.
 $$
 \end{lemma}
  \begin{proof} 	
  	Thanks to Fubini's theorem and by using integration by part with respect to $x$, we get
  	\begin{align*}
  	&\vert \alpha_N\int_A\int_0^t\int_{\mathbb{T}^2}\int_{\mathbb{R}^2}\nabla_x\widetilde{f}^N(x,r,s) \cdot\nabla_x\phi(x)\psi(r) dr dxdsdP\vert\\&=\vert\alpha_N\int_A\int_0^t\int_{\mathbb{R}^2}\int_{\mathbb{T}^2}\widetilde{f}^N(x,r,s) \Delta_x\phi(x)\psi(r)  dxdrdsdP\vert.
  	\\&\qquad\leq C\alpha_N\Vert \widetilde{f}^N\Vert_{L^2(\Omega\times[0, T];L^2(\mathbb{T}^2\times\mathbb{R}^2)))}\Vert \Delta_x\phi\Vert_{\infty}\Vert \psi\Vert_{\infty},
  	\end{align*}
  	where $C$ is a constant depends only on the measure of $supp(\psi), T$  and the volume of $\mathbb{T}^2.$ We recall  from \eqref{alpha-N}    that    $\displaystyle\lim_N\alpha_N=0$ and  the result follows. 
  \end{proof} 	

  \begin{lemma}\label{lem-cv-3}
 	Let   $t\in]0,T[,$    we have \begin{align*}
 	&\lim_N- \dfrac{1}{2}\int_A\int_0^t\int_{\mathbb{T}^2}\int_{\mathbb{R}^2}(\sum_{k\in K}\left((\nabla \sigma_k^Nr) \otimes (\nabla \sigma_k^Nr)\right) \nabla_r\widetilde{f}^N(x,r,s))\cdot\nabla_r\psi(r) \phi(x) dr dxdsdP\\&=-\dfrac{1}{2}\int_A\int_0^t\int_{\mathbb{T}^2}\int_{\mathbb{R}^2}A(r) \nabla_r\overline{f}(x,r,s))\cdot\nabla_r\psi(r) \phi(x)  dr dxdsdP.
 	\end{align*}
 \end{lemma}
  \begin{proof} 
  	By using  	 \autoref{lem-matrix-r}, we get
  	\begin{align}\label{cv-3-*}
  	&- \dfrac{1}{2}\int_A\int_0^t\int_{\mathbb{T}^2}\int_{\mathbb{R}^2}(\sum_{k\in K}\left((\nabla \sigma_k^Nr) \otimes (\nabla \sigma_k^Nr)\right) \nabla_r\widetilde{f}^N(x,r,s))\cdot\nabla_r\psi(r) \phi(x) dr dxdsdP\notag\\&=-\dfrac{1}{2}\int_A\int_0^t\int_{\mathbb{T}^2}\int_{\mathbb{R}^2}A(r) \nabla_r\widetilde{f}^N(x,r,s))\cdot\nabla_r\psi(r) \phi(x) dr dxdsdP+R_N,
  	\end{align}
  	where $R_N$ satisfies
  		\begin{align*}
  	\vert	R_N\vert  &\leq  \dfrac{C}{N} 	\int_A\int_0^T\int_{\mathbb{T}^2}\int_{\mathbb{R}^2}\vert P(r)\vert \vert\widetilde{f}^N(x,r,s)\vert\cdot\vert D_r^2\psi(r)\phi(x)\vert\vert\vert dr dxdsdP \\
  &\leq\dfrac{C}{N} 	\Vert \widetilde{f}^N\Vert_{L^2(\Omega\times[0, T];H))}\Vert D_r^2\psi\Vert_{\infty}\Vert \phi\Vert_{\infty} \to 0   \text{  as  }   N\to    +\infty.
  		 	\end{align*}
  		 	Now, by passing to the limit in \eqref{cv-3-*} and using  \autoref{lem-cv-1}, the conclusion follows.
  \end{proof} 	
Concerning  the stochastic  integral    part,   we  have    the following   result.
		\begin{lemma}\label{lem-stoch-vanishes}     Let $t\in   ]0,T],$  the following   convergence holds
		\begin{align*}
\lim_N\vert\E\int_0^t\sum_{k\in K}\int_{\mathbb{T}^2}\int_{\mathbb{R}^2}\widetilde{f}^N(x,r,s)\left(\sigma_k^N.\nabla_x\phi(x)\psi(r)  +(\nabla \sigma_k^Nr).\nabla_r \psi(r) \phi(x)\right) drdxdW^{N,k}(s)1_A\vert=0.
	 	\end{align*}
\end{lemma}
\begin{proof}
    Let   $A\in \mathcal{F}$, $\phi  \in C^\infty(\mathbb{T}^2)$ and  $\psi  \in C^\infty_c(\mathbb{R}^2)$,    by  using   It\^o isometry    we  get
    \begin{align*}&\E\left(\int_0^t\sum_{k\in K}\int_{\mathbb{T}^2}\int_{\mathbb{R}^2}\widetilde{f}^N(x,r,s)\left(\sigma_k^N.\nabla_x\phi(x)\psi(r)  +(\nabla \sigma_k^Nr).\nabla_r \psi(r) \phi(x)\right) drdxdW^{N,k}(s)1_A\right)^2\\&=
  \E\int_0^t\sum_{k\in K}\left[\int_{\mathbb{R}^2}\int_{\mathbb{T}^2}1_A\widetilde{f}^N(x,r,s)\left(\sigma_k^N.\nabla_x\phi(x)\psi(r)  +(\nabla \sigma_k^Nr).\nabla_r \psi(r) \phi(x)\right) dxdr\right]^2ds
  \\&=
  \E\int_0^t\sum_{k\in K}\left[\int_{\mathbb{R}^2}\int_{\mathbb{T}^2}(1_A\widetilde{f}^N(x,r,s)\nabla_x\phi(x)\psi(r))\cdot\sigma_k^N  +\left(1_A\widetilde{f}^N(x,r,s)\nabla_r \psi(r) \phi(x)\right)\cdot(\nabla \sigma_k^Nr) dxdr\right]^2ds
  \\&\leq
  2\E\int_0^t\sum_{k\in K}\left[\int_{\mathbb{R}^2}\int_{\mathbb{T}^2}(1_A\widetilde{f}^N(x,r,s)\nabla_x\phi(x)\psi(r))\cdot\sigma_k^N dxdr\right]^2ds\\&\quad +2\E\int_0^t\sum_{k\in K}\left[\int_{\mathbb{R}^2}\int_{\mathbb{T}^2}\left(1_A\widetilde{f}^N(x,r,s)\nabla_r \psi(r) \phi(x)\right)\cdot(\nabla \sigma_k^Nr) dxdr\right]^2ds:=2(I_1^N+I_2^N).
    	\end{align*}
   By   using   the definition  of  $\sigma_k^N$,   we  get
   \begin{align*}
       I_1^N&=\E\int_0^t\sum_{k\in K}\left[\int_{\mathbb{R}^2}\int_{\mathbb{T}^2}(1_A\widetilde{f}^N(x,r,s)\nabla_x\phi(x)\psi(r))\cdot\sigma_k^N dxdr\right]^2ds\\&=\E\int_0^t\sum_{k\in K_+}(\theta_{\vert    k\vert}^N)^2\left[\int_{\mathbb{R}^2}\int_{\mathbb{T}^2}(1_A\widetilde{f}^N(x,r,s)\nabla_x\phi(x)\psi(r))\cdot  \dfrac{k^\perp}{\vert   k\vert}\cos{k\cdot  x}dxdr\right]^2ds\\&+\E\int_0^t\sum_{k\in K_-}(\theta_{\vert    k\vert}^N)^2\left[\int_{\mathbb{R}^2}\int_{\mathbb{T}^2}(1_A\widetilde{f}^N(x,r,s)\nabla_x\phi(x)\psi(r))\cdot  \dfrac{k^\perp}{\vert   k\vert}\sin{k\cdot  x}dxdr\right]^2ds.
   \end{align*}
 By using   that    $(\dfrac{k^\perp}{\vert   k\vert}\cos{k\cdot  x},\dfrac{k^\perp}{\vert   k\vert}\sin{k\cdot  x})_{k\in  K}$  is an
(incomplete) orthonormal system in  $L^2(\T^2;\mathbb{R}^2)$,    we  obtain by  using   \eqref{coefficient}     
\begin{align*}
      \vert I_1^N\vert&\leq     \sup_{k\in  K}(\theta_{\vert    k\vert}^N)^2\E\int_0^t\int_{\mathbb{R}^2}\int_{\mathbb{T}^2}(\widetilde{f}^N(x,r,s)\nabla_x\phi(x)\psi(r))^2dxdrds\\
      &\leq    \Vert \nabla_x\phi\psi\Vert_\infty^2\sup_{k\in  K}(\theta_{\vert    k\vert}^N)^2  \E\int_0^T\int_{\mathbb{R}^2}\int_{\mathbb{T}^2}\widetilde{f}^N(x,r,s)^2dxdrds\\&\leq  \mathbf{K}^2 \Vert \nabla_x\phi\psi\Vert_\infty^2\sup_{k\in  K}(\theta_{\vert    k\vert}^N)^2\to   0   \text{  as  }   N\to    +\infty.
   \end{align*}
   Concerning   $I_2^N$,    we  have
   \begin{align*}
      I_2^N&= \E\int_0^t\sum_{k\in K}\left[\int_{\mathbb{R}^2}\int_{\mathbb{T}^2}\left(1_A\widetilde{f}^N(x,r,s)\nabla_r \psi(r) \phi(x)\right)\cdot(\nabla \sigma_k^Nr) dxdr\right]^2ds\\&=\E\int_0^t\sum_{k\in K_+}(\theta_{\vert    k\vert}^N)^2\vert   k\vert^2\left[\int_{\mathbb{R}^2}\int_{\mathbb{T}^2}\dfrac{k}{\vert   k\vert}\cdot   r\left(1_A\widetilde{f}^N(x,r,s)\nabla_r \psi(r) \phi(x)\right)\cdot  \dfrac{k^\perp}{\vert   k\vert}\sin{k\cdot  x}dxdr\right]^2ds\\&+\E\int_0^t\sum_{k\in K_-}(\theta_{\vert    k\vert}^N)^2\vert   k\vert^2\left[\int_{\mathbb{R}^2}\int_{\mathbb{T}^2}\dfrac{k}{\vert   k\vert}\cdot   r\left(1_A\widetilde{f}^N(x,r,s)\nabla_r \psi(r) \phi(x)\right)\cdot  \dfrac{k^\perp}{\vert   k\vert}\cos{k\cdot  x}dxdr\right]^2ds.
   \end{align*}
   Thus,    by  using   the same    argument    as  above   we  obtain
   \begin{align*}
      \vert I_1^N\vert&\leq     \sup_{k\in  K}[(\theta_{\vert    k\vert}^N)^2\vert   k\vert^2]\E\int_0^T\int_{\mathbb{R}^2}\int_{\mathbb{T}^2}   \vert    r\vert^2\left(\widetilde{f}^N(x,r,s)\nabla_r \psi(r) \phi(x)\right)^2dxdrds\\
      &\leq \sup_{k\in  K}[(\theta_{\vert    k\vert}^N)^2\vert   k\vert^2]\Vert \nabla_r \psi \phi\Vert_\infty^2\Vert  \widetilde{f}^N\Vert_{L^2(\Omega\times[0,T];H)}^2\to   0   \text{  as  }   N\to    +\infty.
\end{align*}
\end{proof}

\begin{lemma}\label{lem-cv-cont}   For any $t\in]0,T[,$ the    following    convergence holds   (up to  a   subsequence)
    \begin{align*}
   \lim_N     \int_A\int_{\mathbb{T}^2}\int_{\mathbb{R}^2}\widetilde{f}^N(x,r,t)\phi(x)\psi(r) drdxdP=\int_A\int_{\mathbb{T}^2}\int_{\mathbb{R}^2}\overline{f}(x,r,t)\phi(x)\psi(r) drdxdP.
    \end{align*}
\end{lemma}
\begin{proof}
    From    \eqref{cv-fN-time},   
    $\widetilde{f}^N\rightarrow \overline{f} \quad \text{in } C([0,T];U^\prime)$ $P$-a.s.
    Since   $(\widetilde{f}^N)_N$    is  bounded in  $L^2_{w-*}(\Omega;L^\infty([0,T];H))$,  Vitali's convergence theorem ensures the convergence of  $\widetilde{f}^N$    to  $\overline{f}$ in  $L^1(\Omega;C([0,T];U^\prime))$.
\end{proof}
\subsubsection{Proof    of \autoref{main-thm-2}}\label{section-scaling-PF}

   Let  $\phi  \in C^\infty(\mathbb{T}^2)$ and  $\psi  \in C^\infty_c(\mathbb{R}^2)$,	
Let $t\in   ]0,T]$  and $A\in \mathcal{F}$.   From	\autoref{Def-sol},  point   (3)  and  by multiplying   by $I_A$ and integrating over $\Omega\times[0,t]$,  we derive
 \begin{align*}&\int_A\int_{\mathbb{T}^2}\int_{\mathbb{R}^2}\widetilde{f}^N(t)\phi\psi drdxdP-\int_A\int_{\mathbb{T}^2}\int_{\mathbb{R}^2}f_0\phi\psi drdxdP\\&-\int_A\int_0^t\sum_{k\in K}\int_{\mathbb{T}^2}\int_{\mathbb{R}^2}\widetilde{f}^N(s)\left(\sigma_k^N.\nabla_x\phi\psi  +(\nabla \sigma_k^Nr).\nabla_r \psi \phi\right) drdxdW^{N,k}(s)dP\\&= \int_A\int_0^t\int_{\mathbb{T}^2}\int_{\mathbb{R}^2}\widetilde{f}^N(s)\left(u_L(s)\cdot \nabla_x \phi \psi+ (\nabla u_L(s)r-\dfrac{1}{\beta}r) \cdot\nabla_r\psi \phi \right) dr dxdsdP\\	&-	\sigma^2\int_A\int_0^t\int_{\mathbb{T}^2}\int_{\mathbb{R}^2}\nabla_r\widetilde{f}^N(s) \cdot\nabla_r\psi \phi  dr dxdsdP-\alpha_N\int_A\int_0^t\int_{\mathbb{T}^2}\int_{\mathbb{R}^2}\nabla_x\widetilde{f}^N(s) \cdot\nabla_x\phi\psi dr dxdsdP\\
&- \dfrac{1}{2}\int_A\int_0^t\int_{\mathbb{T}^2}\int_{\mathbb{R}^2}(\sum_{k\in K}\left((\nabla \sigma_k^Nr) \otimes (\nabla \sigma_k^Nr)\right) \nabla_r\widetilde{f}^N(s))\cdot\nabla_r\psi \phi dr dxdsdP\\
&:=J_1^N+J_2^N+J_3^N+J_4^N.\end{align*}
We  pass    to  the limit  as   $N\to+\infty$ in  the RHS of  the last    equation.   By using  \autoref{lem-cv-1}, one has
\begin{align*}&\lim_N\int_A\int_0^T\int_{\mathbb{T}^2}\int_{\mathbb{R}^2}\widetilde{f}^N(s)\left(u_L(s)\cdot \nabla_x \phi \psi+ (\nabla u_L(s)r-\dfrac{1}{\beta}r) \cdot\nabla_r\psi \phi \right) dr dxdsdP\\&=\int_A\int_0^T\int_{\mathbb{T}^2}\int_{\mathbb{R}^2}\overline{f}(s)\left(u_L(s)\cdot \nabla_x \phi \psi+ (\nabla u_L(s)r-\dfrac{1}{\beta}r) \cdot\nabla_r\psi \phi \right) dr dxdsdP.
\end{align*}
Additionally,   we  have
\begin{align*}&\lim_N	\int_A\hspace{-0.08cm}\int_0^T\hspace{-0.08cm}\int_{\mathbb{T}^2}\hspace{-0.08cm}\int_{\mathbb{R}^2}\nabla_r\widetilde{f}^N(s) \cdot\nabla_r\psi \phi  dr dxdsdP= \int_A\hspace{-0.08cm}\int_0^T\hspace{-0.08cm}\int_{\mathbb{T}^2}\hspace{-0.08cm}\int_{\mathbb{R}^2}\nabla_r\overline{f}(s) \cdot\nabla_r\psi \phi  dr dxdsdP.
\end{align*}
 From     \autoref{lem-cv-2}  and    \autoref{lem-cv-3}, we  get
 \begin{align*}
     \lim_{N}  ( J_3^N+J_4^N)=-\dfrac{1}{2}\int_A\int_0^T\int_{\mathbb{T}^2}\int_{\mathbb{R}^2}A(r) \nabla_r\overline{f}(s))\cdot\nabla_r\psi \phi dr dxdsdP.
 \end{align*}
Finally,   we use \autoref{lem-stoch-vanishes}   and   \autoref{lem-cv-cont}      to  pass    to  the limit   in  LHS to complete    the proof.\\

\subsection{Uniqueness of the limit problem \eqref{equ-final}} 
Let us conclude this section by showing that the limit equation    \eqref{equ-final} has at most one solution.
\begin{lemma}\label{lemma-uniqunes-limit}
   The solution    $\overline{f}$  to  \eqref{equ-final}   is  unique.
\end{lemma}
\begin{proof}
Let $\overline{f}_1$   and $\overline{f}_2$ be two solutions   to  \eqref{equ-final} and denote by  $\overline{f}$ be their difference.   Then
   $P$-a.s for any $t\in ]0, T]$   and $\phi\in    Y$, we  have
		\begin{align}&(    \overline{f}(t),\phi)-\int_0^t\langle \overline{f}(s),u_L(s)\cdot \nabla_x \phi+ (\nabla u_L(s)r-\dfrac{1}{\beta}r) \cdot\nabla_r\phi\rangle   ds\label{eqn.form-uniq}\\	=&-	\int_0^t\sigma^2\langle \nabla_r\overline{f}(s), \nabla_r\phi\rangle   ds- \dfrac{1}{2}\int_0^t\langle A(r) \nabla_rf^N(s)),\nabla_r\phi\rangle    ds.\notag\end{align}
  Since   the above   equation    is  in  weak    form,   we  need    first   to  consider    an  appropriate  regularization  to   get an  equation    for $\Vert  \overline{f}(t)  \Vert^2$   for any $t\in[0,T].$ We use analogous   argument    to  \textit{"Step   $1$"}  in  the proof   of  \autoref{lem-uniq-Vn}     to  obtain
   \begin{align}&\dfrac{1}{2} \Vert [\overline{f}(t)]_{\delta}\Vert^2+\int_0^t\langle [u_L(s)\cdot \nabla_x\overline{f}(s)]_{\delta}+[\Div_r(\nabla u_L(s)r-\dfrac{1}{\beta}r)\overline{f}(s)]_{\delta}, [\overline{f}(s)]_{\delta}\rangle   ds\label{uniq-reg-delta--}\\	=&\int_0^t\langle \sigma^2[\Delta_r\overline{f}(s)]_{\delta}, [\overline{f}(s)]_{\delta}\rangle   ds+ \dfrac{1}{2}\int_0^t\langle [\Div_rA(r) \nabla_r\overline{f}(s))]_{\delta},[\overline{f}(s)]_{\delta}\rangle    ds\notag.\end{align}
    Notice    that    $[\Div_rA(r) \nabla_r\overline{f}(s))]_{\delta},[\overline{f}(s)]_{\delta}\rangle    =\langle \Div_rA(r) \nabla_r[\overline{f}(s))]_{\delta},[\overline{f}(s)]_{\delta}\rangle    $ hence
\begin{align*}
    \dfrac{1}{2}\int_0^t\langle [\Div_rA(r) \nabla_r\overline{f}(s))]_{\delta},[\overline{f}(s)]_{\delta}\rangle    ds&=-
\dfrac{3k_T}{2}\int_0^t\Vert \vert    r\vert\nabla_r  [\overline{f}(s)]_\delta\Vert^2ds+k_T\int_0^t  \Vert  r\cdot\nabla_r [\overline{f}(s)]_\delta\Vert^2ds\\
&\leq   -\dfrac{k_T}{2}\int_0^t\Vert [\vert    r\vert\nabla_r  \overline{f}(s)]_\delta\Vert^2ds\to  -\dfrac{k_T}{2}\int_0^t\Vert \vert    r\vert\nabla_r  \overline{f}(s)\Vert^2ds  \text{  as  }   \delta  \to 0.
\end{align*}
Next,   arguments similar to that used in \textit{"Step   $2$"}  of  the proof   of  \autoref{lem-uniq-Vn}   allow  to  pass
 to  the limit   as  $\delta\to0$    in  \eqref{uniq-reg-delta--} and   we  get
 \begin{align}\label{limit-delta-inequality--}&\dfrac{1}{2} \Vert \overline{f}(t)\Vert^2+\sigma^2\int_0^t\Vert \nabla_r\overline{f}(s)\Vert^2   ds+\dfrac{k_T}{2}\int_0^t\Vert \vert    r\vert\nabla_r  \overline{f}(s)\Vert^2ds	\leq  \dfrac{1}{\beta}\int_0^t\Vert  \overline{f}(s)\Vert^2ds.    \end{align}
 The    last    inequality  \eqref{limit-delta-inequality--}  and Grönwall    lemma   completes    the proof of    \autoref{lemma-uniqunes-limit}.
\end{proof}  
\begin{remark}
Another way to  prove   that    uniqueness  to \eqref{equ-final} holds  is  to  notice  that 
\eqref{FP_SDE-limit}   
has \eqref{equ-final} as FP equation associated. We have uniqueness in law of weak solutions of the SDE \eqref{FP_SDE-limit} due to the properties of $\Sigma(r)$. The latter ensures uniqueness of solutions of \eqref{equ-final} due to \cite[Thm. 2.5]{trevisan2016well}.

\end{remark}

\subsection{About the 3D case}\label{remark-3D-case}
 It is worth drawing the reader's attention to the fact that with cosmetic changes to what was made above, it is possible to prove a similar result in a three-dimensional setting.  However, we preferred to present the result in 2D because we could present the stochastic turbulent velocity explicitly, see \autoref{assumption_noise-2D}. The main difference between 2D and 3D will be in the form of the stochastic turbulent velocity. In 3D, we can formulate the result using Fourier decomposition.  Let us give an example of noise where we get a similar result in 3D.\\

We  introduce the partition $\mathbb{Z}^3_0=\Gamma_{3,+}\cup \Gamma_{3,-}$\footnote{$\mathbb{Z}^3_0=\mathbb{Z}^3-\{(0,0,0)\}$} such that $\Gamma_{3,+}=-\Gamma_{3,-} $
and  we consider a family  of  real valued independent Brownian motions $(B_t^{k,j})_t^{k\in    \mathbb{Z}^3_0}, j\in \{1,2\}$   defined  on the complete   filtered probability space    $(\Omega,\mathcal{F},(\mathcal{F}_t)_t,P)$, that is 
$ \mathbb{E}(B_t^{k,j}B_s^{l,m})= \min(t,s)\delta_{k,l}\delta_{j,m}.$
Then,   we introduce a sequence of complex-valued Brownian motions adapted to $(\mathcal{F}_t)_t$  defined as follows
\begin{align*}
    W^{k,j}_t&=\begin{cases}
        B^{k,j}_t+iB^{-k,j}_t & \textit{if } k\in \Gamma_{3,+} \\
   B^{-k,j}_t-iB^{k,j}_t & \textit{if } k\in \Gamma_{3,-}.
    \end{cases}
    \end{align*}
Note that
$W^{-k,j}_t=\overline{W^{k,j}_t}$\footnote{ For a complex number $z$, $\overline{z}$ denotes its  complex conjugate.} and satisfies
\begin{align*}
    \E(W^{k,j}_1,\overline{W^{l,m}_1})&=\begin{cases}
        2 & \textit{if } k=l, \ m=j\\
        0 &  \textit{otherwise} ;
    \end{cases}
    \quad \left[W^{k,j}_\cdot, W^{l,m}_\cdot\right]_t=\begin{cases}
        2t & \textit{if } k=-l, \ m=j\\
        0 &  \textit{otherwise }.    
        \end{cases}
\end{align*}
  Let $N\in \mathbb{N}^*$,    define
  $
   \theta^N_{k,j} =\dfrac{a}{\lvert k\rvert^{5/2}}\mathbf{1}_{\{N\leq \lvert k\rvert\leq 2N\} },\  
$ for   a positive constant $a$.
Then, for each $ k\in \mathbb{Z}^{3}_0,\ j\in\{1,2\}$ we denote by $\sigma_{k,j}^N(x)=\theta_{k,j}^N  a_{k,j}e^{ik\cdot x}$, where $\{\frac{k}{\lvert k\rvert}, a_{k,1}, a_{k,2}\}$ is an orthonormal system of $\R^3$ for $k\in \Gamma_{3,+}$ and $a_{k,j}=a_{-k,j}$ if $k\in \Gamma_{3,-}$. Set
\begin{align}\label{noise-3D}
  \mathbf{W}^N(t,x)  =\sum_{{\substack{ k\in \mathbb{Z}^3_0, j\in \{1,2 \} }}}\theta_{k,j}^N  a_{k,j}e^{ik\cdot x}W^{k,j}(t).
\end{align}
  Thanks to the assumptions on  $\theta_{k,j}^N,  a_{k,j}$ and $W^{k,j}$, we can prove that  $\mathbf{W}^N$ is a real, mean-zero and divergence free 
random distribution.  Moreover, $\mathbf{W}^N(t,\cdot)\in L^2(\Omega,L^2(\mathbb{T}^3,\mathbb{R}^3)),$  see \cite[Sect. 2.2.]{Galeati}. On the other hand,  it is not diffucult to check that  the covariance operator associated with $\mathbf{W}^N$  is space-homogeneous  and has a mirror symmetry property. Now, by considering
$u^N=u_L+\circ \partial_t \mathbf{W}^N$ instead of   \eqref{turb-velocity-2D-eq} and then replace in \eqref{Stra-FP} we  obtain a  stochastic FP  in 3D  similar to  \eqref{Ito-FP}. Then, by cosmetic changes to what has been done above and  noticing that 
 \begin{align*}
 \sum_{\substack{k\in \mathbb{Z}^3_0, j\in \{1,2 \}\\ N\leq \lvert k\rvert \leq 2N}}\left((\nabla \sigma_{k,j}^Nr) \otimes \overline{(\nabla \sigma_{k,j}^Nr)}\right) &=  \sum_{\substack{k\in \mathbb{Z}^3_0\\ N\leq \lvert k\rvert \leq 2N}}\frac{(r\cdot k)^2}{\lvert k\rvert^5}\left(I-\frac{k\otimes k}{\lvert k\rvert^2}\right)= \mathcal{M}(r)+O(\dfrac{1}{N})P(r),	
\end{align*}
where $\mathcal{M}(r)= \frac{8\pi \log 2}{15}\left(2\lvert r\rvert^2 I-r\otimes r\right).$
We can prove that the initial stochastic FP \eqref{Ito-FP} in 3D   converges to  a limit PDE as \eqref{Limit-FP-eqn} where the matrix $A(r)$ is replaced by $\mathcal{M}(r).$

\begin{acknowledgements}  The research of F.F. and Y.T. is
funded by the European Union (ERC, NoisyFluid, No. 101053472). Views and
opinions expressed are however those of the authors only and do not necessarily
reflect those of the European Union or the European Research Council. Neither
the European Union nor the granting authority can be held responsible for them.
\end{acknowledgements}


\bibliography{scaling-limit-polymers_arXiv}{}
\bibliographystyle{plain}

\end{document}